\documentclass[a4paper,fleqn]{cas-sc}

\usepackage[numbers]{natbib}

\usepackage[utf8]{inputenc} 
\usepackage{accents}				
\usepackage{todonotes}
\usepackage{subcaption}
\usepackage{stmaryrd} 
\usepackage{amsthm} 
\usepackage[makeroom]{cancel} 
\usepackage{bookmark} 
\usepackage{empheq}
\usepackage{ulem}
\usepackage{tcolorbox}  
\usepackage{pdflscape} 
\usepackage{bm}

\newcommand{\change}[1]{#1}
\newcommand{\q}[2]{#2}

\newcommand{\red}[1]{{\color{red}#1}}
\newcommand{\blue}[1]{{\color{blue}#1}}
\newcommand{\norm}[1]{\left\lVert#1\right\rVert}
\newcommand{\mat}[1]{\underline{\mathbf{#1}}}
\newcommand{\matcal}[1]{\underline{\bm{\mathcal{#1}}}}
\newcommand\acclrvec[1]{\accentset{\,\leftrightarrow}{#1}}	
\newcommand{\blocktensor}[1]{\acclrvec{{\mathbf #1}}}	
	
\newcommand{\Nabla} {\vec{\nabla}}
\newcommand{\numfluxb}[1]{\hat{\mathbf{#1}} }
\newcommand{\numflux}[1]{\hat{{#1}} }
\newcommand\threeMatrix[1]{\underline{ #1}}				
\newcommand{\ec}{{\mathrm{EC}}}			
	
\def\d{\mathrm{d}}

\newcommand{\bigpartialderiv}[2]{ \frac{\partial {#1}}{\partial {#2} } }

\newcommand\stateG[1]{\boldsymbol #1}			

\newcommand{\noncon}{\stateG{\Upsilon}}	

\newcommand{\entVar}{{\mathbf{w}}}

\newcommand\state[1]{\mathbf{#1}}

\newcommand{\supEuler}{{\mathrm{Euler}}}
\newcommand{\supMHD}{{\mathrm{mMHD}}}
\newcommand{\supGLM}{{\mathrm{GLM}}}
\newcommand{\Powell}{{\mathrm{GP}}}

\newcommand{\Jan}{\stateG{\Phi}}

\newcommand{\numnonconsD}[1]{ #1^{\diamond} }
\newcommand{\numnonconsS}[1]{ #1^{\star} }

\newcommand{\avg}[1]{\left\{\hspace*{-3pt}\left\{#1\right\}\hspace*{-3pt}\right\}}
\newcommand{\jump}[1]{\ensuremath{\left\llbracket #1 \right\rrbracket}}

\newcommand{\jumpR}[1]{\langle\hspace*{-2pt}\langle#1\rangle\hspace*{-2pt}\rangle}

\newcommand{\Emagbar}{\overline{E}_{\mathrm{mag}}^2}
\newcommand{\betaavg}{\overline{\beta^2}}
\newcommand{\rholn}{\rho_k^{\ln}}
\newcommand{\betaln}{\beta_k^{\ln}}
\newcommand{\pavg}{\overline{p}_k}
\newcommand{\Eline}{\overline{E}_k}
\newcommand{\uavg}{\overline{\|\vec{v}_k\|^2}}

\newcommand{\hydroEner}{\mathcal{E}}

\newtheorem{lemma}{Lemma}

\newtheorem{remark}{Remark}

\begin{document}

\let\WriteBookmarks\relax
\def\floatpagepagefraction{1}
\def\textpagefraction{.001}

\shorttitle{Entropy-Stable DG for Ideal Multi-Ion MHD}
\shortauthors{Rueda-Ramírez~et~al.}

\title [mode = title]{An Entropy-Stable Discontinuous Galerkin Discretization of the Ideal Multi-Ion Magnetohydrodynamics System}

\author[1,2]{Andrés M. Rueda-Ramírez}[
                         orcid=0000-0001-6557-9162]
 \cormark[1]
 \ead{aruedara@uni-koeln.de}

 \credit{Conceptualization, Methodology, Software, Validation, Formal analysis, Data Curation, Writing - Original Draft, Visualization}

\author[1]{Aleksey Sikstel}
 [orcid=0000-0001-5921-7050]
 \ead{a.sikstel@uni-koeln.de}
 \credit{Formal analysis, Validation, Writing - Original Draft}

\author[1,3]{Gregor J. Gassner}
 [orcid=0000-0002-1752-1158]
 \ead{ggassner@uni-koeln.de}
 \credit{Conceptualization, Methodology, Writing - Original Draft}

\address[1]{Department of Mathematics and Computer Science, University of Cologne, Weyertal 86-90, 50931 Cologne, Germany}

\address[2]{Applied and Computational Mathematics, RWTH Aachen University, Schinkelstraße 2, 52062 Aachen, Germany}

\address[3]{Center for Data and Simulation Science, University of Cologne, 50931 Cologne, Germany}

\cortext[cor1]{Corresponding author}

\begin{abstract}
In this paper, we present an entropy-stable (ES) discretization using a nodal discontinuous Galerkin (DG) method for the ideal multi-ion magneto-hydrodynamics (MHD) equations.

We start by performing a continuous entropy analysis of the ideal multi-ion MHD system, described by, e.g., [Toth (2010) Multi-Ion Magnetohydrodynamics] \cite{Toth2010}, which describes the motion of multi-ion plasmas with independent momentum and energy equations for each ion species.
Following the continuous entropy analysis, we propose an algebraic manipulation to the multi-ion MHD system, such that entropy consistency can be transferred from the continuous analysis to its discrete approximation.
Moreover, we augment the system of equations with a generalized Lagrange multiplier (GLM) technique to have an additional cleaning mechanism of the magnetic field divergence error.

We first derive robust entropy-conservative (EC) fluxes for the alternative formulation of the multi-ion GLM-MHD system that satisfy a Tadmor-type condition and are consistent with existing EC fluxes for single-fluid GLM-MHD equations.
Using these numerical two-point fluxes, we construct high-order EC and ES DG discretizations of the ideal multi-ion MHD system using collocated Legendre--Gauss--Lobatto summation-by-parts (SBP) operators.
The resulting nodal DG schemes satisfy the second-law of thermodynamics at the semi-discrete level, while maintaining high-order convergence and local node-wise conservation properties.

We demonstrate the high-order convergence, and the EC and ES properties of our scheme with numerical validation experiments. Moreover, we demonstrate the importance of the GLM divergence technique and the ES discretization to improve the robustness properties of a DG discretization of the multi-ion MHD system by solving a challenging magnetized Kelvin-Helmholtz instability problem that exhibits MHD turbulence.
\end{abstract}




\maketitle

\section*{Connection to Prof. Godunov's work} On the occasion of the Journal of Computational Physics' special issue in honor of Prof. Sergei Godunov's seminal work, this paper draws a significant parallel to his pioneering numerical analysis of non-conservative hyperbolic partial differential equations with involutions.
In particular, \citet{godunov1972symmetric} showed the necessity of incorporating of the now-called Godunov-Powell non-conservative term \cite{powell1994approximate,powell1999solution}, which is proportional to the divergence of the magnetic field, along with a strictly convex entropy function, for the symmetrization and entropy consistency of the single-fluid MHD system.
In this paper, we build upon Godunov's work and propose an alternative formulation of the multi-ion MHD system that allows the derivation of a strictly convex entropy function and leads to the generalization of the Godunov-Powell non-conservative term to multiple ion species.
This formulation allows us to derive entropy-consistent finite volume and discontinuous Galerkin discretizations of the multi-ion MHD equations.

\section{Introduction} \label{sec:intro}

The modeling and numerical simulation of complex plasma dynamics is pivotal in advancing our understanding of various physical phenomena that occur in astrophysics \cite{Stone2008, Mignone2012, mignone2010high}, space physics \cite{glocer2009multifluid, rubin2015self, rubin2014comet}, nuclear fusion reactors \cite{ghosh2019multispecies, rambo1995comparison}, among others. 
One particularly comprehensive framework for describing the intricate behavior of multi-species plasmas is encapsulated in the multi-ion magnetohydrodynamics (MHD) equations, as described by, e.g., \citet{Toth2010}.
These equations account for the individual dynamics of multiple ion species, incorporating distinct mass, momentum and energy equations for each species within a unified MHD framework.

Two physical constraints to the multi-ion MHD system that are relevant to this paper and are not explicitly built into the equations are:
\begin{enumerate}
    \item The involution constraint that dictates the divergence free condition on the magnetic field, $\Nabla \cdot \vec{B} = 0$.
    \item The second law of thermodynamics, which states that the total thermodynamic entropy of a closed system does not decrease in time.
\end{enumerate}

Numerous techniques exist for addressing the divergence-free constraint in numerical discretizations of the ideal Magnetohydrodynamics (MHD) system, such as those found in \cite{Munz2000,Dedner2002,balsara2021globally,li2005locally,balsara2004second,fey2003constrained,balsara2015divergence,balsara2009efficient}. 
In this work, we adopt the divergence cleaning method proposed by \citet{Munz2000} and \citet{Dedner2002} and extend them to the multi-ion MHD system. 
This approach expands the multi-ion MHD system with a generalized Lagrange multiplier (GLM), i.e. an additional auxiliary quantity, which is advected and damped to minimize divergence errors. 
Given that the GLM technique reduces, but does not entirely eliminate, divergence errors, it becomes imperative for the extended system to incorporate the Godunov-Powell non-conservative term \cite{godunov1972symmetric,powell1994approximate,powell1999solution}. 
The emergence of the Godunov-Powell term from Maxwell's equations, specifically when $\Nabla \cdot \vec{B} \ne 0$, is pivotal for symmetrizing the system of Partial Differential Equations (PDEs). This symmetrization is crucial to satisfying the second law of thermodynamics in the presence of non-vanishing divergence errors \cite{Derigs2018,Chandrashekar2016}.

The discontinuous Galerkin (DG) method is a popular method to discretize advection-dominated equations \cite{Wang2013High, Cockburn2000}, such as the multi-ion MHD equations, as it provides arbitrary high-order accuracy, a very local data dependency foot print that favors parallelization, and can be readily used in 3D curvilinear and unstructured meshes \cite{Hindenlang2012, Krais2019, ferrer2023high}.
However, the DG method might suffer instability issues in under resolved strongly nonlinear simulations due to aliasing-driven instabilities incurred by the insufficient integration of the highly nonlinear fluxes of the equation, and in the presence of very steep gradients or discontinuities.
In this paper, we will focus on enhancing the high-order DG scheme to deal with aliasing-driven instabilities in multi-ion MHD problems. 

An effective approach for mitigating aliasing-driven instabilities, inspired by the work of \citet{Fisher2013}, \citet{Gassner2013}, and \citet{Carpenter2014}, leverages the summation-by-parts (SBP) property inherent in certain Discontinuous Galerkin (DG) schemes. This property serves as a discrete analogue to integration-by-parts, as discussed in, for instance, \cite{svard2014,fernandez2014review}.
The SBP operators' capacity to mimic integration-by-parts at the discrete level becomes a valuable tool for constructing entropy-stable discretizations, as highlighted by \cite{Fisher2013,Fisher2013a,Carpenter2014}. 
Additionally, it facilitates the development of other nonlinearly stable discretizations based on split forms of the governing equations, as demonstrated by \cite{Gassner2013,Gassner2016}.
Specifically, entropy stability represents a form of nonlinear stability intricately linked to the adherence of the numerical scheme to the second law of thermodynamics. 
This strategic utilization of SBP operators has proven instrumental in enhancing the robustness of DG methods across various applications, including but not limited to the shallow water equations, e.g., \cite{wintermeyer2017entropy}, the incompressible Navier-Stokes equations, e.g., \cite{manzanero2020entropy}, the compressible Navier-Stokes equations, e.g., \cite{Gassner2018}, multi-phase fluid equations, e.g.,  \cite{Renac2019,manzanero2020entropy2,coquel2021entropy}, and the Generalized Lagrange Multiplier Magnetohydrodynamics (GLM-MHD) system \cite{Bohm2018}.
In this work, we use a nodal (collocated) variant of the DG method known as the discontinuous Galerkin spectral element method (DGSEM) \cite{black1999conservative,kopriva2009implementing} on Legendre--Gauss--Lobatto (LGL) points, which satisfies the SBP property \cite{Gassner2013}.

The construction and formulation of entropy-stable SBP discretizations initiates with a comprehensive exploration of the entropic properties of the continuous problem, i.e., the PDE of interest. A continuous entropy analysis of the governing equation is needed to discern the entropy pair: the mathematical entropy and the induced entropy flux. The mathematical entropy, a strictly convex function, intricately depends on all system state variables. 
In scenarios where the PDE captures physical phenomena, it exhibits a linear dependence on thermodynamic entropy. In the absence of shocks, viscous, and resistive effects, manipulation of the governing equations allows the derivation of a scalar conservation law for the mathematical entropy evolved by the divergence of the corresponding entropy flux. However, in the presence of shocks, viscous, or resistive effects, the mathematical entropy no longer adheres to a conservation law but rather to an inequality - (mathematical) entropy is decaying in time. This inequality, reflective of the second law of thermodynamics, underscores the behavior of entropy under varying physical conditions.

The second step to construct high-order entropy-stable SBP discretizations for nonlinear PDEs is to find finite-volume (FV) EC two-point fluxes and non-conservative terms for the governing equations.
When these fluxes are used in a two-point FV discretization of the governing equations, the semi-discrete scheme mimics the scalar entropy conservation law of the system that is expected for smooth and inviscid solutions.
High-order EC SBP schemes can be readily obtained by plugging in the EC fluxes into a high-order split-form framework, see, e.g., \cite{Gassner2016, rueda2023entropy, rueda2024flux}.
These discretizations, however, are virtually dissipation free and hence offer no dissipative mechanism to get rid of under resolved modes, which is why some form of numerical dissipation via the numerical surface fluxes is introduced in a controlled way, such that the scheme is guaranteed entropy dissipative.

The paper and its main contributions are organized as follows.
In Section \ref{sec:equations}, we review the ideal multi-ion MHD system, then analyze its entropy properties and propose an algebraic manipulation of the system that guarantees that entropy consistency can be transferred from the continuous to the discrete level, and augment the system with the GLM divergence cleaning technique.
In Section \ref{sec:discretization}, we derive FV EC fluxes and non-conservative terms for the ideal multi-ion GLM-MHD system and construct a high-order ES DGSEM discretization using the EC fluxes for the volume integral and Rusanov fluxes for the surface integrals.
Since the multi-ion GLM-MHD system employed in this study represents a generalization of the single-fluid GLM-MHD system, we place particular emphasis on deriving EC and ES fluxes for the multi-ion GLM-MHD system that are consistent with the single-fluid fluxes of \citet{Derigs2018}. This consistency guarantees that in the special case of a single fluid, we recover the known ES discretizations exactly. Finally, in Section \ref{sec:results}, we validate the high-order convergence and entropy conservation/stability properties of the scheme with numerical tests. 
Moreover, we demonstrate numerically that the ES discretization and the GLM technique enhance the robustness of a high-order DG scheme by simulating a challenging multi-ion MHD Kelvin-Helmholtz instability problem.

\section{Governing Equations} \label{sec:equations}

In this section, we analyze the governing equations of multi-species plasmas.
We start by presenting the multi-species plasma model available in literature and analyzing its thermodynamic properties in Section \ref{sec:tothsystem}.
Next, in Section \ref{sec:modmultiion}, we propose an algebraic manipulation of the multi-ion MHD system, which makes it  possible to perform a continuous entropy analysis and show entropy consistency.
In Section \ref{sec:glm}, we augment the multi-ion MHD system with a GLM divergence cleaning technique and analyze the entropy properties of the augmented system.
In the last part of this section, we write a ``quasi-conservative'' form of the multi-ion MHD system and the one-dimensional multi-ion MHD equation, as these forms will be usefull to obtain our numerical schemes.
We place particular emphasis on obtaining governing equations that are consistent with the single-fluid GLM-MHD equations used by the authors in previous studies.

The systems of PDEs that we analyze in this section can be written as a sum of conservative and non-conservative terms,
\begin{equation}
    \frac{\partial \state{u}}{\partial t}
    + \Nabla \cdot \blocktensor{f} (\state{u})
    + \state{g} (\state{u})
    + \noncon (\state{u},\Nabla \state{u})
    =
    \state{s},
\end{equation}
where $\state{u}$ is the state vector, $\blocktensor{f}$ is a (conservative) flux function, $\state{g}(\state{u})$ is a source term that accounts for electromagnetic interactions of the different ion species, $\noncon$ is a non-conservative term that depends on the state quantities and their gradients, and $\state{s}$ is a source term that can account for several physical phenomena, such as gravitational forces, chemical reactions, ion-ion, ion-neutral and ion-electron collisions, among others.
For the entropy analysis presented in this paper we focus on the PDE parts, in particular the advective parts, and assume that the external source terms are zero, i.e., $\state{s} = \state{0}$.

\subsection*{Notation} \label{sec:notation}
As in previous works, we adopt the notation of \cite{Bohm2018,Gassner2018,RuedaRamirez2020,Rueda-Ramirez2020, rueda2023entropy} to deal with vectors of different nature. 
Spatial vectors are noted with an arrow (e.g. $\vec{x}=(x,y,z) \in \mathbb{R}^3$), state vectors are noted in bold (e.g. $\state{u}=(\rho, \rho \vec{v}, \rho E, \vec{B}, \psi)^T$), and block vectors, which contain a state vector in every spatial direction, are noted as
\begin{equation}
\blocktensor{f} =
\begin{pmatrix}
\mathbf{f}_1 \\ 
\mathbf{f}_2 \\
\mathbf{f}_3
\end{pmatrix}.
\end{equation}

The gradient of a state vector is hence a block vector,
\begin{equation}
\Nabla \mathbf{q} =
\begin{pmatrix}
\partial_x \mathbf{q} \\
\partial_y \mathbf{q} \\
\partial_z \mathbf{q} 
\end{pmatrix},
\end{equation}
and the gradient of a spatial vector is defined as the transpose of the outer product of this vector with the nabla operator,
\begin{equation}
\Nabla \vec{v} := 
\left( \Nabla \otimes \vec{v} \right)^T = 
\left( \Nabla \vec{v}^T \right)^T = 
\mat{L} =
\begin{pmatrix}
\bigpartialderiv{v_1}{x} & \bigpartialderiv{v_1}{y} & \bigpartialderiv{v_1}{z} \\
\bigpartialderiv{v_2}{x} & \bigpartialderiv{v_2}{y} & \bigpartialderiv{v_2}{z} \\
\bigpartialderiv{v_3}{x} & \bigpartialderiv{v_3}{y} & \bigpartialderiv{v_3}{z}
\end{pmatrix},
\end{equation}
where we remark that we note general matrices with an underline.

\q{c11_r2}{
We define the notation for the jump operator, arithmetic and logarithmic means between a left and right state, $a_L$ and $a_R$, as
\begin{equation}
\jump{a}_{(L,R)} := a_R-a_L, 
~~~~~~~~~ 
\avg{a}_{(L,R)} := \frac{1}{2}(a_L+a_R), 
~~~~~~~~ 
a^{\ln}_{(L,R)} := \jump{a}_{(L,R)}/\jump{\ln(a)}_{(L,R)},
\label{means}
\end{equation}
and recall that \change{both} averages\change{, $\avg{a}_{(L,R)}$ and $a^{\ln}_{(L,R)}$,} are symmetric with respect to the arguments $L$ and $R$. A numerically stable \change{and computationally efficient} procedure to evaluate the logarithmic mean \change{and its inverse} is given in \change{\cite{ranocha2023efficient} based on ideas first developed by} \citet{Ismail2009}.
}
 
\subsection{The Ideal Multi-Ion MHD System and its Thermodynamic Properties} \label{sec:tothsystem}

There are different variants of the multi-ion MHD equations of varying complexity.
In this work, we study the ideal multi-ion MHD system described by \citet{Toth2010}.
This non-conservative system of PDEs has been successfully employed to study phenomena in the solar system \cite{Toth2010, glocer2009multifluid, rubin2014comet, rubin2015self}.

The multi-ion magnetohydrodynamic (MHD) system proposed by \citet{Toth2010} is constructed through the formulation of conservation equations governing mass, momentum, and hydrodynamic energy for electrons and $N_i$ ion species. 
This formulation is based on the Euler equations of gas dynamics, supplemented by a Lorentz force term in the momentum and energy equation of the ions, and an induction equation describing the evolution of the magnetic field, derived from Maxwell's equations. 
To circumvent complexities associated with stiff terms dependent on the speed of light, Ampère's law is not directly solved for the evolution of the electric field. 
Instead, the electron momentum equation is streamlined by assuming negligible electron density compared to the ions, leading to the derivation of a generalized Ohm's law.
The obtained generalized Ohm's law is subsequently incorporated into the Lorentz force terms and the induction equation. A comprehensive derivation of this system is provided in \cite{Toth2010}. 
In our work, we adopt a version of the multi-ion MHD system, as presented in \cite{Toth2010}, wherein the explicit self-consistent solution for the electron hydrodynamic energy is not pursued. Instead, we employ a phenomenological expression for the electron pressure, denoted as $p_e$.

The multi-ion MHD model described by \citet{Toth2010} can be written compactly in non-conservative ``curl'' form as
\begin{equation} \label{eq:multi-ion_simp}
    \frac{\partial}{\partial t}
    \underbrace{
    \begin{pmatrix}
    \rho_k \\ \rho_k \vec{v}_k  \\ \hydroEner_k \\ \vec{B}
    \end{pmatrix}
    }_{\state{u}^{\rm{Toth}}}
    + \Nabla \cdot
    \underbrace{
    \begin{pmatrix} 
    \rho_k \vec{v}_k \\
    \rho_k \vec{v}_k\, \vec{v}_k^{\,T} + 
     p_k \threeMatrix{I}   \\
    \vec{v}_k\left(\hydroEner_k + p_k \right)  \\
    \threeMatrix{0} 
    \end{pmatrix}
    }_{\blocktensor{f}^{\supEuler}}
    +
    \underbrace{
    \begin{pmatrix}
    0 \\
    n_k q_k (\vec{v}^+ - \vec{v}_k) \times \vec{B} \\
    n_k q_k \vec{v}_k \cdot (\vec{v}^+ - \vec{v}_k) \times \vec{B}  \\
    \vec{0} 
    \end{pmatrix}
    }_{\state{g}}
    +
    \underbrace{
    \begin{pmatrix}
    0 \\
    - \frac{n_k q_k}{n_e e } \left( \vec{J} \times \vec{B} - \Nabla p_e \right) \\
     - \frac{n_k q_k \vec{v}_k}{n_e e} \cdot \left( \vec{J} \times \vec{B} - \Nabla p_e \right)  \\
    - \Nabla \times (\vec{v}^+ \times \vec{B})
    \end{pmatrix}
    }_{\noncon^{\rm{Toth}}}
    =
    \state{0},
\end{equation}
where $\rho_k$, $n_k$, $q_k$, $\vec{v}_k = (v_{k,1}, v_{k,2}, v_{k,3})$, and $\hydroEner_k$  denote the density, number density, charge, velocity, and hydrodynamic energy of the ion species $k$, respectively, $\vec{J}$ is the electric current, $\vec{B} = (B_1, B_2, B_3)$ is the magnetic field, $e$ is the elementary charge,
and the number density of electrons is obtained assuming quasi-neutrality,
\begin{equation} \label{eq:electronNumberDensity}
    n_e = \frac{1}{e} \sum_k n_k q_k.
\end{equation}
The charge-averaged ion velocity is defined as
\begin{equation}\label{eq:chargeAveragedVel}
\vec{v}^+ = \sum_k n_k q_k \vec{v}_k / (e n_e).
\end{equation}
We assume calorically perfect plasmas, for which the hydrodynamic energy is computed as
\begin{equation} \label{eq:hydroEnergy}
    \hydroEner_k = \frac{1}{2}\rho_k \left\|\vec{v}_k\right\|^2 + \frac{p_k}{\gamma_k -1},
\end{equation}
with the heat capacity ratio $\gamma_k:=c_{p,k}/c_{v,k}$ and the ion pressure $p_k$. All individual densities and pressures are positive, i.e. none of the species is in a vacuum state. 

As mentioned above, the electron pressure $p_e$ can be estimated with a phenomenological model. 
For instance, it can be computed as a fraction of the total ion pressure, e.g., $p_e = \alpha \sum_k p_k$, where $\alpha$ is heuristically determined \cite{Toth2010}, or it can be computed assuming a constant electron temperature \cite{glocer2009multifluid}.

We remark that we write the system only for species $k$ for compactness, but the multi-ion system considers $k=1,\ldots,N_i$. 
Furthermore, we note that the vacuum permeability has been omitted for readability, which implies a non-dimensionalization such that $\mu_0 =1$.

The density of each species is related to the number density through $\rho_k := n_k m_k$, where $m_k$ is the molecular mass of species $k$.
Furthermore, we compute the electric current with the non-relativistic Ampère's law, $\vec{J}=\Nabla \times \vec{B}$.
Therefore, we can rewrite \eqref{eq:multi-ion_simp} by removing some unknowns as
\begin{equation} \label{eq:multi-ion}
    \frac{\partial}{\partial t}
    \underbrace{
    \begin{pmatrix}
    \rho_k \\ \rho_k \vec{v}_k  \\ \hydroEner_k \\ \vec{B}
    \end{pmatrix}
    }_{\state{u}^{\rm{Toth}}}
    + \Nabla \cdot
    \underbrace{
    \begin{pmatrix} 
    \rho_k \vec{v}_k \\
    \rho_k \vec{v}_k\, \vec{v}_k^{\,T} + 
     p_k \threeMatrix{I}   \\
    \vec{v}_k\left(\hydroEner_k + p_k \right)  \\
    \threeMatrix{0} 
    \end{pmatrix}
    }_{\blocktensor{f}^{\supEuler}}
    +
    \underbrace{
    \begin{pmatrix}
    0 \\
    r_k \rho_k (\vec{v}^+ - \vec{v}_k) \times \vec{B} \\
    r_k \rho_k \vec{v}_k \cdot (\vec{v}^+ - \vec{v}_k) \times \vec{B}  \\
    \vec{0} 
    \end{pmatrix}
    }_{\state{g}}
    +
    \underbrace{
    \begin{pmatrix}
    0 \\
    - \frac{r_k \rho_k}{n_e e } \left( \Nabla \times \vec{B} \times \vec{B} - \Nabla p_e \right) \\
     - \frac{r_k \rho_k \vec{v}_k}{n_e e} \cdot \left( \Nabla \times \vec{B} \times \vec{B} - \Nabla p_e \right)  \\
    - \Nabla \times (\vec{v}^+ \times \vec{B})
    \end{pmatrix}
    }_{\noncon^{\rm{Toth}}}
    =
    \state{0},
\end{equation}
where the (constant) charge to mass ratio of each species is $r_k = q_k / m_k$.
We rewrite the quasi-neutrality constraint \eqref{eq:electronNumberDensity} and the charge-averaged velocity as \eqref{eq:chargeAveragedVel}
\begin{align} \label{eq:electronNumberDensity2}
    n_e e &= \sum_k \rho_k r_k,
\\
\label{eq:chargeAveragedVel2}
 \vec{v}^+ &= \sum_k r_k \rho_k \vec{v}_k / (e n_e).
\end{align}

We remark that \eqref{eq:multi-ion} cannot be written in conservation form.
Moreover, we remark that \eqref{eq:multi-ion} does not make any assumptions on the divergence of the magnetic field since the spatial derivatives in $\noncon^{\rm{Toth}}$ are written in ``curl'' form.

\paragraph{Thermodynamic properties of the system:}

To study the entropic properties of the governing equations, the conventional approach involves seeking a strictly convex mathematical entropy, denoted as $S$. 
This entropy is directly proportional to the thermodynamic entropy of the system and is carefully chosen to satisfy a conservation law for smooth solutions. 
The derivation of this conservation law involves contracting the PDE with entropy variables, defined as $\entVar \coloneqq \partial S / \partial \state{u}$.
The strict convexity of the mathematical entropy $S$ is needed for the symmetrization of the PDE, which imparts the system with a host of desirable properties, as elucidated in previous studies \cite{godlewski1998numerical, Chandrashekar2016}. 

A suitable mathematical entropy for the single-fluid MHD system is the entropy density with a switched sign divided by the quantity $(\gamma -1)$ \cite{barth1999numerical, Chandrashekar2016, Derigs2018}. \q{c14r2}{Thus, following the analysis done for \change{the} multi-component Euler \change{equations} in \cite{hantke2018analysis}, a natural choice for the mathematical entropy of the multi-ion system is the sum of the individual entropies of all species,}
\begin{equation} \label{eq:thermoEntropy}
    S = \sum_k \frac{-\rho_k s_k}{\gamma_k -1},
\end{equation}
where $s_k$ is the specific entropy of ion $k$,
\begin{align}\label{s_k-def}
    \begin{split}
        s_k &= \ln(p_k) - \gamma_k \ln(\rho_k) \\
            &= \ln\left((\gamma_k - 1) \left(\hydroEner_k - \frac{1}{2 \rho_k} \norm{\rho_k v_k}^2  \right)\right) - \gamma_k \ln(\rho_k).
    \end{split}
\end{align}

The entropy variables for this choice of mathematical entropy read
\begin{equation}
\entVar^{\rm{Toth}} = \frac{\partial S}{\partial \state{u}^{\rm{Toth}} } = \left(\frac{\gamma_k-s_k}{\gamma_k-1} - \beta_k \norm{\vec{v_k}}^2,~2\beta_k v_{k,1},~2\beta v_{k,2},~2\beta v_{k,3},~-2\beta_k,0, 0, 0\right)^T,
\label{eq:entvars_toth}
\end{equation}
where $\beta_k := \frac{\rho_k}{2p_k}$ is proportional to the temperature of the $k$-th species.

Unfortunately, \eqref{eq:thermoEntropy} is not strictly convex with respect to $\state{u}^{\rm{Toth}} =  (\rho_k, \rho_k \vec{v}_k, \hydroEner_k, \vec{B})^T$ as it only depends on the hydrodynamic states of species $k=1, \ldots, N_i$ and does not depend on the magnetic field $\vec{B}$. 
As a result, the last three entropy variables are zero, and hence the Hessian matrix $\matcal{H}^{-1} \coloneqq \partial \entVar^{\rm{Toth}} / \partial \state{u}^{\rm{Toth}}$ is singular.

It is possible to use $\entVar^{\rm{Toth}}$ from \eqref{eq:entvars_toth} to contract \eqref{eq:multi-ion} into a conservation law for $S$ \eqref{eq:thermoEntropy}.
However, since the transformation from state to entropy variables $\state{u}^{\rm{Toth}} \mapsto \entVar^{\rm{Toth}}$ \eqref{eq:entvars_toth} is not injective, it is not possible to use it to construct entropy-stable discretizations of \eqref{eq:multi-ion} that need the injective property.
For instance, the schemes in \cite{chan2019efficient, chan2019discretely, rueda2023entropy} require an injective transformation from state to entropy variables $\state{u} \mapsto \entVar$ and its inverse transformation $\entVar \mapsto \state{u}$.
Moreover, the fact that $S$ is not strictly convex with respect to the state quantities implies that the transformation $\state{u}^{\rm{Toth}} \mapsto \entVar^{\rm{Toth}}$ does not symmetrize the system \cite{godlewski1998numerical, barth1999numerical, Chandrashekar2016}.

In what follows, we  propose an equivalent formulation of the system \eqref{eq:multi-ion}, such that a thermodynamics-based mathematical entropy \eqref{eq:thermoEntropy} produces a one-to-one (injective) mapping between the state and entropy variables and becomes strictly convex with respect to the state quantities, as in the entropy-consistent single-fluid MHD equations \cite{Chandrashekar2016, Derigs2017, Derigs2018}.

\subsection{Alternative Formulation of the Ideal Multi-Ion MHD System and its Thermodynamic Properties} \label{sec:modmultiion}

The main difference between \eqref{eq:multi-ion} and the entropy-consistent single-fluid MHD system employed in prior works \cite{Chandrashekar2016, Derigs2017, Derigs2018} under the one-ion-species limit lies in the fact that the entropy-consistent single-fluid MHD equations characterize the evolution of the total energy encompassing both hydrodynamic and magnetic components, as opposed to exclusively the hydrodynamic energy.

To address this difference, we propose an algebraic manipulation, i.e., a reformulation of the multi-ion MHD system based on another set of state variables. We introduce an artificial variable $E_k$ replacing the hydrodynamic energy variable:
\begin{equation} \label{eq:hydroEnerPlusMag}
E^{\mathrm{mod}}_k = \hydroEner_k + \underbrace{\frac{1}{2} \norm{\vec{B}}^2}_{E_{\rm{mag}}}.
\end{equation}

Note that $E_k$ serves as a mathematical state variable and has no direct physical interpretation. For instance, the physical quantity total energy of the whole system can be computed with these variables as
\begin{equation}
E = \sum_{k=1}^{N_i} \hydroEner_k + E_{\rm{mag}} = \sum_{k=1}^{N_i} E^{\mathrm{mod}}_k - (N_i - 1) E_{\rm{mag}}.
\end{equation}
The equation for the evolution of $E_k$ can be obtained from \eqref{eq:multi-ion} by a brief calculation
\begin{align*}
    \bigpartialderiv{E^{\mathrm{mod}}_k}{t} 
    =
    \bigpartialderiv{\hydroEner_k}{t} + \bigpartialderiv{E_{\rm{mag}}}{t}
    =
    \bigpartialderiv{\hydroEner_k}{t} +  \frac{1}{2} \bigpartialderiv{\norm{B}^2}{t}
    =
    \bigpartialderiv{\hydroEner_k}{t} + \vec{B} \cdot \bigpartialderiv{\vec{B}}{t}.
\end{align*}
and \eqref{eq:multi-ion} is then rewritten as
\begin{align} \label{eq:multi-ion_mod}
    \frac{\partial}{\partial t}
    \underbrace{
    \begin{pmatrix}
    \rho_k \\ \rho_k \vec{v}_k  \\ E^{\mathrm{mod}}_k \\ \vec{B}
    \end{pmatrix}
    }_{\state{u}^{\rm{mod}}}
    + \Nabla \cdot
    \underbrace{
    \begin{pmatrix} 
    \rho_k \vec{v}_k \\
    \rho_k \vec{v}_k\, \vec{v}_k^{\,T} + 
     p_k \threeMatrix{I}   \\
    \vec{v}_k\left(\hydroEner_k + p_k \right)  \\
    \threeMatrix{0} 
    \end{pmatrix}
    }_{\blocktensor{f}^{\supEuler}}
    &+
    \underbrace{
    \begin{pmatrix}
    0 \\
    r_k \rho_k (\vec{v}^+ - \vec{v}_k) \times \vec{B} \\
    r_k \rho_k \vec{v}_k \cdot (\vec{v}^+ - \vec{v}_k) \times \vec{B}  \\
    \vec{0} 
    \end{pmatrix}
    }_{\state{g}}
    \nonumber \\
    &+
    \underbrace{
    \begin{pmatrix}
    0 \\
    - \frac{r_k \rho_k}{n_e e } \left( \Nabla \times \vec{B} \times \vec{B} - \Nabla p_e \right) \\
     - \frac{r_k \rho_k \vec{v}_k}{n_e e} \cdot \left( \Nabla \times \vec{B} \times \vec{B} - \Nabla p_e \right) - \vec{B} \cdot \Nabla \times (\vec{v}^+ \times \vec{B})  \\
    - \Nabla \times (\vec{v}^+ \times \vec{B})
    \end{pmatrix}
    }_{\noncon^{\rm{mod}}}
    =
    \state{0},
\end{align}

We remark that \eqref{eq:multi-ion_mod} reduces to the entropy-consistent single-fluid MHD system of \cite{Chandrashekar2016, Derigs2017, Derigs2018} when a single ion species is used.
The only difference is that the non-conservative spatial derivative terms are given in curl form instead of including a conservative  magneto-hydrodynamic term and the Godunov-Powell non-conservative term.

As before, we choose the total mathematical entropy as the sum of the individual entropies of all species,
\eqref{eq:thermoEntropy},
\begin{equation*}
    S = \sum_k \frac{-\rho_k s_k}{\gamma_k -1},
\end{equation*}
where $s_k$ can now be written in terms of the new state variables, $\state{u}^{\rm{mod}} = (\rho_k, \rho_k \vec{v}_k, E^{\mathrm{mod}}_k, \vec{B})^T$, as
\begin{align}\begin{split}
    s_k &= \ln(p_k) - \gamma_k \ln(\rho_k) \\
        &= \ln\left((\gamma_k - 1) \left(E^{\mathrm{mod}}_k - \frac{1}{2 \rho_k} \norm{\rho_k \vec{v}_k}^2 - \frac{1}{2} \norm{\vec{B}}^2  \right)\right) - \gamma_k \ln(\rho_k).
        \end{split}
\end{align}
Since the thermodynamic entropy now depends on the magnetic field, we obtain the entropy variables
\begin{equation}\label{eq:entvars_mod}
\entVar^{\rm{mod}} 
= \frac{\partial S}{\partial \state{u}^{\rm{mod}}} 
= \left(\frac{\gamma_k-s_k}{\gamma_k-1} - \beta_k \norm{\vec{v_k}}^2,
    ~2\beta_k v_{k,1}, ~2\beta_k v_{k,2}, ~2\beta_k v_{k,3},
    ~-2\beta_k,
    ~2\beta_+ B_1,~2\beta_+B_2,~2\beta_+B_3\right)^T,
\end{equation}
with $\beta_+ = \sum_k \beta_k$.

The mathematical entropy $S$ is obtained as a sum of functions that are strictly convex with respect to the state variables $\state{u}^{\rm{mod}}$, due to the fact the individual densities and pressures are positive.
Since the whole set of state variables in $\state{u}^{\rm{mod}}$ is encompassed by the individual strict convex functions that we sum to obtain $S$, the resulting mathematical entropy is strictly convex with respect to $\state{u}^{\rm{mod}}$ (see, e.g., \cite{boyd2004convex}[Chapter 3.2: Operations that preserve convexity]).
Moreover, \eqref{eq:entvars_mod} is a bijective transformation $\state{u}^{\rm{mod}} \mapsto \entVar^{\rm{mod}}$,
whose inverse is obtained as follows: division by $\beta_k > 0$, that can directly be read from $\entVar^{\rm{mod}}$, yields the velocities $v_{k,1}, v_{k,2}, v_{k,3}$. Division by $\beta_+ = \sum_k \beta_k$, in turn, yields the magnetic field $B_1, B_2, B_3 $. With these quantities at hand, $s_k$ is recovered from the first component of $\entVar^{\rm{mod}}$. The relation $\beta_k = \frac{\rho_k}{2 p_k}$ and equation \eqref{s_k-def} provide, firstly, the values of density and pressure, and subsequently the remaining components of $\state{u}^{\rm{mod}}$.

For smooth solutions the contraction of \eqref{eq:multi-ion_mod} with \eqref{eq:entvars_mod} yields an auxiliary conservation law.
In particular, analyzing \eqref{eq:multi-ion_mod} term by term we obtain:
\begin{align*}
    (\entVar^{\mathrm{mod}})^T \Nabla \cdot \blocktensor{f}^{\supEuler} =&
    \Nabla \cdot \left[\sum_k \vec{v}_k \left( \frac{-\rho_k s_k}{\gamma_k - 1} \right) \right], & \textrm{(same as Euler)}\\
    (\entVar^{\mathrm{mod}})^T \state{g} =& \,\,0, \\
    (\entVar^{\mathrm{mod}})^T \noncon^{\rm{mod}} =& -2 \sum_k \beta_k \vec{B} \cdot \left( - \Nabla \times (\vec{v}^+ \times \vec{B})
    +\right)
    + 2 \beta^+ \vec{B} \cdot \left( - \Nabla \times (\vec{v}^+ \times \vec{B})
    +\right) = 0
\end{align*}
Therefore, for smooth solutions, the following entropy conservation law holds
\begin{equation}\label{eq:entropy_cons_law}
    \bigpartialderiv{S}{t} + \Nabla \cdot \vec{f}^S = 0,
\end{equation}
where the entropy flux is given by
\begin{equation}
\vec{f}^S = \sum_k \vec{v}_k \left( \frac{-\rho_k s_k}{\gamma_k - 1} \right).    
\end{equation}

\q{c35r3}{
\change{
\begin{remark}
    The approach used in this section shares some similarities with the strategy proposed by \citet{kumar2012entropy} for obtaining an entropy-stable discretization for a multi-fluid MHD model, which includes hydrodynamic equations for the electrons and a single ion species, along with the (linear) Maxwell equations. However, there are three key differences between our method and that of \citet{kumar2012entropy}.
:
    \begin{enumerate}
        \item \citet{kumar2012entropy} demonstrate that a strictly convex entropy function for their system can be formulated by adding the term $S_2 = \frac{1}{2} \norm{\vec{B}}^2 + \frac{1}{2} \norm{\vec{\mathbb{E}}}^2$ to the \textbf{mathematical entropy} of the hydrodynamic system ($S$), where $\mathbb{E}$ represents the electric field, and $S_2$ is the quadratic entropy associated with Maxwell's equations. In contrast, we show that a strictly convex entropy function for our system can be formulated by adding the term $E_{\rm{mag}} = \frac{1}{2} \norm{\vec{B}}^2$ to the \textbf{total hydrodynamic energy} of each ion species, while using the sum of the entropies of our ion species \eqref{eq:thermoEntropy} as the mathematical entropy.

        \item In the multi-ion MHD model, the electric field and the electron hydrodynamic state are not included as state variables in order to avoid the stiffness introduced by the speed of light and the small electron mass. As a result, the model features a non-linear induction equation, derived from the generalized Ohm's law (see, e.g., \cite{Toth2010}), and not a linear induction equation as in Maxwell's equations. Additionally, the multi-ion MHD system simultaneously solves for multiple ion species.

        \item Due to the absence of the electric field as a state variable, a similar approach to that of \citet{kumar2012entropy}  involves retaining the total hydrodynamic energies as state quantities and redefining the mathematical entropy as $S_K = S + E_{\rm{mag}}$. This leads to a different set of entropy variables,
\begin{equation*}
\entVar^K = \frac{\partial S_K}{\partial \state{u}} = \left(\frac{\gamma_k-s_k}{\gamma_k-1} - \beta_k \norm{\vec{v_k}}^2,~2\beta_k v_{k,1},~2\beta v_{k,2},~2\beta v_{k,3},~-2\beta_k,~B_1,~B_2,~B_3\right)^T.
\end{equation*}
Unfortunately, we have not been able to use $\entVar^K$ to contract the multi-ion MHD system into an entropy conservation law for smooth solutions.

    \end{enumerate} 
\end{remark}
}
}

\subsection{GLM Formulation and its Thermodynamic Properties} \label{sec:glm}
According to \cite[Eq. (13)]{Toth2010}  the \emph{total} momentum of the system \eqref{eq:multi-ion_mod} is conserved if the magnetic field is assumed to be divergence free, i.e. $\nabla\cdot \vec{B} =0$.  One possibility to construct a numerical scheme that evolves towards a solution satisfying the divergence-free constraint is the method of general Lagrange multipliers (GLM), introduced in \cite{Munz2000}. The GLM approach was first applied to the ideal MHD equations in \cite{Dedner2002} and, more recently, an entropy-stable version of the GLM-MHD formulation was derived and analyzed in \cite{Derigs2018}. 

As we will see below, the GLM technique does not only enhance the physical accuracy of magnetic field evolution but also significantly bolsters the robustness of computational schemes.
In this section, we will integrate the GLM technique into the multi-ion MHD framework, thereby evolving it into the multi-ion GLM-MHD system.
Furthermore, in the forthcoming sections, we will detail the derivation of both entropy-conservative and entropy-stable schemes tailored for the multi-ion MHD and multi-ion GLM-MHD systems, respectively.
Throughout the remainder of this paper, to facilitate better understanding, highlight the modularity of the GLM approach, and enhance the document's readability, we will denote GLM terms in \red{red}.

The method consists in modifying the PDE system by introducing a Lagrange multiplier, i.e.~an auxiliary scalar field, $\red{\psi}$,  that drives the magnetic field towards a divergence-free state. The evolution of the Lagrange multiplier is represented by a supplementary transport PDE. The additional terms are constructed such that they vanish as $\nabla \cdot \vec{B} = 0$, i.e. whenever the constraint is fulfilled.

In the following we \change{follow} the derivation in \cite{Derigs2018}. Firstly, we introduce the GLM-variable $\red{\psi}$ that evolves according to the  PDE
\begin{equation}
    {\color{red}\frac{\partial \psi}{\partial t} + \Nabla\cdot(c_h\vec{B}) - (\Nabla\psi)\cdot \vec{v}^+ = 0},
\end{equation}
where $\red{c_h}$ denotes the hyperbolic divergence cleaning speed and the transport speed of  ${\color{red}\psi}$ is set to the charge averaged velocity $\vec{v}^+$ in accordance with the velocity acting on $\Vec{B}$ in the PDE~\eqref{eq:multi-ion_mod}. 
Secondly, we augment the dynamics of the magnetic field:
\begin{equation}
    \frac{\partial B}{\partial t} = - \Nabla \times (\vec{v}^+ \times \vec{B}) + {\color{red}\Nabla\cdot(c_h\psi\vec{I})}.
\end{equation}
The total energy  $E_k$, cf.~\eqref{eq:hydroEnerPlusMag}, includes the magnetic energy which now depends on the GLM-variable ${\color{red}\psi}$. Thus, any variation of ${\color{red}\psi}$ causes erroneous variation of the hydrodynamic energy $\hydroEner_k$ by introducing non-physical thermal effects.  Similar to \cite{Derigs2018} we take this into account  by considering the alternative total energy:
\begin{equation}
    E_{k} :=  \hydroEner_k +\frac12 \|\vec{B}\|^2 + {\color{red}\frac12 \psi^2}.
\end{equation}
Its time derivative is computed as 
\begin{equation}
    \frac{\partial  E_{k }}{\partial t} 
    = \frac{\partial \hydroEner_k}{\partial t} 
    + \vec{B}\cdot \frac{\partial  \vec{B}}{\partial t} 
    \red{+ \psi \frac{\partial \psi}{\partial t}}.
\end{equation}
Hence, the multi-ion GLM-MHD system reads:
\begin{align} \label{eq:multi-ion_mod_glm}
    \frac{\partial}{\partial t}
    \underbrace{
    \begin{pmatrix}
    \rho_k \\ \rho_k \vec{v}_k  \\ E_{k} \\ \vec{B} \\ 
    \red{\psi}
    \end{pmatrix}
    }_{\state{u}}
    + \Nabla \cdot
    \underbrace{
    \begin{pmatrix} 
    \rho_k \vec{v}_k \\
    \rho_k \vec{v}_k\, \vec{v}_k^{\,T} + 
     p_k \threeMatrix{I}   \\
    \vec{v}_k\left(\hydroEner_k + p_k \right)  \\
    \threeMatrix{0} \\ 
    \vec{0}
    \end{pmatrix}
    }_{\blocktensor{f}^{\supEuler}}
    &+
    \red{
    \Nabla \cdot
    \underbrace{
    \begin{pmatrix} 
    \vec{0} \\
    \threeMatrix{0} \\
    c_h \psi \vec{B} \\
    c_h \psi \threeMatrix{I} \\ c_h \vec{B}\\
    \end{pmatrix}
    }_{\blocktensor{f}^{\supGLM}}
    }
    +
    \underbrace{
    \begin{pmatrix}
    0 \\
    r_k \rho_k (\vec{v}^+ - \vec{v}_k) \times \vec{B} \\
    r_k \rho_k \vec{v}_k \cdot (\vec{v}^+ - \vec{v}_k) \times \vec{B}  \\
    \vec{0} \\
    0
    \end{pmatrix}
    }_{\state{g}}
    \nonumber \\
    &+
    \underbrace{
    \begin{pmatrix}
    0 \\
    - \frac{r_k \rho_k}{n_e e } \left( \Nabla \times \vec{B} \times \vec{B} - \Nabla p_e \right) \\
     - \frac{r_k \rho_k \vec{v}_k}{n_e e} \cdot \left( \Nabla \times \vec{B} \times \vec{B} - \Nabla p_e \right) - \vec{B} \cdot \Nabla \times (\vec{v}^+ \times \vec{B})  \\
    - \Nabla \times (\vec{v}^+ \times \vec{B}) \\
    0
    \end{pmatrix}
    }_{\noncon^{\rm{mod}}}
    +
    \red{
    \underbrace{
    \begin{pmatrix}
        \vec{0} \\
        \threeMatrix{0}\\
        \vec{v}^+ \psi \\
        \threeMatrix{0} \\
        \vec{v}^+
    \end{pmatrix}
    \cdot
    \Nabla \psi
    }_{\noncon^{\supGLM}}
    }
    =
    \state{0}.
\end{align}

The specific entropy of each species is now given by:
\begin{align}\begin{split}
    s_k &= \ln(p_k) - \gamma_k \ln(\rho_k) \\
        &= \ln\left((\gamma_k - 1) \left(E^{\supGLM}_{ k} - \frac{1}{2 \rho_k} \norm{\rho_k v_k}^2 - \frac{1}{2} \norm{B}^2 - {\color{red}\frac12\psi^2}  \right)\right) - \gamma_k \ln(\rho_k).
        \end{split}
\end{align}
and a simple computation yields the entropy variables
\begin{align}
    \entVar &= 
    \left( (\entVar^\text{mod})^T, {\color{red}2\beta_+\psi }\right)^T
    \nonumber\\
    &=
    \left(\frac{\gamma_k-s_k}{\gamma_k-1} - \beta_k \norm{\vec{v_k}}^2,
    ~2\beta_k v_{k,1}, ~2\beta_k v_{k,2}, ~2\beta_k v_{k,3},
    ~-2\beta_k,
    ~2\beta_+ B_1,~2\beta_+B_2,~2\beta_+B_3,
    {\color{red}2\beta_+\psi }\right)^T.
\end{align}

The last component of the entropy variables, that originates from the GLM modification, does not play any role while contracting the non-GLM parts of the PDE~\eqref{eq:multi-ion_mod_glm}. Therefore, to obtain an entropy conservation law  \eqref{eq:entropy_cons_law} it is sufficient to check that
\begin{align}
    &\frac{1}{2\beta_+}\entVar^T\,\blocktensor{f}^{\supGLM} = {\color{red}-\Nabla \cdot \left(c_h\psi\vec{B}\right) 
    + \vec{B} \cdot \Nabla \cdot \left(c_h \psi \mat{I}\right) 
    + \psi\left(\Nabla \cdot \left( c_h\vec{B}\right)\right)} = 0,\\
    &\frac{1}{2\beta_+}\entVar^T\,\noncon^{\supGLM} = 
    {\color{red}-\left(\Nabla\psi\right)\vec{v}^+\psi + \psi\left(\Nabla\psi\right)\vec{v}^+} = 0.
\end{align}
With the same argument as in the previous section one can compute the inverse of the bijective mapping $\state{u} \mapsto \state{w}$, i.e. the multi-ion GLM-MHD system is entropy-consistent.

\subsection{``Quasi-conservative'' Form of the Alternative Formulation of the Multi-Ion MHD System}

Now that we have an entropy-consistent multi-ion MHD system with incorporated GLM diverence cleaning, the goal is to write it in ``quasi-conservative'' form.
In other words, we want to encapsulate as many terms as possible into a divergence operator, such that the system resembles the well-known conservative formulation of the single-fluid MHD equations.
The ``quasi-conservative'' form will be useful for three main reasons:
\begin{enumerate}
    \item It will facilitate the implementation of numerical schemes that are momentum-conservative in the limit of vanishing magnetic field divergence and total-energy-conservative under more specific conditions.
    A detailed discussion about the achievable conservation properties for multi-ion MHD can be found below.
    \item It will facilitate the derivation of entropy-conservative and entropy-stable numerical schemes that are consistent with existing state-of-the-art discretization schemes for the single-fluid MHD equations, such as the one proposed by \citet{Derigs2018}.
    \item It will allow us to derive a generalization of the Godunov-Powell non-conservative term for plasmas with multiple ion species.
\end{enumerate}

In a similar fashion as in single-fluid MHD, the non-conservative term of the momentum equation of the alternative multi-ion MHD system \eqref{eq:multi-ion_mod} can be written in quasi-conservative form by rearranging the curl operator,
\begin{align}
\label{eq:TransformMomentumEq}
   \noncon^{\rm{mod}}_{\rho_k \vec{v}_k} 
   = - \frac{r_k \rho_k}{n_e e } \left( \Nabla \times \vec{B} \times \vec{B} - \Nabla p_e \right)
   =
   \frac{r_k \rho_k}{n_e e} \Nabla \cdot \left(\frac{1}{2} \norm{\vec{B}}^2 \mat{I} - \vec{B}\vec{B}^T 
    + p_e \mat{I}\right)
    +
    \frac{r_k \rho_k}{n_e e} \vec{B} 
    \left( \Nabla \cdot \vec{B} \right).
\end{align}

\q{c17r2}{
Similarly, the non-conservative term of the induction equation of  \eqref{eq:multi-ion_mod} can be written in conservation form,
\begin{align} \label{eq:TransformInductionEq}
   \noncon^{\rm{mod}}_{\vec{B}} 
   = - \Nabla \times (\vec{v}^+ \times \vec{B})
   =
   \Nabla \cdot \left(
   \vec{v}^+ \vec{B}^T - \vec{B} (\vec{v}^+)^T
   \right).
\end{align}
\q{c17r2}{
\change{}
}

The obtention of a quasi-conservative form for each species' energy term is more involved and requires us to define two new auxiliar variables: (i) each species' contribution to the charge-averaged ion velocity,
\begin{equation} \label{eq:v_plus_k}
    \vec{v}^+_k \coloneqq \frac{r_k \rho_k \vec{v}_k}{e n_e},
\end{equation}
and (ii) its complement,
\begin{equation} \label{eq:v_minus_k}
    \vec{v}^-_k \coloneqq \vec{v}^+ - \vec{v}^+_k = \sum_{s\ne k} \frac{r_s \rho_s \vec{v}_s}{e n_e}.
\end{equation} 

Using \eqref{eq:v_plus_k} and \eqref{eq:v_minus_k}, we can rewrite the non-conservative term of the energy equation as
\begin{align} \label{eq:TransformEnergyEq}
    \noncon^{\rm{mod}}_{E_k} =&  - \frac{r_k \rho_k \vec{v}_k}{n_e e}\cdot \left( \Nabla \times \vec{B} \times \vec{B} - \Nabla p_e \right) - \vec{B} \cdot \Nabla \times (\vec{v}^+ \times \vec{B}) 
    \nonumber \\
    =&  - \vec{v}^+_k \cdot \left( \Nabla \times \vec{B} \times \vec{B} - \Nabla p_e \right) 
    - \vec{B} \cdot \Nabla \times (\vec{v}^+_k \times \vec{B})
    - \vec{B} \cdot \Nabla \times (\vec{v}^-_k \times \vec{B}).
\end{align}

Equation \eqref{eq:TransformEnergyEq} can be further manipulated to obtain
\begin{align} \label{eq:TransformEnergyEq2}
    \noncon^{\rm{mod}}_{E_k} =& - \vec{v}^+_k \cdot \left( \Nabla \times \vec{B} \times \vec{B} - \Nabla p_e \right) 
    - \vec{B} \cdot \Nabla \times \left(\vec{v}^+_k \times \vec{B} \right)
    - \vec{B} \cdot \Nabla \times \left(\vec{v}^-_k \times \vec{B} \right)
    \\
    =&
    \Nabla \cdot \left( \vec{v}^+_k\,\|\vec{B}\|^2 - \vec{B}\left(\vec{v}^+_k\cdot\vec{B}\right) \right)
    + \vec{v}^+_k \cdot \Nabla p_e
    + \vec{B} \cdot 
    \Nabla \cdot \left(\vec{v}^-_k \vec{B}^T - \vec{B} (\vec{v}^-_k)^T  \right).
\end{align}

\change{
The quasi-conservative versions of the non-conservative terms for the momentum \eqref{eq:TransformMomentumEq} and total energy \eqref{eq:TransformEnergyEq2} of each ion species, and for the induction equation \eqref{eq:TransformInductionEq}, contain a single term that is proportional to $(\Nabla \cdot \vec{B})$.
It is possible to show that these terms are consistent with the single-fluid MHD non-conservative term proposed by
\citet{brackbill1980effect}.
However, they are not consistent with the non-conservative terms in the single-fluid GLM-MHD system of \citet{Derigs2018}, which uses the Godunov-Powell non-conservative term containing terms proportional to $(\nabla \cdot \vec{B})$ in both the induction and total energy equations.
To achieve consistency with \citet{Derigs2018}, we modify \eqref{eq:TransformInductionEq} and \eqref{eq:TransformEnergyEq2} to obtain
\begin{align} \label{eq:nonconB_quasi}
   \noncon^{\rm{quasi}}_{\vec{B}} 
   &= 
   \noncon^{\rm{mod}}_{\vec{B}} 
   +
   \vec{v}^{m}
   \left( \Nabla \cdot \vec{B} \right),
\\
\label{eq:nonconEk_quasi}
    \noncon^{\rm{quasi}}_{E_k}
    &=
    \noncon^{\rm{mod}}_{E_k}
    +
    \left(\vec{v}^m \cdot \vec{B} \right)
   \left( \Nabla \cdot \vec{B} \right).
\end{align}

These additional terms that are proportional to the magnetic field divergence are entropy-consistent and can be derived directly from the generalized Maxwell's equations, which do not assume $\Nabla \cdot \vec{B} = 0$, as shown by \citet{Derigs2018}.
The quantity $\vec{v}^{m}$ denotes the velocity, with which magnetic monopoles would move in space.
As long as $\vec{v}^{m}$ is chosen as some average of the ion velocities, it is possible to show consistency with the model of \citet{Derigs2018}.
In this work, we have made the assumption that magnetic monopoles move at the ion charge-averaged velocity,
$\vec{v}^m \coloneqq \vec{v}^+$.

We note that the addition of these terms symmetrizes the system in the limit of one ion species, as shown by \citet{godunov1972symmetric}.
The symmetric form of the equations reads
\begin{equation} \label{eq:symmform}
    \underline{\bigpartialderiv{\state{u}}{\entVar}}
    \bigpartialderiv{\entVar}{t}
    +
    \mat{A}^{x_j}
    \underline{\bigpartialderiv{\state{u}}{\entVar}}
    \bigpartialderiv{\entVar}{x_j} = \state{0},
\end{equation}
where we make use of Einstein's notation with the index $j$ and $\mat{A}^{x_j}$ is the so-called flux (and non-conservative term's) Jacobian in the direction $x_j$. 
Equation \eqref{eq:symmform} is known as the symmetric form of the system since the so-called entropy Jacobian $\underline{\partial \state{u} / \partial \entVar}$ is a symmetric positive-definite (SPD) matrix and the products $\underline{A}^{x_j} \underline{\partial \state{u} / \partial \entVar}$ are symmetric matrices.

In the case of an arbitrary number of ion species,  $\underline{\partial \state{u} / \partial \entVar}$ is still SPD (see Section \ref{sec:ES_flux} for a proof in the discrete setting).
However, we have found that the multi-ion MHD and multi-ion GLM-MHD systems are unfortunately not symmetrizable as the matrices $\underline{A}^{x_j} \underline{\partial \state{u} / \partial \entVar}$ are not symmetric, regardless of the choice of $\vec{v}^m$.
Interestingly, the addition of the GLM technique does not affect the symmetrization of the system. In fact, all entries of the rows and columns that correspond to the GLM method are symmetric.

It is worth noting that we observed that the choice of $\vec{v}^m = \vec{v}^+$ in \eqref{eq:nonconB_quasi} and \eqref{eq:nonconEk_quasi} results in more entries of the matrices $\underline{A}^{x_j} \underline{\partial \state{u} / \partial \entVar}$ becoming symmetric compared to other choices we investigated. We conclude that further modifications to the continuous multi-ion MHD model are necessary to achieve full symmetrization while still ensuring consistency with the single-fluid MHD model.

}
}

Gathering everything, we rewrite the multi-ion GLM-MHD system \eqref{eq:multi-ion_mod} in quasi-conservative form as
\begin{align} \label{eq:multi-ion_mod_div}
    \frac{\partial}{\partial t}
    \begin{pmatrix}
    \rho_k \\ \rho_k \vec{v}_k  \\ E_{k} \\ \vec{B} \\
    \red{\psi}
    \end{pmatrix}
    &+ \Nabla \cdot
    \underbrace{
    \begin{pmatrix} 
    \rho_k \vec{v}_k \\
    \rho_k \vec{v}_k\, \vec{v}_k^{\,T} + 
     p_k \threeMatrix{I}   \\
    \vec{v}_k\left(\hydroEner_k + p_k \right)  \\
    \threeMatrix{0} \\
    \vec{0}
    \end{pmatrix}
    }_{\blocktensor{f}^{\supEuler}}
    +
    \Nabla \cdot 
    \underbrace{
    \begin{pmatrix} 
    \vec{0}  \\
    \threeMatrix{0} \\
    \vec{v}^+_k\,\|\vec{B}\|^2 - \vec{B}\left(\vec{v}^+_k\cdot\vec{B}\right)\\
    \vec{v}^+ \vec{B}^T - \vec{B} (\vec{v}^+)^T \\
    \vec{0}
    \end{pmatrix} 
    }_{\blocktensor{f}^{\supMHD}}
    +
    \red{
    \Nabla \cdot
    \underbrace{
    \begin{pmatrix} 
    \vec{0} \\
    \threeMatrix{0} \\
    c_h \psi \vec{B} \\
    c_h \psi \threeMatrix{I} \\ c_h \vec{B}\\
    \end{pmatrix}
    }_{\blocktensor{f}^{\supGLM}}
    }
    +
    \underbrace{
    \begin{pmatrix}
    0 \\
    r_k \rho_k (\vec{v}^+ - \vec{v}_k) \times \vec{B} \\
    r_k \rho_k \vec{v}_k \cdot (\vec{v}^+ - \vec{v}_k) \times \vec{B}  \\
    \vec{0} \\
    0
    \end{pmatrix}
    }_{\state{g}}
    \nonumber \\
    &+
    \underbrace{
    \begin{pmatrix}
    0 \\
    \frac{r_k \rho_k}{n_e e} \vec{B} \\
    \vec{v}^+ \cdot \vec{B} \\
    \vec{v}^+ \\
    0
    \end{pmatrix}
    (\Nabla \cdot \vec{B})
    }_{\noncon^{\Powell} \coloneqq \stateG{\phi}^{\Powell} (\Nabla \cdot \vec{B}) }
    +
    \begin{pmatrix}
    0 \\
    \frac{r_k \rho_k}{n_e e} \Nabla \cdot \left(\frac{1}{2} \norm{\vec{B}}^2 \mat{I} - \vec{B}\vec{B}^T 
    + p_e \mat{I}\right) \\
    \vec{v}^+_k \cdot \Nabla p_e
    + \vec{B} \cdot 
    \Nabla \cdot \left(\vec{v}^-_k \vec{B}^T - \vec{B} (\vec{v}^-_k)^T  \right) \\
    \vec{0} \\
    0
    \end{pmatrix}
    +
    \red{
    \underbrace{
    \begin{pmatrix}
        \vec{0} \\
        \threeMatrix{0}\\
        \vec{v}^+ \psi \\
        \threeMatrix{0} \\
        \vec{v}^+
    \end{pmatrix}
    \cdot
    \Nabla \psi
    }_{\noncon^{\supGLM}}
    }
    =
    \state{0},
\end{align}
where $\noncon^{\Powell}$ is the multi-ion version of the Godunov-Powell non-conservative term, which is evaluated with the charge-averaged ion velocity and whose momentum components are scaled with the ratio of the charge of each ion species to the total ion charge.

Note that \eqref{eq:multi-ion_mod_div} reduces to the quasi-conservative form of the single-fluid MHD system when one neglects the electron pressure and only one ion species is considered.
We will derive the discretization of the multi-ion MHD system using the quasi-conservative form, such that our discretization is consistent with the quasi-conservative single-fluid discretization \cite{Derigs2018} when $N_i=1$.

\begin{remark}
The \textbf{total} momentum is conserved with \underline{vanishing magnetic field divergence}.
This result can be obtained by summing up all the momentum equations of the individual ion species and taking into account that the definition of the total electron charge \eqref{eq:electronNumberDensity2} implies
\begin{align*}
    \frac{1}{n_e e} \sum_{k=1}^{N_i} \rho_k r_k = 1.
\end{align*}

We obtain the following conservation law for the total momentum:
\begin{align} \label{eq:total_momentum_conservation}
    \bigpartialderiv{(\rho \vec{v})}{t} := \sum_k \bigpartialderiv{(\rho_k \vec{v}_k)}{t}
    =&
    - \Nabla \cdot \blocktensor{f}_{\rho \vec{v}}
    - \left[ \sum_k r_k \rho_k (\vec{v}^+ - \vec{v}_k) \times \vec{B} \right]
    -  (\Nabla \cdot \vec{B}) \vec{B} 
    \nonumber\\
    =&
    - \Nabla \cdot \blocktensor{f}_{\rho \vec{v}}
    - \left[ \vec{v}^+ \sum_k r_k \rho_k - \sum_k r_k \rho_k \vec{v}_k  \right] \times \vec{B}
    -  (\Nabla \cdot \vec{B}) \vec{B} 
    \nonumber\\
    \text{(def. of $v^+$ \eqref{eq:chargeAveragedVel})} \qquad
    =&
    - \Nabla \cdot \blocktensor{f}_{\rho \vec{v}}
    - \left[\left( \sum_k r_k \rho_k \vec{v}_k\right) \frac{1}{n_e e} \sum_k r_k \rho_k - \sum_k r_k \rho_k \vec{v}_k  \right] \times \vec{B}
    -  (\Nabla \cdot \vec{B}) \vec{B} 
    \nonumber\\
    \text{(def. of $n_e e$ \eqref{eq:electronNumberDensity2})} \qquad
    =&
    - \Nabla \cdot \blocktensor{f}_{\rho \vec{v}}
    - \underbrace{
    \left[ \sum_k r_k \rho_k \vec{v}_k - \sum_k r_k \rho_k \vec{v}_k  \right]}_{=0} \times \vec{B}
    - (\Nabla \cdot \vec{B}) \vec{B},
\end{align}
where the total momentum flux is
\begin{equation}
    \blocktensor{f}_{\rho \vec{v}} = 
    \left(\sum_{k=1}^{N_i} \rho_k \vec{v}_k\, \vec{v}_k^{\,T} + 
     p_k \threeMatrix{I}
     \right)
     +
     \frac{1}{2} \norm{\vec{B}}^2 \mat{I} - \vec{B}\vec{B}^T 
    + p_e \mat{I}.
\end{equation}
\end{remark}

\q{c19r2}{

\begin{remark} \label{remark:energy}
The total energy is conserved in the multi-ion MHD system with \underline{vanishing magnetic field divergence} when the \underline{electron pressure \change{gradient} is negligible}.
This result can be obtained by summing up all the energy equations of the individual ion species and accounting for the additional magnetic energy and GLM terms:
\begin{align*}
    \bigpartialderiv{E}{t} = \sum_{k=1}^{N_i} \bigpartialderiv{E_k}{t} - (N_i - 1)
    \left[
    \bigpartialderiv{E_{\rm{mag}}}{t}
    +
    \red{\bigpartialderiv{E^{\supGLM}}{t}}
    \right].
\end{align*}

We obtain the following equation for the evolution of the total energy:
\begin{align} \label{eq:totalEnergy_glm}
    \bigpartialderiv{E}{t} 
    &+ \Nabla \cdot 
    \underbrace{
    \left[ \left( \sum_{k=1}^{N_i} \vec{v}_k (\hydroEner_k + p_k) \right)
    + \vec{v}^+ \norm{B}^2 - \vec{B}(\vec{v}^+\cdot\vec{B})
    + \red{c_h \psi \vec{B}}
    \right]
    }_{\vec{f}_{E}}
    + {\vec{v}^+} \cdot \Nabla p_e
    \nonumber\\
    &+ \left(\Nabla \cdot \vec{B}\right) \left(\vec{v}^+ \cdot \vec{B}\right)
    +
    \red{\psi \vec{v}^+ \cdot \Nabla \psi }
    + \underbrace{\sum_{k=1}^{N_i} r_k \rho_k \vec{v}_k \cdot (\vec{v}^+ - \vec{v}_k) \times \vec{B}}_{= (e n_e) \vec{v}^+ \cdot \left(\vec{v}^+ \times \vec{B}\right) - \sum_{k=1}^{N_i} r_k \rho_k \vec{v}_k \cdot \left(\vec{v}_k \times \vec{B}\right)=0}
    = 0,
\end{align}
\change{where we have used the definition of the charge-averaged velocity \eqref{eq:chargeAveragedVel2}. 
The last term is equal to zero as it contains two scalar products of vectors with vectors obtained using cross products of themselves.

Equation \eqref{eq:totalEnergy_glm} shows that in the case of vanishing magnetic field divergence, $(\Nabla \cdot \vec{B}) = {\color{red}\psi} = 0$, the total energy fulfills a conservation law if the gradient of the electron pressure is negligible.
A non-vanishing gradient of the electron pressure destroys the total energy conservation because the electron energy is not self-consistently described in this particular multi-ion MHD model.
We note that this is also the case for the famous single-fluid MHD model, which typically assumes $\Nabla p_e = 0$.
}
\end{remark}
}

\subsection{One-dimensional Version of the Multi-Ion GLM-MHD System} \label{sec:multi_ion_1D}

To simplify the derivation of numerical discretization schemes for the alternative formulation of the multi-ion MHD system, we write a one-dimensional version of it,
\begin{equation} \label{eq:multi_ion-MHD_1D}
\bigpartialderiv {\mathbf{u}}{t}
+ \bigpartialderiv {\state{f}} {x}
+ \state{g}
+ \underbrace{\stateG{\phi}^{\Powell} \bigpartialderiv{B_1}{x}}_{\noncon^{\Powell}}
+ \underbrace{\stateG{\phi}^{\mathrm{Lor}} \circ \bigpartialderiv{\state{h}^{\rm{Lor}}}{x}}_{\noncon^{\rm{Lor}}}
+ \underbrace{\vec{B} \cdot \bigpartialderiv{\blocktensor{h}^{\rm{multi}}}{x}}_{\noncon^{\rm{multi}}}
+
{\color{red}
\underbrace{
\stateG{\phi}^{\rm{GLM}}\bigpartialderiv{\psi}{x}
}_{\noncon^{\supGLM}}
}
= \state{0},
\end{equation}
where the state variable is $\state{u} = (\rho_k, \rho_k \vec{v}_k, E_k, \vec{B})^T$, as before, and $\circ$ denotes the Hadamard (element-wise) vector multiplication operator.
The advective flux in $x$ is
\begin{equation}
\state{f}(\mathbf{u}) :=
\state{f}^{\supEuler} +\state{f}^{\supMHD} + {\color{red}\state{f}^{\supGLM}}=
\begin{pmatrix} 
\rho_k v_{k,1} \\[0.15cm]
\rho_k v_{k,1}^2 + p_k \\[0.15cm]
\rho_k v_{k,1} v_{k,2} \\[0.15cm]
\rho_k v_{k,1} v_{k,3} \\[0.15cm]
v_{k,1} \left( \frac{1}{2} \rho_k \|\vec{v}_k\|^2 + \frac{\gamma_k p_k}{\gamma_k -1}\right) \\[0.15cm]
0 \\[0.15cm]
0 \\[0.15cm]
0 \\[0.15cm]
0 \\[0.15cm]
\end{pmatrix} +
\begin{pmatrix} 
0 \\[0.15cm]
0 \\[0.15cm]
0 \\[0.15cm]
0 \\[0.15cm]
 v^+_{k,1} \|\vec{B}\|^2 - B_1 \left(\vec{v}^+_k\cdot\vec{B}\right) \\[0.15cm]
0 \\[0.15cm]
v^+_1 B_2 - v^+_2 B_1 \\[0.15cm]
v^+_1 B_3 - v^+_3 B_1 \\[0.15cm]
0 \\[0.15cm]
\end{pmatrix}
+
\color{red}{
c_h
\begin{pmatrix} 
0 \\[0.15cm]
0 \\[0.15cm]
0 \\[0.15cm]
0 \\[0.15cm]    
\psi {B}_1 \\[0.15cm]
\psi  \\[0.15cm]
0 \\[0.15cm]
0 \\[0.15cm]
{B}_1\\[0.15cm]
\end{pmatrix}
}.
\label{eq:advective_fluxes_1D}
\end{equation}

The one-dimensional version of the Godunov-Powell non-conservative term reads
\begin{equation}\label{eq:noncon_Powell}
    \noncon^{\Powell} = 
    \left(\stateG{\phi}^{\Powell} \right)  \bigpartialderiv{B_1}{x}
    =
    \begin{pmatrix}
    0 \\ 
    \frac{r_k \rho_k}{n_e e} B_1 \\[0.15cm]
    \frac{r_k \rho_k}{n_e e} B_2 \\[0.15cm]
    \frac{r_k \rho_k}{n_e e} B_3 \\[0.15cm]
    \vec{v}^+ \cdot \vec{B} \\[0.15cm]
    v^+_{1} \\[0.15cm]
    v^+_{2} \\[0.15cm] 
    v^+_{3} \\[0.15cm]
    0
    \end{pmatrix}
    \bigpartialderiv{B_1}{x}.
\end{equation}

The second non-conservative term is denoted $\noncon^{\rm{Lor}}$ as it contains some of the terms from the Lorentz force for the momentum and the energy equation of each ion species.
The one-dimensional version of this term reads
\begin{equation}\label{eq:noncon_Lor}
\noncon^{\rm{Lor}} = 
\stateG{\phi}^{\mathrm{Lor}} \circ \bigpartialderiv{\state{h}^{\rm{Lor}}}{x} =
\begin{pmatrix} 
0 \\[0.15cm]
\frac{r_k \rho_k}{n_e e} \\[0.15cm]
\frac{r_k \rho_k}{n_e e} \\[0.15cm]
\frac{r_k \rho_k}{n_e e} \\[0.15cm]
v^+_{k,1} \\[0.15cm]
\vec{0}\\[0.15cm]
0
\end{pmatrix}
\circ
\frac{\partial}{\partial x}
\begin{pmatrix} 
0 \\[0.15cm]
\frac{1}{2} \|\vec{B}\|^2 - B_1 B_1 + p_e\\[0.15cm]
- B_1 B_2 \\[0.15cm]
- B_1 B_3 \\[0.15cm]
p_e  \\[0.15cm]
\vec{0}\\[0.15cm]
0 \\
\end{pmatrix}.
\end{equation}

The third non-conservative term $\noncon^{\rm{multi}}$ is denoted as such because it contains terms that are only nonzero in the presence of more than one ion species, $N_i > 1$.
The one-dimensional version of this term reads as
\begin{align}\label{eq:noncon_multi}
    \noncon^{\rm{multi}} =
    \vec{B} \cdot \bigpartialderiv{\blocktensor{h}^{\rm{multi}}}{x} &=
    \vec{B} \cdot \bigpartialderiv{}{x}
    \begin{pmatrix}
    \begin{pmatrix}
    0 \\ \vec{0} \\ 0 \\ \vec{0} \\0
    \end{pmatrix},
    \begin{pmatrix}
    0 \\ \vec{0} \\ v^-_{k,1} B_2 - v^-_{k,2} B_1  \\ \vec{0} \\0
    \end{pmatrix},
    \begin{pmatrix}
    0 \\ \vec{0} \\ v^-_{k,1} B_3 - v^-_{k,3} B_1  \\ \vec{0} \\
    0
    \end{pmatrix}
    \end{pmatrix}
    \nonumber\\
    &=
    B_2 \bigpartialderiv{}{x}
    \begin{pmatrix}
    0 \\ \vec{0} \\ v^-_{k,1} B_2 - v^-_{k,2} B_1  \\ \vec{0} \\0
    \end{pmatrix}
    + B_3 \bigpartialderiv{}{x}
    \begin{pmatrix}
    0 \\ \vec{0} \\ v^-_{k,1} B_3 - v^-_{k,3} B_1  \\ \vec{0}\\
    0
    \end{pmatrix}.
\end{align}

Finally, the one-dimensional non-conservative GLM term reads:
\begin{align}
{\color{red}
    \noncon^{\supGLM}
    =
    \stateG{\phi}^{\rm{GLM}}\bigpartialderiv{\psi}{x}
    =
    \begin{pmatrix}
        0 \\
        \vec{0}\\
         {v}^+_1\psi \\
        \vec{0} \\        
        {v}^+_1\\
    \end{pmatrix}    
    \bigpartialderiv{\psi}{x}.
}
\end{align}

\section{Entropy-Stable Numerical Schemes} \label{sec:discretization}

In this section, we develop both low- and high-order discretizations that ensure entropy conservation and stability for the one-dimensional multi-ion GLM-MHD system presented in Section \ref{sec:multi_ion_1D}. 
The decision to focus on the one-dimensional version stems from the system's inherent complexity, involving numerous state variables, conservative and non-conservative terms. 
This choice enhances the paper's readability, allowing for a clearer exposition of the methodology.
While we specifically address the one-dimensional case, \change{we provide the extension of the main derivations to 2D Cartesian meshes in Appendix \ref{app:2d}.
In the \nameref{sec:Conclusions} Section, we provide some hints on the extension to curvilinear meshes of quadrilaterals/hexahedra and triangles/tetrahedra}.
In Section \ref{sec:results}, we present a numerical validation of the methods of this section when extended to two spatial dimensions.

Our approach to constructing entropy-consistent discretizations begins with the derivation of suitable numerical fluxes for a low-order entropy-conservative Finite Volume (FV) discretization, detailed in Section \ref{sec:fvec}. 
Subsequently, we integrate these derived entropy-conservative (EC) fluxes into a high-order split-form framework, specifically the Legendre–Gauss–Lobatto Discontinuous Galerkin Spectral Element Method (LGL-DGSEM), which is outlined in Section \ref{sec:dges}. 
The result is a high-order scheme that not only upholds the conservation properties for the PDE's state quantities but is also provably entropy conservative. 
To enhance its dissipative and upwinding characteristics, the scheme is complemented with provably entropy-dissipative surface numerical fluxes.

\subsection{Low-Order Entropy-Conservative Finite Volume Discretization} \label{sec:fvec}

In the context of a one-dimensional tessellation of a computational domain with finite volume cells indexed as $i = 1, \ldots, N_{\rm{DOFs}}$, we aim to derive a semi-discrete finite volume discretization of \eqref{eq:multi_ion-MHD_1D} in the form
\begin{equation} \label{eq:fv_1D}
    \Delta x \dot{\state{u}}_i =
    \numfluxb{f}_{(i,i-1)} - \numfluxb{f}_{(i,i+1)}
    +
    \numnonconsD{\Jan}_{(i,i-1)} - \numnonconsD{\Jan}_{(i,i+1)},
\end{equation}
where $\numfluxb{f}_{(i,j)} \coloneqq \numfluxb{f}(\state{u}_i,\state{u}_j)$ is a two-point numerical flux function that approximates the solution to the Riemann problem between degrees of freedom $i$ and $j$. 
This flux function is symmetric (conservative) and consistent, i.e., it satisfies $\numfluxb{f}_{(i,j)} = \numfluxb{f}_{(j,i)}$ and $\numfluxb{f}_{(i,i)} = \state{f}(\state{u}_i)$. 
\q{c20r2}{
Furthermore, $\numnonconsD{\Jan}_{(i,j)} = \numnonconsD{\Jan}(\state{u}_i,\state{u}_j)$ is a generally non-symmetric (non-conservative) but consistent \change{numerical} two-point term. That is,
\[
\numnonconsD{\Jan}_{(i,i)} = \Jan (\state{u}_i),
\]
where $\Jan$ represents a ``non-derivative'' version of the non-conservative terms.
\change{
We adopt the notation of \cite{Derigs2018, Bohm2018, rueda2023entropy} and use the diamond symbol ``$\diamond$'' to indicate that this is a discrete term.}
}

\change{For the one-dimensional multi-ion GLM-MHD system \eqref{eq:multi_ion-MHD_1D}, 
the numerical non-conservative two-point term contains Godunov-Powell, Lorentz, multi-ion, and GLM components,}
\begin{equation}\label{eq:noncons_different_terms}
    {\Jan}^{\diamond}_{(i,j)} = 
    {\Jan}^{\Powell\diamond}_{(i,j)} +
    {\Jan}^{\rm{Lor}\diamond}_{(i,j)} +
    {\Jan}^{\rm{multi}\diamond}_{(i,j)} +
    {\color{red}{\Jan}^{\rm{GLM}\diamond}_{(i,j)}},
\end{equation}
\change{and} the ``non-derivative'' version of the non-conservative terms reads
\begin{equation} \label{eq:nonder_noncons}
    \Jan = 
    \underbrace{
    \stateG{\phi}^{\Powell}B_1 
    }_{\Jan^{\Powell}}
    + \underbrace{
    \stateG{\phi}^{\mathrm{Lor}} \circ \state{h}^{\rm{Lor}}
    }_{\Jan^{\mathrm{Lor}}}
    + 
    \underbrace{
    \vec{B} \cdot \blocktensor{h}^{\rm{multi}}
    }_{\Jan^{\rm{multi}}}
    + {\color{red}
    \underbrace{
        \stateG{\phi}^{\rm{GLM}}\psi
    }_{\Jan^{\mathrm{GLM}}}
    }
\end{equation}
and, in addition, we introduce
\[
\Jan^{\mathrm{\supMHD}} := \Jan^{\Powell} + \Jan^{\mathrm{Lor}} + \Jan^{\rm{multi}}.
\]

The discretization \eqref{eq:fv_1D} is considered entropy-conservative if, when contracted with the entropy variables, it simplifies to
\begin{equation}
    \Delta x \entVar_i^T \dot{\state{u}}_i
    =
    \Delta x \dot{S}
    =
    \numflux{f}^S_{(i,i-1)} - \numflux{f}^S_{(i,i+1)},
\end{equation}
where $\numflux{f}^S_{(i,j)} \coloneqq \numflux{f}^S(\state{u}_i, \state{u}_j)$ represents symmetric and consistent two-point numerical entropy fluxes. Specifically, $\numflux{f}^S_{(i,j)} = \numflux{f}^S_{(j,i)}$ and $\numflux{f}^S_{(i,i)} = f^S(\state{u}_i)$.

We aim to determine numerical two-point fluxes $\numfluxb{f}_{(i,j)}^{\ec}$ and non-conservative terms ${\Jan}^{\diamond, \ec}_{(j,i)}$ such that the entropy production between two degrees of freedom is zero. This corresponds to finding fluxes and terms satisfying a generalization of Tadmor's shuffle condition \cite{tadmor1983entropy,tadmor1986minimum,Tadmor2003}, as discussed in \cite{Rueda-Ramirez2020, Manzanero2020},
\begin{equation} \label{eq:entropy_prod_ec}
    r_{(i,j)} =
    r (\state{u}_i,\state{u}_j) = 
    \jump{\entVar}_{(i,j)}^T 
    \numfluxb{f}_{(i,j)}^{\ec}
    + \entVar^T_{j} {\Jan}^{\diamond, \ec}_{(j,i)}
    - \entVar^T_{i} {\Jan}^{\diamond, \ec}_{(i,j)}
    - \jump{\Psi}_{(i,j)} = 0,
\end{equation}
where $r_{(i,j)}$ is the entropy production between degrees of freedom $i$ and $j$, and $\Psi$ is the so-called entropy (flux) potential.

The entropy potential of our non-conservative system can be defined in an analogous manner as in single-fluid MHD \cite{Derigs2018, Bohm2018, Rueda-Ramirez2020},
\begin{equation}\label{eq:entropy_potential}
\Psi = \entVar^T \left(\state{f} + \Jan \right) - {f}^S 
=
\underbrace{\entVar^T \state{f}^{\supEuler} - {f}^S}_{\coloneqq \Psi^{\supEuler}}
+
\underbrace{\entVar^T \left(\state{f}^{\supMHD} 
+ \Jan^{\supMHD} \right)}_{\coloneqq \Psi^{\supMHD}}
+ {\color{red}\underbrace{\entVar^T\left(
\state{f}^{\supGLM} + \Jan^{\supGLM}
\right)
}_{\coloneqq \Psi^{\supGLM}}},
\end{equation}
which leads to the numerical entropy flux function
\begin{equation} \label{eq:numEntFlux} 
\numflux{f}^S_{(j,k)} = 
\avg{\entVar}_{(j,k)}^T \numfluxb{f}^{\ec}_{(j,k)} 
+ \frac{1}{2} \entVar^T_j \Jan^{\diamond,\ec}_{(j,k)}
+ \frac{1}{2} \entVar^T_{k} \Jan^{\diamond, \ec}_{(k,j)}
- \avg{\Psi}_{(j,k)}.
\end{equation}

Entropy-conserving two-point FV discretizations of the form \eqref{eq:fv_1D} exhibit second-order accuracy in regular grids and first-order accuracy on irregular grids \cite{Fjordholm2012, Winters2016, Derigs2018}.

Since \eqref{eq:entropy_prod_ec} is a scalar equation, there are several possible solutions for the fluxes and non-conservative terms that satisfy entropy conservation.
While the demand for consistency in the fluxes and non-conservative terms significantly narrows down the possibilities, a substantial number of potential solutions remains. To further refine our approach, we introduce an additional constraint: in the presence of a single ion species, the entropy-conservative (EC) fluxes and non-conservative terms must reduce to the fluxes and non-conservative terms presented by \citet{Derigs2018} for the single-fluid MHD equations.

Through meticulous manipulation of the terms in \eqref{eq:entropy_prod_ec}, we derive entropy-conserving fluxes and non-conservative terms for the multi-ion GLM-MHD system, ensuring consistency with the results of \citet{Derigs2018}. 
Our approach involves breaking down the entropy-conserving flux into Euler, MHD and GLM components, 
\begin{equation}
    \numfluxb{f}^{\ec}_{(i,j)} = \numfluxb{f}^{\ec,\supEuler}_{(i,j)}
    + \numfluxb{f}^{\ec, \supMHD}_{(i,j)}
    + {\color{red}  \numfluxb{f}^{\ec, \supGLM}_{(i,j)}},
\end{equation}
and the entropy-conserving two-point non-conservative terms into Godunov-Powell, Lorentz, multi-ion, and GLM components, \change{as in \eqref{eq:noncons_different_terms}}.
We then analyze each part individually, ensuring compliance with the consistency requirements. For a detailed and comprehensive derivation, please refer to Appendix \ref{app:ec_derivation}.

\q{c2_r1}{
The EC flux, which is composed of Euler, MHD and GLM parts, reads as
\begin{equation}\label{eq:ECFlux}
\numfluxb{f}^{\ec}(\state{u}_L,\state{u}_R) =
\begin{pmatrix} 
\rho^{\ln}_{k} \avg{v_{k,1}} \\
\rho^{\ln}_{k} \avg{v_{k,1}}\avg{v_{k,1}} + \overline{p}_k \\ 
\rho^{\ln}_{k} \avg{v_{k,1}} \avg{v_{k,2}} \\
\rho^{\ln}_{k} \avg{v_{k,1}} \avg{v_{k,3}}  \\
f^{\ec, \supEuler}_{E_k} + \blue{\numflux{f}^{\ec,\supMHD}_{E_k}} + \red{2 c_h \avg{\psi} \avg{B_1} - c_h \avg{\psi B_1}}\\
\red{c_h \avg{\psi}} \\
    \blue{\avg{v^+_1}\avg{B_2} - \avg{v^+_2}\avg{B_1}}\\
    \blue{\avg{v^+_1}\avg{B_3} - \avg{v^+_3}\avg{B_1}}\\
    \red{c_h \avg{B_1} }
\end{pmatrix}
\end{equation}
where the mean pressure is defined as
\begin{align} \label{eq:pbar}
\overline{p}_k =& \frac{\avg{\rho_k}}{2\avg{\beta_k}},
\end{align}
and the Euler and MHD energy terms read as
\begin{align}
f^{\ec, \supEuler}_{E_k} = & f_{\rho_k}^{\ec,\supEuler}\bigg[\frac{1}{2 (\gamma_k-1) \beta^{\ln}_{k}} - \frac{1}{2} \left(\avg{v_{k,1}^2} + \avg{v_{k,2}^2} + \avg{v_{k,3}^2}\right) \bigg] \\
&+ f_{\rho_k v_{k,1}}^{\ec,\supEuler} \avg{v_{k,1}} 
+ f_{\rho_k v_{k,2}}^{\ec,\supEuler} \avg{v_{k,2}} 
+ f_{\rho_k v_{k,3}}^{\ec,\supEuler} \avg{v_{k,3}},
\label{eq:EulerEnergyEC}\\
\blue{\numflux{f}^{\ec,\supMHD}_{E_k,(i,j)} =}&
    \blue{\avg{\vec{B}} \cdot \vec{\numflux{f}}^{\ec,\supMHD}_{\vec{B},(i,j)}
    - \frac{1}{2} 
    \avg{v^+_{k,1}  \| \vec{B} \|^2}
    +
    \avg{\vec{v}^+_k  \cdot \vec{B}}_{(i,j)} \avg{B_1}_{(i,j)}
    + \frac{1}{2} 
    \avg{v^+_{k,1} } \avg{\| \vec{B} \|^2}}
    \\
    & 
    \blue{-\avg{\vec{v}^+_k } \cdot  \avg{\vec{B}} \avg{B_1}
    -
    \avg{\vec{B}} \cdot \underbrace{\left( \avg{\vec{B}}
    \avg{v^-_{k,1}} -
    \avg{\vec{v}^-_{k}} \avg{B_1} \right)}_{\vec{h}^{\rm{multi}*}_{E_k,(i,j)}}},
\end{align}
\change{
where $\vec{\numflux{f}}^{\ec,\supMHD}_{\vec{B},(i,j)}$ corresponds to the magnetic field components of the flux given in \eqref{eq:ECFlux}, i.e.,
\begin{equation}
\vec{\numflux{f}}^{\ec,\supMHD}_{\vec{B},(i,j)}
= 
\begin{pmatrix}
\red{c_h \avg{\psi}}, &
    \blue{\avg{v^+_1}\avg{B_2} - \avg{v^+_2}\avg{B_1}}, &
    \blue{\avg{v^+_1}\avg{B_3} - \avg{v^+_3}\avg{B_1}}
\end{pmatrix}^T.
\end{equation}
}}

The black terms come from the Euler discretization, the blue terms come from the MHD discretization, and the red terms come from the GLM discretization.
The four non-conservative two-point terms that provide entropy conservation are the Godunov-Powell term,
\begin{equation}
\label{eq:fv_PhiGP}
    {
    \Jan^{\Powell\diamond, \ec}_{(i,j)} =
    \stateG{\phi}^{\Powell}_i \avg{B_1}_{(i,j)}},
\end{equation}
the Lorentz non-conservative term
\begin{equation}
\label{eq:fv_PhiLor}
{
    \Jan^{\rm{Lor}\diamond, \ec}_{(i,j)} =
    \begin{pmatrix} 
0 \\[0.1cm]
\frac{r_k \rho_k}{n_e e} \\[0.1cm]
\frac{r_k \rho_k}{n_e e} \\[0.1cm]
\frac{r_k \rho_k}{n_e e} \\[0.1cm]
v^+_{k,1} \\[0.1cm]
\vec{0} \\[0.1cm]
0
\end{pmatrix}_i
\circ
\underbrace{
\begin{pmatrix} 
0 \\[0.1cm]
\frac{1}{2} 
\avg{\|\vec{B}\|^2}
- \avg{B_1}^2 + \bar{p}_e\\[0.1cm]
- \avg{B_1} \avg{B_2} \\[0.1cm]
- \avg{B_1} \avg{B_3} \\[0.1cm]
\bar{p}_e \\[0.1cm]
\vec{0} \\[0.1cm]
0 
\end{pmatrix}_{(i,j)}
}_{\state{h}^{\rm{Lor}\diamond, \ec}_{(i,j)}}
},
\end{equation}
with an arbitrary symmetric ``average'' of the electron pressure $\bar{p}_e$, \change{for instance, $\bar{p}_e = \avg{p_e}$ is a valid choice but any other symmetric average can be used},
the ``multi-ion'' term,
\begin{align}
    \Jan^{\rm{multi}\diamond, \ec}_{(i,j)}
    =&
    (B_2)_i
    \begin{pmatrix}
    0 \\ \vec{0} \\ \avg{v^-_{k,1}} \avg{B_2} - \avg{v^-_{k,2}} \avg{B_1}  \\ \vec{0} \\ 0
    \end{pmatrix}
    + (B_3)_i
    \begin{pmatrix}
    0 \\ \vec{0} \\ \avg{v^-_{k,1}} \avg{B_3} - \avg{v^-_{k,3}} \avg{B_1}  \\ \vec{0}\\ 0
    \end{pmatrix}
    \nonumber\\
    =&
    \begin{pmatrix}
    0 \\ \vec{0} \\ 
    \vec{B}_i \cdot \avg{\vec{B}}
    \avg{v^-_{k,1}} -
    \vec{B}_i \cdot
    \avg{\vec{v}^-_{k}} \avg{B_1}  \\ \vec{0}\\ 0
    \end{pmatrix}
    =
        \begin{pmatrix}
        0 \\ \vec{0} \\ 
        \vec{B}_i \cdot \vec{h}^{\rm{multi}*}_{E_k,(i,j)}  \\ \vec{0}\\ 0
        \end{pmatrix},
\end{align}
and the GLM term,
\begin{align} \label{eq:ECGLM}
\red{
        \Jan^{\rm{GLM}\diamond, \ec}_{(i,j)} =
    \begin{pmatrix} 
0 \\
\vec{0} \\
v^+_{1} \psi \\
\vec{0} \\
v^+_{1}
\end{pmatrix}_i
\avg{\psi}_{(i,j)}.
    }
\end{align}

\q{c19r2_2}{
\change{

\begin{remark}\label{remark:momentum_disc}
    The \textbf{total} momentum is conserved with vanishing magnetic field divergence by the finite volume scheme. As in the continuous case, this result can be obtained by summing up all the momentum components for the individual ion species of the equations, as represented in the semi-discrete finite-volume discretization \eqref{eq:fv_1D}.
The condition $\Nabla \cdot \vec{B} = 0$ implies that the Godunov-Powell term \eqref{eq:fv_PhiGP} cancels itself out ($B_1 = \,$constant in 1D), leaving the Lorentz non-conservative term \eqref{eq:fv_PhiLor} as the only non-conservative term for the momentum equations.
Moreover, considering that the definition of the total electron charge \eqref{eq:electronNumberDensity2} implies
\begin{align*}
    \frac{1}{n_e e} \sum_{k=1}^{N_i} \rho_k r_k = 1,
\end{align*}
it is possible to show that the sum of the momentum components of the Lorentz non-conservative terms reduces to a flux form, which is symmetric and consistent.
For the total momentum in any direction $m \in \{1, 2, 3\}$, we have:
\begin{equation}
    \sum_{k=1}^{N_i} \left(\Jan^{\rm{Lor}\diamond, \ec}_{(i,j)}\right)_{\rho_k v_{k,m}} = 
    \frac{1}{n_e e} \sum_{k=1}^{N_i} \rho_k r_k 
    \left( \state{h}^{\rm{Lor}\diamond, \ec}_{(i,j)}\right)_{\rho_k v_{k,m}}
    =
    \left( \state{h}^{\rm{Lor}\diamond, \ec}_{(i,j)}\right)_{\rho_s v_{s,m}} ~~~ \forall s \in \{1, \ldots, N_i\},
\end{equation}
because the momentum components of the symmetric part of the Lorentz non-conservative term, $\left(\state{h}^{\rm{Lor}\diamond, \ec}_{(i,j)}\right)_{\rho_s v_{s,m}}$, are all equal.

\end{remark}

\begin{remark}\label{remark:energy_disc}
    The total energy is conserved by the finite volume scheme under two specific conditions: (i) with vanishing magnetic field divergence and vanishing electron pressure gradient, $\Nabla \cdot \vec{B} = \Nabla p_e = 0$, in the case of one ion species, or (ii) with vanishing magnetic field and vanishing electron pressure gradient, $B_i = \Nabla p_e = 0$ for $i \in \{1, 2, 3\}$, in the case of multiple ion species.

    Condition (i) holds because our FV discretization is consistent with the FV discretization of the single-fluid GLM-MHD system by \citet{Derigs2018}, which is total energy conservative under those conditions. 
    
    Condition (ii) is necessary in the multi-species setting because of the nonlinear (quadratic) dependence of the total energy on the magnetic field $\vec{B}$ and divergence-cleaning field \red{$\psi$},
\begin{equation}\label{eq:total_energy_glm}
    E = 
    \sum_{k=1}^{N_i} E_k - \frac{1}{2} (N_i - 1) \left(\norm{\vec{B}}^2 + \psi^2 \right).
\end{equation}
    The chain rule, which was used to show total energy conservation at the continuous level (Remark \ref{remark:energy}), does not apply at the discrete level. Consequently, it is not possible to write a conservative semi-discrete FV expression for the total energy or to show that the total energy remains constant for arbitrary time-stepping schemes, unless the last term of \eqref{eq:total_energy_glm} vanishes.

    It is important to note that discrete total energy conservation is also a challenge for the discretizations of the multi-ion MHD system described in Section \ref{sec:tothsystem} (without our algebraic manipulation). 
    In that system, the total energy is also not linearly dependent on the state quantities, which complicates achieving energy conservation.

\end{remark}
}
}

\subsection{High-Order Entropy-Stable Discontinuous Galerkin Discretization}\label{sec:dges}

To achieve an entropy-stable LGL-DGSEM discretization of \eqref{eq:multi_ion-MHD_1D}, the simulation domain is partitioned into elements, and all variables are approximated within each element using piece-wise Lagrange interpolating polynomials of degree $N$ on Legendre-Gauss-Lobatto (LGL) nodes, $\ell_k (\xi), \, k=0, \ldots, N$. 
These polynomials are continuous within each element and discontinuous at the element interfaces.

Additionally, \eqref{eq:multi_ion-MHD_1D} is multiplied by an arbitrary polynomial (test function) of degree $N$ and integrated by parts within each element of the mesh. The resulting integrals are numerically evaluated using an LGL quadrature rule with $N+1$ points on a reference element, $\xi \in [-1,1]$, and the volume integrals are replaced by the so-called split-form formulation, yielding the expression \cite{rueda2023entropy, rueda2024flux}
\begin{align} \label{eq:DGSEM}
J \omega_j \dot{\state{u}}_j 
+ &
\underbrace{
\sum_{k=0}^N S_{jk} \left( \state{f}^{*}_{(j,k)} + \numnonconsS{\Jan}_{(j,k)} \right)
}_{\mathrm{Volume \, term}}
- 
\underbrace{
\delta_{j0} \left( \numfluxb{f}_{(0,L)} + \numnonconsD{\Jan}_{(0,L)} \right)
+ \delta_{jN} \left( \numfluxb{f}_{(N,R)} + \numnonconsD{\Jan}_{(N,R)} \right)
}_{\mathrm{Surface \,  term}}
= \state{0},
\end{align}
for each degree of freedom $j$ of each element. 

In \eqref{eq:DGSEM}, $\omega_j$ is the reference-space quadrature weight, $J$ is the geometry mapping Jacobian from reference space to physical space, which is constant within each element in the 1D discretization, $S_{jk}$ are the entries of a skew-symmetric matrix obtained as $\mat{S} = 2 \mat{Q} - \mat{B}$, where $\mat{Q} \in \mathbb{R}^{N+1,N+1}$ is the SBP derivative matrix with entries
$$
Q_{jk}\coloneqq\omega_j D_{jk}=\omega_j \ell'_k(\xi_j), \,\,\,\,\,\, j,k=0, \ldots, N,
$$ 
which are defined in terms of the LGL Lagrange interpolating polynomials,  $\mat{B} := diag(-1, 0,  \ldots, 0, 1) \in \mathbb{R}^{N+1,N+1}$ is the so-called boundary matrix, and $\delta_{ij}$ denotes Kronecker's delta function with node indexes $i$ and $j$.

As in the low-order FV discretization, $\numfluxb{f}_{(i,j)}$ is a symmetric (conservative) and consistent two-point numerical flux function and $\numnonconsD{\Jan}_{(i,j)}$ is a generally non-symmetric but consistent two-point numerical term.
These two terms are now applied at the surface of the DG elements.
The sub-indices $(0,L)$ and $(N,R)$ indicate that the numerical fluxes and non-conservative terms are computed between a boundary node and an outer state (left or right).

Similarly, in the volume term we use the so-called volume numerical flux ${\state{f}}^{*}_{(j,k)}$, a two-point flux function evaluated between nodes $j$ and $k$ that needs to be consistent with the continuous flux and symmetric in its two arguments, and the so-called volume numerical non-conservative term $\numnonconsS{\Jan}_{(j,k)}$, a non-symmetric two-point term evaluated between nodes $j$ and $k$ that is consistent with $\Jan$ \eqref{eq:nonder_noncons}. 

The specific choice of numerical volume fluxes enables the generation of versatile split formulations for the nonlinear PDE terms, aiding in de-aliasing and potentially facilitating provable entropy stability, as demonstrated in works such as \cite{Fisher2013,Fisher2013a,Carpenter2014,Gassner2013,Gassner2016,Renac2019}. 
Notably, it can be shown that opting for ``standard'' averages for the volume numerical fluxes and non-conservative terms, defined as
\begin{equation} \label{eq:std_vol_fluxes}
{\state{f}}^{*}_{(j,k)} = \avg{\state{f}}_{(j,k)},
\quad
\Jan^{\star}_{(j,k)} = \stateG{\phi}_j^{\Powell} \avg{B_1}_{(j,k)}
+
    \stateG{\phi}_j^{\rm{Lor}} \circ \avg{\state{h}^{\rm{Lor}}}_{(j,k)}
    +
    \vec{B}_j \cdot \avg{\blocktensor{h}^{\rm{multi}}}_{(j,k)}
    +
    {\color{red}
        \stateG{\phi}_j^{\rm{GLM}}\avg{\psi}_{(j,k)}
    },
\end{equation}
results in the standard DGSEM discretization of the equations.

It has been demonstrated in various studies, e.g., \cite{rueda2023entropy, Rueda-Ramirez2020, Manzanero2020}, that contracting the LGL-DGSEM discretization expressed in \eqref{eq:DGSEM} with the entropy variables of node $j$ results in a semi-discrete entropy conservation law (when summing over the entire element) if the volume and surface numerical fluxes and non-conservative terms satisfy the generalized Tadmor shuffle condition \eqref{eq:entropy_prod_ec}.
The total entropy production in an element with this choice of numerical fluxes and non-conservative terms reads
\begin{align*}
    \sum_{j=0}^N J \omega_j \entVar_j^T \dot{\state{u}}_j 
&= - \sum_{j=0}^N \entVar_j^T
\left(
\sum_{k=0}^N S_{jk} \left( \state{f}^{*\ec}_{(j,k)} + \Jan^{\star \ec}_{(j,k)} \right)
-
\delta_{j0} \left( \numfluxb{f}^{\ec}_{(0,L)} + {\Jan}^{\diamond, \ec}_{(0,L)} \right)
+
 \delta_{jN} \left( \numfluxb{f}^{\ec}_{(N,R)} + {\Jan}^{\diamond, \ec}_{(N,R)} \right)
 \right)
 \\
&=
\numflux{f}^S_{(0, L)} - \numflux{f}^S_{(N,R)},
\end{align*}
where the numerical entropy fluxes are defined at the element surfaces are defined as in the low-order setting \eqref{eq:numEntFlux}.
The proof relies on the fact that the LGL-DGSEM fulfills the summation by parts property, $\mat{Q} + \mat{Q}^T = {\mat{B}}$.

\subsubsection{Towards an Entropy-Stable Formulation}\label{sec:ES_flux}

A standard approach \cite{Derigs2017, Derigs2018} to obtain entropy-stable LGL-DGSEM discretizations of non-conservative systems consists in using entropy conservative numerical fluxes and non-conservative terms in the volume integral terms, and entropy-stable numerical fluxes and non-conservative terms in the surface integral.
It has been demonstrated in various studies, e.g., \cite{Bohm2018, rueda2023entropy, Rueda-Ramirez2020, Manzanero2020}, that this approach leads to overall entropy dissipation.

The approach of \citet{Derigs2017} to obtain a suitable entropy-stable discretization rewrites the surface numerical fluxes and non-conservative terms as
\begin{align}\label{eq:esflux}
    \numfluxb{f}_{(L,R)} &= \numfluxb{f}_{(L,R)}^{\ec} - \frac{1}{2} \lambda^{\max}_{(L,R)} \matcal{H} \jump{\entVar}_{(L,R)},
    \nonumber\\
    \numnonconsD{\Jan}_{(L,R)} &= \Jan^{\diamond,\ec}_{(L,R)},
\end{align}
where $\lambda^{\max}_{(L,R)}$ is an estimation of the maximum wave speed between $L$ and $R$ and $\matcal{H}$ is a symmetric positive-definite (SPD) matrix.
The additional dissipation term in the surface numerical flux function can be easily shown to be entropy dissipative for any value of $\lambda^{\max}$ if the matrix $\matcal{H}$ is SPD.
Moreover, the dissipation term is consistent with the local Lax-Friedrichs dissipation if the matrix $\matcal{H}$ is consistent with the so-called entropy Jacobian $\underline{\partial \state{u} / \partial \entVar}$.

We follow the procedure outlined by \citet{Derigs2017} with the aim to find a dissipation matrix that is consistent with the local Lax-Friedrichs (Rusanov) dissipation, i.e.,
\begin{equation}\label{eq:consistencyLLF}
    \matcal{H} \jump{\entVar}_{(L,R)} = \jump{\state{u}}_{(L,R)}.
\end{equation}
To do that, we first rewrite the jump of state and entropy variables using the identity
\begin{equation}
    \jump{a\,b} = \jump{a}\avg{b} + \avg{a}\jump{b}.
\end{equation}

\q{c6_r1}{
For the jump of state variables, we obtain
\begin{align*}
\jump{\state{u}} &= 
	\jump{\begin{pmatrix}\rho_k \\ \rho_k v_{k,1} \\ \rho_k v_{k,2} \\ \rho_k v_{k,3} \\ E_k \\ B_1 \\ B_2 \\ B_3 \\ \red{\psi} \end{pmatrix}}
 =
 \change{
 \jump{\begin{pmatrix}\rho_k \\ \rho_k v_{k,1} \\ \rho_k v_{k,2} \\ \rho_k v_{k,3} \\ 
  \frac{p_k}{\gamma_k -1} + \frac{1}{2}\rho_k \left\|\vec{v}_k\right\|^2 + \frac12 \|\vec{B}\|^2 + {\color{red}\frac12 \psi^2}
 \\ B_1 \\ B_2 \\ B_3 \\ \red{\psi} \end{pmatrix}}
 }
    \\
	&=
	\begin{pmatrix}
	\jump{\rho_k} \\ \avg{\rho_k} \jump{v_{k,1}} + \avg{v_{k,1}} \jump{\rho_k} \\ 
    \avg{\rho_k} \jump{v_{k,2}} + \avg{v_{k,2}} \jump{\rho_k}\\ \avg{\rho_k} \jump{v_{k,3}} + \avg{v_{k,3}} \jump{\rho_k} \\ 
    \left(\frac{\avg{\beta_k^{-1}}}{2(\gamma_k-1)} 
    + \frac{1}{2}{\avg{|\vec{v}_k\|^2}}\right)\jump{\rho_k}
    + \avg{\rho_k} \change{\jumpR{\vec{v}_k}^2} 
    - \frac{\avg{\rho_k}}{2\betaavg_k(\gamma_k-1)}\jump{\beta_k} + \sum\limits_{i=1}^{3}\avg{B_i}\jump{B_i} + \red{\avg{\psi} \jump{\psi}} \\ 
    \jump{B_1} \\ \jump{B_2} \\ \jump{B_3} \\
    \red{\jump{\psi}}
	\end{pmatrix},
\end{align*}
with the newly defined \change{quantities}
\begin{equation*}
    \betaavg = 2 \avg{\beta}^2 - \avg{\beta^2}, 
    ~~~~
    \change{\jumpR{\vec{v}_k}^2 = \left(\avg{v_{k,1}}\jump{v_{k,1}} + \avg{v_{k,2}}\jump{v_{k,2}} + \avg{v_{k,3}}\jump{v_{k,3}}\right) }.
\end{equation*}

For the jump of entropy variables, we obtain
\begin{align*}
\jump{\entVar} &=
	\jump{\begin{pmatrix}\frac{\gamma_k - s_k}{\gamma_k - 1}-\beta_k \lVert\vec{v}_k\rVert^2\\2\beta_k  v_{k,1}\\2\beta_k  v_{k,2}\\2\beta_k  v_{k,3}\\-2\beta_k \\2\beta_+ B_1\\2\beta_+ B_2\\2\beta_+ B_3 \\ \red{2\beta_+ \psi}\end{pmatrix}}
	=
	\begin{pmatrix}
	\frac{\jump{\rho_k}}{\rholn}
    +\frac{\jump{\beta_k}}{\betaln(\gamma_k-1)}-\Big(\avg{\|\vec{v}_k\|^2}\Big)\jump{\beta_k}
    -2\avg{\beta_k} \change{\jumpR{\vec{v}_k}^2} 
    \\
	2 \avg{\beta_k}\jump{v_{k,1}} + 2 \avg{v_{k,1}}\jump{\beta_k} \\
	2 \avg{\beta_k}\jump{v_{k,2}} + 2 \avg{v_{k,2}}\jump{\beta_k} \\
	2 \avg{\beta_k}\jump{v_{k,3}} + 2 \avg{v_{k,3}}\jump{\beta_k} \\
	-2 \jump{\beta_k} \\
	2 \avg{\beta_+}\jump{B_1} + 2 \avg{B_1}\jump{\beta_+} \\
	2 \avg{\beta_+}\jump{B_2} + 2 \avg{B_2}\jump{\beta_+} \\
	2 \avg{\beta_+}\jump{B_3} + 2 \avg{B_3}\jump{\beta_+} \\
    \red{2 \avg{\beta_+}\jump{\psi} + 2 \avg{\psi}\jump{\beta_+} }\\
	\end{pmatrix}.
\end{align*}
}

Given the similarities of these quantities with the jumps of state and entropy variables of single-fluid MHD, the derivation of the dissipation matrix follows the steps proposed by \citet{Derigs2017} for the single-fluid MHD case.
The two main differences are: (i) in the multi-ion case, the hydrodynamic quantities generate diagonal blocks in the entropy Jacobian, and (ii) in multi-ion MHD, there are two terms that are inversely proportional to the temperature, $\beta_k$ and $\beta_+ \coloneqq \sum \beta_k$, instead of one.
This latter difference leads to off-diagonal entries in the entropy Jacobian.

As in \cite{Derigs2017}, we are unable to find a symmetric matrix $\matcal{H}$ that fulfills \eqref{eq:consistencyLLF}.
As suggested in \cite{Derigs2017}, this hints that a dissipation operator of the form $\lambda^{\max} \jump{\state{u}}$ is not entropic.
As a remedy, we follow the procedure by \citet{Derigs2017} and modify the consistency condition for the energy equations and approximate the total energy jump of each ion species as
\begin{align*}
    \jump{E_k} \simeq \overline{\jump{E_k}} =&
    \left(\frac{1}{2(\gamma_k-1)\betaln} + \frac{1}{2}\uavg\right)\jump{\rho_k} - \frac{\rholn}{2(\gamma_k-1)}\frac{\jump{\beta_k}}{(\betaln)^2} \\
    &+ \avg{\rho_k}\left(\avg{v_{k,1}}\jump{v_{k,1}} + \avg{v_{k,2}}\jump{v_{k,2}} + \avg{v_{k,3}}\jump{v_{k,3}}\right) + \sum\limits_{i=1}^{3} \avg{B_i}\jump{B_i} + \red{\avg{\psi} \jump{\psi}},
\end{align*}
with
\begin{equation*}
    \uavg = 2 \norm{\avg{\vec{v}_k}}^2 - \avg{\norm{\vec{v}_k}^2}.
\end{equation*}
Note that $\overline{\jump{E_k}}$ behaves asymptotically as $\jump{E_k}$  when $\avg{\beta_k^{-1}} \to \frac{1}{\betaln}$, $\rholn\to \avg{\rho_k}$ and $\betaln \to \betaavg $, i.e. when the jumps in density and $\beta$ approach zero.

The goal is now to obtain an SPD matrix that fulfills
\begin{equation}\label{eq:consistencyLLFrelaxed}
    \hat{\matcal{H}} \jump{\entVar}_{(L,R)} 
    =
    \begin{pmatrix} \jump{\rho_k} \\ \jump{\rho_k v_{k,1}} \\ 
    \jump{\rho_k v_{k,2}} \\ 
    \jump{\rho_k v_{k,3}}\\ 
    \overline{\jump{E_k}} \\ 
    \jump{B_1} \\ 
    \jump{B_2} \\ 
    \jump{B_3} \\ 
    \red{\jump{\psi}} \end{pmatrix}
    \approx \jump{\state{u}}_{(L,R)}.
\end{equation}

We employ identical procedures as detailed in \cite{Derigs2017}, with careful substitution of $\beta_+$ in place of $\beta_k$ where necessary.
In particular, the new variable $\beta_+$ appears multiplying jumps and averages of the components of the magnetic field and the divergence cleaning operator.
Using the identities
\begin{align}
    \jump{\beta_+} = \sum_{k=1}^{N_i} \jump{\beta_k},
    ~~~~
    \avg{\beta_+} = \sum_{k=1}^{N_i} \avg{\beta_k},
\end{align}
which follow from the linearity of the jump and average operators, we derive the subsequent dissipation matrix for the multi-ion GLM-MHD system:
\begin{equation}\label{eq:dissipationMatrix}
    \hat{\matcal{H}} = \begin{pmatrix}
        \matcal{A} & \matcal{B}\\
        \matcal{B}^T & \matcal{C}
    \end{pmatrix},
    \text{ where }
\,\,
    \matcal{A}
    =
    \begin{pmatrix}
        \mat{A}_1 & \mat{A}_{\text{off}} & \mat{A}_{\text{off}} &  \cdots & \mat{A}_{\text{off}} \\
         \mat{A}_{\text{off}} &\mat{A}_2 &  \mat{A}_{\text{off}} &  \cdots & \mat{A}_{\text{off}} \\
         \mat{A}_{\text{off}} &  \mat{A}_{\text{off}} &\mat{A}_3  &   & \vdots \\
          \vdots & \vdots & & \ddots  & \mat{A}_{\text{off}} \\
          \mat{A}_{\text{off}} & \mat{A}_{\text{off}}  &  \cdots & \mat{A}_{\text{off}} &\mat{A}_{N_i} \\
    \end{pmatrix},
    \,\,
    \matcal{B} =
    \begin{pmatrix}
        \mat{B}_{\text{off}} \\ \mat{B}_{\text{off}} \\ \vdots \\ \mat{B}_{\text{off}}  
    \end{pmatrix}, \,\,\text{ and } \,\,
    \matcal{C} = \text{diag}(\tau^+, \tau^+, \tau^+, \tau^+).
\end{equation}
The left upper block $\matcal{A} \in \mathbb{R}^{5N_i \times 5N_i}$ consists of the diagonal blocks:
\begin{equation*}
    \mat{A}_k = 
    \begin{pmatrix}
        \rholn & \rholn\avg{v_{k,1}} & \rholn\avg{v_{k,2}} & \rholn\avg{v_{k,3}} & \Eline \\
\rholn\avg{v_{k,1}} & \rholn\avg{v_{k,1}}^2 + \pavg & \rholn\avg{v_{k,1}}\avg{v_{k,2}} & \rholn\avg{v_{k,1}}\avg{v_{k,3}} & \left(\Eline + \pavg \right) \avg{v_{k,1}}\\
\rholn\avg{v_{k,2}} & \rholn\avg{v_{k,2}}\avg{v_{k,1}} & \rholn\avg{v_{k,2}}^2 + \pavg & \rholn\avg{v_{k,2}}\avg{v_{k,3}} & \left(\Eline + \pavg \right) \avg{v_{k,2}} \\
\rholn\avg{v_{k,3}} & \rholn\avg{v_{k,3}}\avg{v_{k,1}} & \rholn\avg{v_{k,3}}\avg{v_{k,2}} & \rholn\avg{v_{k,3}}^2 + \pavg & \left(\Eline + \pavg \right) \avg{v_{k,3}} \\
\Eline & \left(\Eline + \pavg \right) \avg{v_{k,1}} & \left(\Eline + \pavg \right) \avg{v_{k,2}} & \left(\Eline + \pavg \right) \avg{v_{k,3}} & \hat{\matcal{H}}_{k,(5,5)} &
    \end{pmatrix}
\end{equation*}
where $k\in \{1, \ldots, N_i \}$ and the energy component reads
\begin{equation}\label{eq:H_k_EnergyTerm}
    \hat{\matcal{H}}_{k,(5,5)} = 
    \frac{1}{\rho_k^{\ln} }\left(\frac{(p_k^{*})^2}{\gamma_k - 1} + \overline{E}_k^2 \right) 
              + \overline{p}_k \norm{\avg{\vec{v_k}}}^2
              + \Emagbar,
\end{equation}
the pressure mean $\overline{p}_k$ is defined as in the EC flux, see equation \eqref{eq:pbar}, and the new auxiliary quantities are defined as
\begin{align} \label{eq:Hmat_AuxVars}
    p_k^* := \frac{\rho_k^{\ln}}{2 \beta_k^{\ln}},
    ~~~~
    \tau^+ := \frac{1}{2\avg{\beta_+}},
    ~~~~
    \overline{E}_k := \frac{p_k^*}{\gamma_k - 1} + \frac{1}{2}\rho_k^{\ln}\uavg, 
    \quad
    \Emagbar := \tau^+ \left(\norm{\avg{\vec{B}}}^2 + \red{\avg{\psi}^2}\right).
\end{align}

The off-diagonal blocks $\mat{A}^{\text{off}} \in \mathbb{R}^{5\times 5}$ of $\matcal{A}$ contain a single non-zero element $\mat{A}^{\text{off}}_{5,5}  = \Emagbar$. Finally, the off-diagonal blocks $\matcal{B}$ of $\hat{\matcal{H}}$  comprised of blocks $\mat{B}_{\text{off}} \in \mathbb{R}^{5\times \red{4}}$ that are zero except for the last row:
\[
\mat{B}_{\text{off}}  = \begin{pmatrix}
    0 & \cdots & & 0\\
    \vdots &  & & \vdots\\
    0 &  & & 0\\
    \tau^+ \avg{B_1} & \tau^+ \avg{B_2} & \tau^+ \avg{B_3} & \red{\tau^+ \avg{\psi}}
\end{pmatrix}
\]

\begin{figure}
    \centering
    \begin{tabular}{rl}
        $\hat{\matcal{H}} = \begin{pmatrix}
        \matcal{A} & \matcal{B}\\
        \matcal{B}^T & \matcal{C}
    \end{pmatrix} = $ & 
    \raisebox{-.5\height}{
    \includegraphics[trim=100 0 0 0,clip,width=0.4\textwidth]{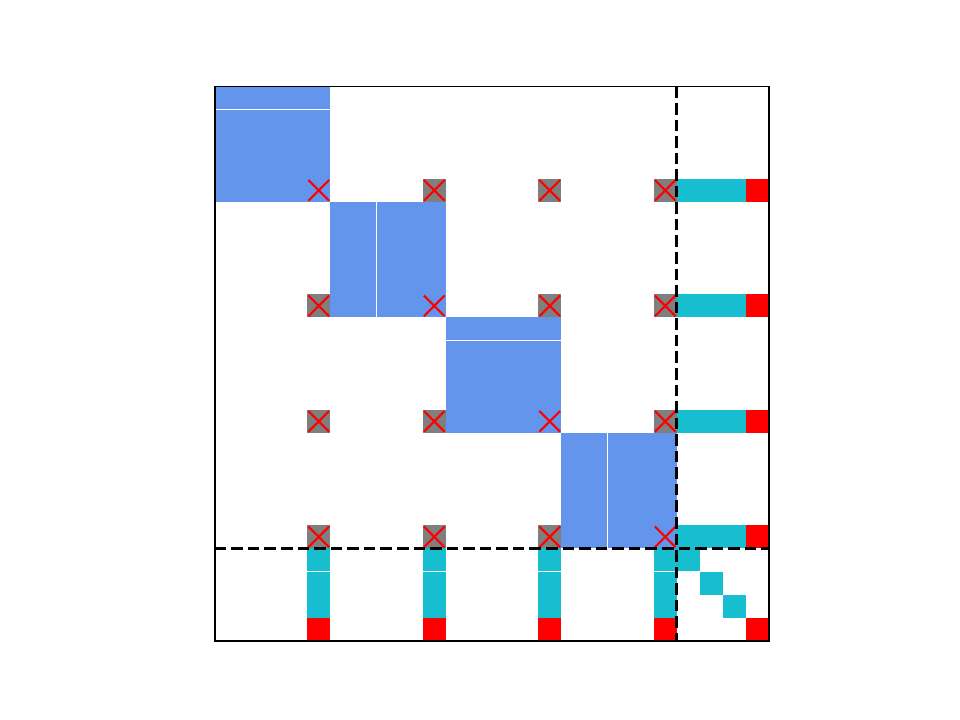}
    }
    \end{tabular}
    \caption{Sparsity pattern of the matrix $\hat{\matcal{H}}$ for the case of four ion species, $N_i=4$.}
    \label{fig:sparsity}
\end{figure}

The sparsity pattern of the matrix $\hat{\matcal{H}}$  is displayed in Figure \ref{fig:sparsity} in the case of four ion species ($N_i=4$). The blue diagonal blocks contain the hydrodynamic quantities of each ion species $\mat{A}_k$.
The gray off-diagonal entries contain $\Emagbar$, i.e., the single entry of $\mat{A}_{\text{off}}$, which indicates how the energy equation for each ion species depends on the entropy variables associated with the energy of all other ion species.
The entries in cyan are associated with the entropy variables pertinent to the magnetic field (non-zero rows of $\mat{B}_{\text{off}}$) or the induction equations (diagonal of $\matcal{C}$).
Lastly, the red entries are dedicated to terms that encompass GLM-terms only, whereas the entries marked with a red cross include terms depending on GLM along with other terms.
The dashed lines indicate the borders between the different blocks $\matcal{A}$, $\matcal{B}$, $\matcal{B}^T$, and $\matcal{C}$.

The matrix $\hat{\matcal{H}}$ is consistent with the dissipation matrix of single-fluid MHD \cite{Derigs2017,Derigs2018} in the limit of one ion species $N_i=1$, and it produces a dissipation operator that is asymptotically consistent with the LLF dissipation operator, as shown in the \nameref{sec:results} section. 
The following lemma establishes the symmetric positive definiteness (SPD) of $\hat{\matcal{H}}$ for an arbitrary number of species $N_i$. Our numerical experiments have consistently demonstrated entropy dissipation when utilizing it, as is detailed in the \nameref{sec:results}, Section \ref{sec:results}.

\begin{lemma}\label{lem:dissipation-is-spd} 
    The matrix $\hat{\matcal{H}}$ defined in~\eqref{eq:dissipationMatrix} is symmetric positive definite (SPD).
\end{lemma}
\begin{proof}

The matrix $\hat{\matcal{H}}$ is evidently symmetric. To show positive definiteness, we employ the Schur complement. Since the densities and pressures of all ion species are positive, i.e. $\rho_k, p_k >0 \, \forall k$, we have that $\beta_k := \rho_k / 2 p_k >0$, $\beta_+ := \sum_k \beta_k >0$ and $\tau^+ := 1 / 2 \avg{\beta_+} >0$. Thus, $\matcal{C}$ is invertible and by Schur's complement, cf.~\cite{boyd2004convex}, $\hat{\matcal{H}}$  is PD if and only if $\matcal{C}$ is PD and its complement $\hat{\matcal{H}} \setminus \matcal{C} = \matcal{A} - \matcal{B}\matcal{C}^{-1}\matcal{B}^T$
is PD. By assumption, the fourfold eigenvalue, $\tau^+$, of $\matcal{C}$ is strictly positive and, hence, $\matcal{C}$ is PD and invertible with $\matcal{C}^{-1} = \text{diag}(1/\tau^+,1/\tau^+,1/\tau^+,1/\tau^+)$. Thus,
\begin{align*}
    (\matcal{B}\matcal{C}^{-1}\matcal{B}^T)_{i,j} &= 
    \begin{cases}
         \frac{1}{\tau^+}\tau^+\left(\avg{\vec{B}}^T, \red{\avg{\psi}}\right)
         \cdot
         \tau^+\begin{pmatrix}
             \avg{\vec{B}}\\ \red{\avg{\psi}}
             \end{pmatrix}
             &\text{ if } i,j \!\!\!\! \mod 5 =0,\\
         0 & \text{otherwise}
    \end{cases}
    \\
    &=
     \begin{cases}
         \Emagbar
             &\text{ if } i,j \!\!\!\!\mod 5 =0,\\
         0 & \text{otherwise}
    \end{cases},
\end{align*}
where we use the definition of $\Emagbar$ in Eq.~\eqref{eq:Hmat_AuxVars}.
The non-zero positions of this matrix are exactly those marked with a red cross in Figure~\ref{fig:sparsity}. Therefore, the complement $\hat{\matcal{H}} \setminus \matcal{C} = \text{diag}(\widetilde{\mat{A}}_1, \ldots, \widetilde{\mat{A}}_{N_i}) $ is a block-diagonal matrix with
\[
\widetilde{\mat{A}}_{k, (i,j)} = \begin{cases}
    \mat{A}_{k, (i,j)} & \text{ if } (i,j) \neq (5,5),\\
    \hat{\matcal{H}}_{k,(5,5)} - \Emagbar & \text{ otherwise}.
\end{cases}
\]

We remark that the Schur complement matrix, $\hat{\matcal{H}} \setminus \matcal{C}$, is nothing but the dissipation matrix of the multi-species Euler system with a similar energy term.
To show its positive definiteness, we can decompose it as $\widetilde{\mat{L}} \, \widetilde{\mat{D}} \,\widetilde{\mat{L}}^T$, where $\widetilde{\mat{D}}$ is a diagonal matrix that only contains positive entries.
The $\widetilde{\mat{L}} \, \widetilde{\mat{D}} \,\widetilde{\mat{L}}^T$ decomposition of each diagonal block  $\widetilde{\mat{A}}_k$ reads:  
\begin{equation}\label{eq:ldlt}
\widetilde{\mat{L}}_k=
\left(\begin{array}{ccccc}
1 & 0 & 0 & 0 & 0 \\
\avg{v_{k,1}} & 1 & 0 & 0 & 0 \\
\avg{v_{k,2}} & 0 & 1 & 0 & 0 \\
\avg{v_{k,3}} & 0 & 0 & 1 & 0 \\
\frac{\overline{E}_k^{2}}{\rholn} & \avg{v_{k,1}} & \avg{v_{k,2}} & \avg{v_{k,3}} & 1
\end{array}\right),
\
\widetilde{\mat{D}}_k = 
\left(\begin{array}{ccccc}
\rholn & 0 & 0 & 0 & 0 \\
 0 & \overline{p}_k & 0 & 0 & 0 \\
 0 & 0 & \overline{p}_k & 0 & 0 \\
 0 & 0 & 0 & \overline{p}_k & 0 \\
 0 & 0 & 0 & 0 &  \widetilde{\mat{A}}_{k,(5,5)}\!\! - \overline{p}_k\norm{\avg{\vec{v_k}}}^2\!\! - \frac{\overline{E}_k^{2}}{\rholn}
\end{array}\right),
\end{equation}
where
\[
\widetilde{\mat{A}}_{k,(5,5)} - \overline{p}_k\norm{\avg{\vec{v_k}}}^2 - \frac{\overline{E}_k^{2}}{\rholn}
 = \frac{1}{\rho_k^{\ln} }\frac{(p_k^{*})^2}{\gamma_k - 1}  > 0
\]
by assumption. Therefore,  $\hat{\matcal{H}} \setminus \matcal{C}$ consists of diagonal blocks each having strictly positive eigenvalues, which implies that $\hat{\matcal{H}}$ is SPD.
\end{proof}

\begin{remark}
By Lemma~\ref{lem:dissipation-is-spd}  and Sylvester's criterion, it is clear that the matrix $\matcal{A}$ defined in \eqref{eq:dissipationMatrix}, i.e the left upper diagonal block of $\hat{\matcal{H}}$, is SPD. However, we would like to state an additional direct proof of this fact since some of techniques used therein seem noteworthy.
\end{remark}
\begin{proof}
      We make the same assumptions as in the proof of Lemma~\ref{lem:dissipation-is-spd}. In addition, let $\Emagbar\neq 0$, otherwise $\matcal{A}$ has a simple block-diagonal structure and it is easy to confirm that it is SPD.  The $\mat{L}\mat{D}\mat{L}^T$ decomposition of each diagonal block  $\mat{A}_k$ is the same as in \eqref{eq:ldlt}, except for the last eigenvalue that reads $\mat{D}_{k, (5,5)} = \hat{\mathcal{H}}_{k,(E_k,E_k)} - \overline{p}_k\norm{\avg{\vec{v_k}}}^2 - \frac{\overline{E}_k^{2}}{\rholn}$.

Let $\matcal{L} = \text{diag}(\mat{L}_1, \ldots, \mat{L}_{N_i})$ be a block diagonal matrix and consider the similarity transformation $\matcal{A}' := \matcal{L}^{-1}\matcal{A}\matcal{L}^{-T}$. The inverse of $\matcal{L}$ exists since its diagonal blocks $\mat{L}_k$ are invertible. The transformation of the off-diagonal blocks of $\matcal{A}$ is obtained as follows:
\begin{align*}
    \mat{L}^{-1}_k\cdot 
    \begin{pmatrix}
    0 & \cdots & 0\\
    \vdots & \ddots & 0\\
    0 & \cdots & \Emagbar
    \end{pmatrix}
    =
    \begin{pmatrix}
    0 & \cdots & 0 &\vdots\\ 
    \vdots & \ddots & \vdots & \state{c}\\
    0 & \cdots & 0 & \vdots
    \end{pmatrix}
    \,\,\Leftrightarrow \,\,
    \mat{L}_k \state{c} = 
    \begin{pmatrix}
        0 \\ \vdots \\ 0 \\ \Emagbar
    \end{pmatrix}
\end{align*}
where $\state{c} \in \mathbb{R}^5$. Forwards substitution yields  $\state{c} = \left(0,\ldots,0, \Emagbar \right)^T$. The same holds for the right-hand-side case due to symmetry and transposition rules of matrix-vector products.

Thus, the off-diagonal blocks $\mat{A}_{\text{off}}$ of $\matcal{A}$ are invariant under this similarity transformation and $\matcal{A}'$ has an almost diagonal structure: 
\[
\matcal{A}'[5k-4,\ldots, 5k ;\, 5k-4,\ldots, 5k] = \mat{D}_k, \quad\matcal{A}'_{(5i, 5j)} = \Emagbar\quad \text{for } i\neq j, \, i,j \in\{1,\ldots, N_i\}
\]
and $\matcal{A}'$ is zero everywhere else. Next, we show that $\matcal{A}'$ is SPD. Let 
$\state{x}\in\mathbb{R}^{5N_i}$, such that $\state{x} \neq \state{0}$, then:
\begin{align*}
    \state{x}^T \matcal{A}' \state{x} 
    = &\sum_{i=1}^{5N_i} x_i^2 \mat{D}_{i\text{ div } 5, (i\!\!\!\!\mod 5,\,i\!\!\!\!\mod 5)} + \Emagbar\sum_{k=1}^{N_i} x_{5j}\sum_{l=1, l\neq k}^{N_i}x_{5l} 
    \\
    = &\sum_{k=1}^{N_i} \sum_{m=1}^{4} x_{5(k-1)+m}^2 \mat{D}_{k, (m,m)}
    + \sum_{k=1}^{N_i} x_{5k}^2 (\mat{D}_{k, (5k,5k)} - \Emagbar) 
    + \sum_{k=1}^{N_i} \Emagbar x_{5k}^2  
    + \Emagbar\sum_{k=1}^{N_i} x_{5j}\sum_{l=1, l\neq k}^{N_i}x_{5l} 
    \\
    = &\sum_{k=1}^{N_i} \sum_{m=1}^{4} x_{5(k-1)+m}^2 \mat{D}_{k, (m,m)}
    + \sum_{k=1}^{N_i} x_{5k}^2 (\mat{D}_{k, (5k,5k)} - \Emagbar) 
    + \Emagbar \left(\sum_{k=1}^{N_i} x_{5j} \right)^2
\end{align*}
Since $\Emagbar >0$ and the eigenvalues of $\matcal{A}$ are non-negative, it is left to show that the second summand is non-negative for any $k\in\{1,\ldots,  N_i\}$:
\begin{equation*}
    \mat{D}_{k, (5k,5k)} - \Emagbar = \hat{\mathcal{H}}_{k,(5,5)}\!\! - \overline{p}_k\norm{\avg{\vec{v_k}}}^2\!\! - \frac{\overline{E}_k^{2}}{\rholn} - \Emagbar 
    \stackrel{\eqref{eq:H_k_EnergyTerm}}{=} 
    \frac{1}{\rho_k^{\ln} }\frac{(p_k^{*})^2}{\gamma_k - 1} > 0,
\end{equation*}
 Thus, $\matcal{A}'$ and its similar transform, $\matcal{A}$, are SPD.
\end{proof}

\subsubsection{Standard Rusanov Numerical Flux}

The flux derived in Section \ref{sec:ES_flux} requires the computation of the entropy Jacobian.
Another (simpler) possibility is to use a standard Rusanov (LLF) numerical flux function at the element interfaces,
\begin{align} \label{eq:llf}
{\numfluxb{f}}_{(L,R)} &= \avg{\state{f}}_{(L,R)} - \frac{1}{2} \lambda^{\max}_{(L,R)} \jump{\state{u}}_{(L,R)}, \\
\Jan^{\diamond}_{(L,R)} &= \stateG{\phi}_L^{\Powell} \avg{B_1}_{(L,R)}
+
    \stateG{\phi}_L^{\rm{Lor}} \circ \avg{\state{h}^{\rm{Lor}}}_{(L,R)}
    +
    \vec{B}_L \cdot \avg{\blocktensor{h}^{\rm{multi}}}_{(L,R)}
    +
    {\color{red}
        \stateG{\phi}_L^{\rm{GLM}}\avg{\psi}_{(L,R)}
    }.
\end{align}
While the standard LLF flux has demonstrated entropy stability for various conservative and non-conservative systems such as the Euler equations \cite{Tadmor2003}, shallow water magnetohydrodynamics equations \cite{duan2021high}, and relativistic hydrodynamics equations \cite{duan2019high}, among many others given an accurate estimation of $\lambda^{\max}$ from above, there is currently no formal proof establishing its entropy stability for the multi-ion GLM-MHD system to the authors' knowledge. 
It is noteworthy, however, that our numerical tests revealed entropy dissipation and robust performance of this flux function for all performed simulations, as elaborated in the \nameref{sec:results} section.

\paragraph{A Note on the Maximum Wave Speed:}

As pointed out by \citet{toth2021challenges}, the explicit formulas for the wave modes and wave speeds of the multi-ion MHD system are unfortunately not known, as the generalization to $N_i$ ion species is not trivial.
In this work, we use a similar approach to the one of the BATS-R-US code \cite{toth2021challenges, glocer2009multifluid}, and estimate the maximum wave speeds with a generalization of the single-fluid wave speeds.
In particular, for the dissipation operator of the numerical flux function we use
\begin{equation}
    \lambda^{\max}_{(L,R)} \approx \max \left[ \left( \max_{k=1, \ldots, N_i} \vec{v}_k \cdot \vec{n}_{(L,R)}\right)_L, \left( \max_{k=1, \ldots, N_i} \vec{v}_k \cdot \vec{n}_{(L,R)}\right)_R \right]
    +
    \max \left[ 
    c_f\left(\state{u}_L, \vec{n}_{(L,R)}\right), 
    c_f\left(\state{u}_R, \vec{n}_{(L,R)}\right) \right],
\end{equation}
where $\vec{n}_{(L,R)}$ is the normal vector at the interface $(L,R)$, and the multi-ion fast magneto-sonic speed is estimated with a generalization of the single-fluid speeds:
\begin{equation}
    c_f^2 \left(\state{u}, \vec{n}\right) = \max_{k=1, \ldots, N_i} 
    \frac{1}{2} \left(a^2 + \|\vec{b}\|^2  +
    \sqrt{(a^2 + \|\vec{b}\|^2)^2 - 4 a^2 (\vec{b}\cdot \Vec{n})^2} \right),
    ~~~
    a^2 = \gamma_k \frac{p_k}{\rho_k},
    ~~~
    \vec{b} = \frac{\vec{B}}{\rho}.
\end{equation}

Similarly, to compute the explicit time-step size, we compute a two-dimensional nodal maximum wave speed as
\begin{equation}
    \label{eq:lambdamax_nodal}
    \lambda^{\max} =  
    \max_{k=1, \ldots, N_i} v_{k,1} + c_f\left(\state{u}, \vec{n}_1\right)
    +
    \max_{k=1, \ldots, N_i} v_{k,2} + c_f\left(\state{u}, \vec{n}_2\right),
\end{equation}
where the unit vectors are defined as $\vec{n}_1=(1,0,0)$ and $\vec{n}_2=(0,1,0)$.

\q{c19r2_3}{
\change{
\begin{remark}
    The conservation properties for the state quantities, total momentum, and total energy of the discontinuous Galerkin discretizations derived in this section are the same as those for the finite volume scheme, as discussed in Remarks \ref{remark:momentum_disc} and \ref{remark:energy_disc}. This similarity arises because the split-form discontinuous Galerkin scheme can be expressed in a flux-differencing finite-volume-like form, using two-point numerical fluxes and non-conservative terms \cite{rueda2024flux}.
\end{remark}
}
}

\section{Numerical Results} \label{sec:results}

In this section, we test the numerical accuracy, entropy consistency, and robustness of our entropy-stable LGL-DGSEM discretization of the multi-ion GLM-MHD equations.
For simplicity, we will focus on two-dimensional test cases on Cartesian meshes, but the methods of this paper can be extended to three-dimensional problems on curvilinear meshes, as shown in \cite{Bohm2018, Rueda-Ramirez2021, rueda2023entropy}.

\q{c25_r2}{
In all cases, the time integration was performed with the explicit fourth-order five-stages Runge-Kutta (RK4-5) scheme of \citet{carpenter1994fourth}.
The time-step size is {estimated} as in \cite{schlottke2021purely},
\begin{equation}\label{eq:cfl}
    \Delta t = \min_{i,j} \left( \frac{\mathrm{CFL}}{(N+1)} \frac{h}{\lambda_{ij}^{\max}} \right),
\end{equation}
where $h$ is the element size, \change{$N$ is the polynomial degree of the DG approximation,} the nodal maximum wave speed $\lambda^{\max}_{ij}$ is computed using \eqref{eq:lambdamax_nodal}, and we use CFL$=0.5$ for all cases, unless otherwise stated.
}

At each time step, we compute the explicit time-step size using \eqref{eq:cfl}, and then compute the divergence cleaning speed, such that the GLM technique does not affect the stability of the method:
\begin{equation}
    \red{
    c_h = \frac{\nu}{\Delta t}
    \underbrace{
    \frac{\mathrm{CFL} \, h}{2(N+1)}
    }_{\Delta t_{c_h}},
    }
\end{equation}
where $\Delta t_{c_h}$ is the explicit time-step size that corresponds to $c_h=1$ and $\nu$ is a scaling constant that we choose as $\nu=0.5$.

All the simulations of this section were run with the two-dimensional Cartesian \texttt{TreeMesh} solver of the open-source framework \texttt{Trixi.jl} \cite{schlottkelakemper2020trixi, ranocha2022adaptive, schlottke2021purely}.

\subsection{Convergence Test} \label{sec:convergence}

To test the convergence properties of the new schemes, we solve a manufactured solution test case for a two-species collisionless plasma.
We assume an exact solution to the ideal multi-ion GLM-MHD system of the form
\begin{align}\label{eq:mansol}
    \rho_1 (\vec{x}, t) &= \chi_1, 
    &
    \rho_2  (\vec{x}, t) &= \chi_2, \nonumber\\
    \rho_1 v_{1,1} (\vec{x}, t) &= \chi_1, &
    \rho_2 v_{2,1} (\vec{x}, t) &= \chi_2,\nonumber\\
    \rho_1 v_{1,2} (\vec{x}, t) &= \chi_1,&
    \rho_2 v_{2,2} (\vec{x}, t) &= \chi_2,\nonumber\\
    \rho_1 v_{1,3} (\vec{x}, t) &= 0.1 \chi_1,&
    \rho_2 v_{2,3} (\vec{x}, t) &= 0.1 \chi_2,\nonumber\\
    E_1 (\vec{x}, t) &= 2 \chi_1^2+\chi_1, &
    E_2 (\vec{x}, t) &= 2 \chi_2^2+\chi_2,\nonumber\\
    B_1 (\vec{x}, t) &= 0.25 \chi, &
    B_2 (\vec{x}, t) &= -0.25 \chi,\nonumber\\
    B_3 (\vec{x}, t) &= 0.1 \chi, &
    \red{\psi(\vec{x}, t)} &\red{=0,}
\end{align}
with the auxiliary variables
\begin{align*}
    \chi &= \underbrace{0.1 \sin(\pi (x + y - t))}_{:= \chi_0} + 2,\\
    \chi_1 &= 0.04 \sin(\pi (x + y - t)) + 1,\\
    \chi_2 &= \chi-\chi_1.
\end{align*}
To consider important multi-ion effects, we use different heat capacity ratios for the two ion species, $\gamma_1 = 2$ and $\gamma_2 = 4$, and also different charge-to-mass rations, $r_1=2$ and $r_2=1$.
Moreover, we use a non-trivial electron pressure to test the convergence properties of our scheme using all the terms of the equation.
In particular, we use \cite{Toth2010}
\begin{equation}\label{eq:pe_alpha}
    p_e = \alpha \sum_{k=1}^{N_i} p_k,
\end{equation}
with $\alpha \ne 0$.
Equation \eqref{eq:pe_alpha} is equivalent to taking the electron pressure gradient term in the
momentum and energy equations to be equal to a fraction of the total ion pressure
gradient.
For the numerical experiments of this section, we use $\alpha = 0.2$, but we would like to remark that the choice of $\alpha$ depends on the application. 
For instance, \citet{glocer2009multifluid} used 
$\alpha = 0.2$ for simulations of the interaction of the solar wind with Earth's magnetosphere, whereas \citet{najib2011three} used a value of
$\alpha = 1$ to simulate the interaction of the solar wind with Mars.

We insert \eqref{eq:mansol} into \eqref{eq:multi-ion_mod_div} and use the symbolic algebra tool \texttt{Maxima} \cite{maxima} to obtain the source term:
\begin{align} \label{eq:mansol_source}
\state{s} =
    \begin{pmatrix}
    \frac{2 \, \chi_x}{5} \\
\frac{38055 \, \chi_x \, \chi_0^2 + 185541 \, \chi_x \, \chi_0 + 220190 \, \chi_x}{35000 \, \chi_0 + 75000} \\
\frac{38055 \, \chi_x \, \chi_0^2 + 185541 \, \chi_x \, \chi_0 + 220190 \, \chi_x}{35000 \, \chi_0 + 75000} \\
\frac{\chi_x}{25} \\
\frac{1835811702576186755 \, \chi_x \, \chi_0^2 + 8592627463681183181 \, \chi_x \, \chi_0 + 9884050459977240490 \, \chi_x}{652252660543767500 \, \chi_0 + 1397684272593787500} \\
\frac{3 \, \chi_x}{5} \\
\frac{76155 \, \chi_x \, \chi_0^2 + 295306 \, \chi_x \, \chi_0 + 284435 \, \chi_x}{17500 \, \chi_0 + 37500} \\
\frac{76155 \, \chi_x \, \chi_0^2 + 295306 \, \chi_x \, \chi_0 + 284435 \, \chi_x}{17500 \, \chi_0 + 37500} \\
\frac{3 \, \chi_x}{50} \\
\frac{88755 \, \chi_x \, \chi_0^2 + 338056 \, \chi_x \, \chi_0 + 318185 \, \chi_x}{8750 \, \chi_0 + 18750} \\
\frac{ \chi_x}{4} \\
\frac{ -\chi_x}{4} \\
\frac{ \chi_x}{10} \\
\red{0}
    \end{pmatrix}.
\end{align}

Equation \eqref{eq:mansol_source} shows that the source term required to derive the manufactured solution  \eqref{eq:mansol} is highly intricate. 
This complexity arises from the intricate interdependencies among ions and the nonlinearities inherent in the ideal multi-ion GLM-MHD system.

We solve the multi-ion GLM-MHD system with the initial condition given by \eqref{eq:mansol} and the source term \eqref{eq:mansol_source} in the domain $\Omega = [-1,1]^2$ with periodic boundary conditions, the final time $t_f = 1$, and three different solvers: 
\begin{enumerate}
    \item[(i)] \textbf{EC:} A split-form entropy-conservative LGL-DGSEM that uses the EC fluxes and non-conservative terms \eqref{eq:ECFlux}-\eqref{eq:ECGLM} in both the volume and surface numerical fluxes and terms of \eqref{eq:DGSEM},
    \item[(ii)] \textbf{ES:} A split-form provably entropy-stable LGL-DGSEM that uses the EC fluxes and non-conservative terms \eqref{eq:ECFlux}-\eqref{eq:ECGLM} in the volume numerical fluxes and non-conservative terms, and the entropy-dissipative LLF-like fluxes and non-conservative term \eqref{eq:esflux} in the surface numerical fluxes and terms of \eqref{eq:DGSEM}, and 
    \item[(iii)] \textbf{EC+LLF:} A dissipative split-form LGL-DGSEM that uses the EC fluxes and non-conservative terms \eqref{eq:ECFlux}-\eqref{eq:ECGLM} in the volume numerical fluxes and non-conservative terms, and the standard LLF solver \eqref{eq:llf} in the surface numerical fluxes and terms of \eqref{eq:DGSEM}.
\end{enumerate}

Tables \ref{tab:eoc_ec_1} to \ref{tab:eoc_ec_4} show the errors and experimental orders of convergence (EOCs) obtained with the \textbf{EC} LGL-DGSEM solver (i) for the manufactured solution test at time $t_f=1$.
We observe high-order convergence, but also an even-odd effect, in which even polynomial degrees exhibit an optimal convergence order of approximately $\mathcal{O}(N+1)$ and odd polynomial degrees exhibit a sub-optimal convergence order of approximately $\mathcal{O}(N)$.
This effect has been observed for EC discretizations of other equations, e.g. \cite{hindenlang2020order, wintermeyer2017entropy}.

Tables \ref{tab:eoc_es_1} to \ref{tab:eoc_es_4} show the errors and experimental orders of convergence (EOCs) obtained with the \textbf{ES} (ii) and \textbf{EC+LLF} (iii) solvers for the manufactured solution test at time $t_f=1$.
The experimental orders of convergence and the errors are identical for the solver (iii) and the solver (ii) for the precision shown.

With solvers (ii) and (iii), we observe high-order convergence without any even-odd effect.
All polynomial degrees exhibit an optimal convergence order of approximately $\mathcal{O}(N+1)$.


\begin{table}[]
    \centering
    \caption{$L_2$ errors and EOCs for the convergence test with the \textbf{EC} scheme (i) for $N=2$ and different numbers of elements per direction $N_e$.}
    \label{tab:eoc_ec_1}
    \resizebox{\columnwidth}{!}{
    \begin{tabular}{ccccccccccccccc}
\hline
$N_e$ &
$\norm{\epsilon_{B_1}}$ & EOC &
$\norm{\epsilon_{B_2}}$ & EOC &
$\norm{\epsilon_{B_3}}$ & EOC &
$\norm{\epsilon_{\rho_1}}$ & EOC &
$\norm{\epsilon_{\rho_1 v_{1,1}}}$ & EOC &
$\norm{\epsilon_{\rho_1 v_{1,2}}}$ & EOC &
$\norm{\epsilon_{\rho_1 v_{1,3}}}$ & EOC 
\\
\hline
$16$ & $1.39 \times 10^{-5}$ & $-$ & $1.39 \times 10^{-5}$ & $-$ & $4.17 \times 10^{-6}$ & $-$ & $2.23 \times 10^{-5}$ & $-$ & $3.55 \times 10^{-5}$ & $-$ & $3.55 \times 10^{-5}$ & $-$ & $6.20 \times 10^{-6}$ & $-$ \\
$32$ & $1.64 \times 10^{-6}$ & $3.09$ & $1.64 \times 10^{-6}$ & $3.08$ & $7.18 \times 10^{-7}$ & $2.54$ & $3.72 \times 10^{-6}$ & $2.58$ & $4.54 \times 10^{-6}$ & $2.97$ & $4.55 \times 10^{-6}$ & $2.96$ & $4.87 \times 10^{-7}$ & $3.67$ \\
$64$ & $1.76 \times 10^{-7}$ & $3.22$ & $1.76 \times 10^{-7}$ & $3.22$ & $8.47 \times 10^{-8}$ & $3.08$ & $3.38 \times 10^{-7}$ & $3.46$ & $4.01 \times 10^{-7}$ & $3.50$ & $4.01 \times 10^{-7}$ & $3.50$ & $5.61 \times 10^{-8}$ & $3.12$ \\
$128$ & $3.91 \times 10^{-8}$ & $2.17$ & $3.91 \times 10^{-8}$ & $2.17$ & $1.23 \times 10^{-8}$ & $2.78$ & $4.88 \times 10^{-8}$ & $2.79$ & $7.53 \times 10^{-8}$ & $2.41$ & $7.53 \times 10^{-8}$ & $2.41$ & $6.70 \times 10^{-9}$ & $3.07$ \\
mean & & $2.83$ & & $2.82$ & & $2.80$ & & $2.94$ & & $2.96$ & & $2.96$ & & $3.29$ \\
\hline
$N_e$ &
$\norm{\epsilon_{E_1}}$ & EOC &
$\norm{\epsilon_{\rho_2}}$ & EOC &
$\norm{\epsilon_{\rho_2 v_{2,1}}}$ & EOC &
$\norm{\epsilon_{\rho_2 v_{2,2}}}$ & EOC &
$\norm{\epsilon_{\rho_2 v_{2,3}}}$ & EOC &
$\norm{\epsilon_{E_2}}$ & EOC &
$\norm{\epsilon_{\psi}}$ & EOC 
\\
\hline
$16$ & $1.10 \times 10^{-4}$ & $-$ & $3.12 \times 10^{-5}$ & $-$ & $7.37 \times 10^{-5}$ & $-$ & $7.37 \times 10^{-5}$ & $-$ & $1.58 \times 10^{-5}$ & $-$ & $1.79 \times 10^{-4}$ & $-$ & $5.56 \times 10^{-6}$ & $-$ \\
$32$ & $1.79 \times 10^{-5}$ & $2.62$ & $3.71 \times 10^{-6}$ & $3.07$ & $6.11 \times 10^{-6}$ & $3.59$ & $6.10 \times 10^{-6}$ & $3.59$ & $8.59 \times 10^{-7}$ & $4.21$ & $1.48 \times 10^{-5}$ & $3.59$ & $1.51 \times 10^{-6}$ & $1.88$ \\
$64$ & $1.54 \times 10^{-6}$ & $3.54$ & $4.58 \times 10^{-7}$ & $3.02$ & $7.02 \times 10^{-7}$ & $3.12$ & $7.02 \times 10^{-7}$ & $3.12$ & $5.40 \times 10^{-8}$ & $3.99$ & $1.97 \times 10^{-6}$ & $2.92$ & $1.79 \times 10^{-7}$ & $3.08$ \\
$128$ & $2.81 \times 10^{-7}$ & $2.46$ & $6.45 \times 10^{-8}$ & $2.83$ & $9.76 \times 10^{-8}$ & $2.85$ & $9.76 \times 10^{-8}$ & $2.85$ & $5.88 \times 10^{-9}$ & $3.20$ & $3.49 \times 10^{-7}$ & $2.50$ & $9.87 \times 10^{-9}$ & $4.18$ \\
mean & & $2.87$ & & $2.97$ & & $3.19$ & & $3.19$ & & $3.80$ & & $3.00$ & & $3.05$ \\
\hline
\end{tabular}
}
\end{table}

\begin{table}[]
    \centering
    \caption{$L_2$ errors and EOCs for the convergence test with the \textbf{EC} scheme (i) for $N=3$ and different numbers of elements per direction $N_e$.}
    \resizebox{\columnwidth}{!}{
    \begin{tabular}{ccccccccccccccc}
\hline
$N_e$ &
$\norm{\epsilon_{B_1}}$ & EOC &
$\norm{\epsilon_{B_2}}$ & EOC &
$\norm{\epsilon_{B_3}}$ & EOC &
$\norm{\epsilon_{\rho_1}}$ & EOC &
$\norm{\epsilon_{\rho_1 v_{1,1}}}$ & EOC &
$\norm{\epsilon_{\rho_1 v_{1,2}}}$ & EOC &
$\norm{\epsilon_{\rho_1 v_{1,3}}}$ & EOC 
\\
\hline
$8$ & $1.09 \times 10^{-4}$ & $-$ & $1.09 \times 10^{-4}$ & $-$ & $1.42 \times 10^{-5}$ & $-$ & $4.19 \times 10^{-5}$ & $-$ & $9.53 \times 10^{-5}$ & $-$ & $9.53 \times 10^{-5}$ & $-$ & $2.28 \times 10^{-5}$ & $-$ \\
$16$ & $1.17 \times 10^{-5}$ & $3.22$ & $1.17 \times 10^{-5}$ & $3.22$ & $1.43 \times 10^{-6}$ & $3.30$ & $4.38 \times 10^{-6}$ & $3.26$ & $1.02 \times 10^{-5}$ & $3.22$ & $1.02 \times 10^{-5}$ & $3.23$ & $2.77 \times 10^{-6}$ & $3.05$ \\
$32$ & $1.41 \times 10^{-6}$ & $3.06$ & $1.41 \times 10^{-6}$ & $3.06$ & $1.70 \times 10^{-7}$ & $3.08$ & $5.12 \times 10^{-7}$ & $3.09$ & $1.22 \times 10^{-6}$ & $3.07$ & $1.21 \times 10^{-6}$ & $3.07$ & $3.44 \times 10^{-7}$ & $3.01$ \\
$64$ & $1.74 \times 10^{-7}$ & $3.01$ & $1.74 \times 10^{-7}$ & $3.01$ & $2.09 \times 10^{-8}$ & $3.02$ & $6.30 \times 10^{-8}$ & $3.02$ & $1.50 \times 10^{-7}$ & $3.02$ & $1.50 \times 10^{-7}$ & $3.02$ & $4.29 \times 10^{-8}$ & $3.00$ \\
mean & & $3.10$ & & $3.10$ & & $3.13$ & & $3.12$ & & $3.10$ & & $3.11$ & & $3.02$ \\
\hline
$N_e$ &
$\norm{\epsilon_{E_1}}$ & EOC &
$\norm{\epsilon_{\rho_2}}$ & EOC &
$\norm{\epsilon_{\rho_2 v_{2,1}}}$ & EOC &
$\norm{\epsilon_{\rho_2 v_{2,2}}}$ & EOC &
$\norm{\epsilon_{\rho_2 v_{2,3}}}$ & EOC &
$\norm{\epsilon_{E_2}}$ & EOC &
$\norm{\epsilon_{\psi}}$ & EOC 
\\
\hline
$8$ & $2.53 \times 10^{-4}$ & $-$ & $5.44 \times 10^{-5}$ & $-$ & $1.96 \times 10^{-4}$ & $-$ & $1.96 \times 10^{-4}$ & $-$ & $7.03 \times 10^{-5}$ & $-$ & $5.36 \times 10^{-4}$ & $-$ & $1.10 \times 10^{-5}$ & $-$ \\
$16$ & $2.74 \times 10^{-5}$ & $3.20$ & $4.61 \times 10^{-6}$ & $3.56$ & $2.17 \times 10^{-5}$ & $3.17$ & $2.17 \times 10^{-5}$ & $3.17$ & $8.08 \times 10^{-6}$ & $3.12$ & $5.12 \times 10^{-5}$ & $3.39$ & $1.92 \times 10^{-7}$ & $5.84$ \\
$32$ & $3.27 \times 10^{-6}$ & $3.07$ & $5.30 \times 10^{-7}$ & $3.12$ & $2.60 \times 10^{-6}$ & $3.06$ & $2.61 \times 10^{-6}$ & $3.06$ & $9.90 \times 10^{-7}$ & $3.03$ & $6.04 \times 10^{-6}$ & $3.09$ & $1.56 \times 10^{-8}$ & $3.63$ \\
$64$ & $4.04 \times 10^{-7}$ & $3.02$ & $6.52 \times 10^{-8}$ & $3.02$ & $3.21 \times 10^{-7}$ & $3.01$ & $3.23 \times 10^{-7}$ & $3.01$ & $1.23 \times 10^{-7}$ & $3.01$ & $7.45 \times 10^{-7}$ & $3.02$ & $3.10 \times 10^{-9}$ & $2.33$ \\
mean & & $3.10$ & & $3.23$ & & $3.08$ & & $3.08$ & & $3.05$ & & $3.17$ & & $3.93$ \\
\hline
\end{tabular}
}
\end{table}

\begin{table}[]
    \centering
    \caption{$L_2$ errors and EOCs for the convergence test with the \textbf{EC} scheme (i) for $N=4$ and different numbers of elements per direction $N_e$.}
    \resizebox{\columnwidth}{!}{
    \begin{tabular}{ccccccccccccccc}
\hline
$N_e$ &
$\norm{\epsilon_{B_1}}$ & EOC &
$\norm{\epsilon_{B_2}}$ & EOC &
$\norm{\epsilon_{B_3}}$ & EOC &
$\norm{\epsilon_{\rho_1}}$ & EOC &
$\norm{\epsilon_{\rho_1 v_{1,1}}}$ & EOC &
$\norm{\epsilon_{\rho_1 v_{1,2}}}$ & EOC &
$\norm{\epsilon_{\rho_1 v_{1,3}}}$ & EOC 
\\
\hline
$8$ & $4.97 \times 10^{-7}$ & $-$ & $4.97 \times 10^{-7}$ & $-$ & $5.78 \times 10^{-8}$ & $-$ & $2.61 \times 10^{-7}$ & $-$ & $4.87 \times 10^{-7}$ & $-$ & $4.87 \times 10^{-7}$ & $-$ & $1.24 \times 10^{-7}$ & $-$ \\
$16$ & $6.77 \times 10^{-9}$ & $6.20$ & $6.77 \times 10^{-9}$ & $6.20$ & $1.97 \times 10^{-9}$ & $4.88$ & $7.79 \times 10^{-9}$ & $5.07$ & $1.35 \times 10^{-8}$ & $5.17$ & $1.35 \times 10^{-8}$ & $5.17$ & $1.97 \times 10^{-9}$ & $5.98$ \\
$32$ & $1.61 \times 10^{-10}$ & $5.40$ & $1.61 \times 10^{-10}$ & $5.40$ & $6.95 \times 10^{-11}$ & $4.82$ & $3.35 \times 10^{-10}$ & $4.54$ & $4.20 \times 10^{-10}$ & $5.00$ & $4.20 \times 10^{-10}$ & $5.00$ & $5.31 \times 10^{-11}$ & $5.21$ \\
$64$ & $5.04 \times 10^{-12}$ & $4.99$ & $5.04 \times 10^{-12}$ & $4.99$ & $1.92 \times 10^{-12}$ & $5.18$ & $1.02 \times 10^{-11}$ & $5.04$ & $1.44 \times 10^{-11}$ & $4.87$ & $1.44 \times 10^{-11}$ & $4.86$ & $1.16 \times 10^{-12}$ & $5.52$ \\
mean & & $5.53$ & & $5.53$ & & $4.96$ & & $4.88$ & & $5.01$ & & $5.01$ & & $5.57$ \\
\hline
$N_e$ &
$\norm{\epsilon_{E_1}}$ & EOC &
$\norm{\epsilon_{\rho_2}}$ & EOC &
$\norm{\epsilon_{\rho_2 v_{2,1}}}$ & EOC &
$\norm{\epsilon_{\rho_2 v_{2,2}}}$ & EOC &
$\norm{\epsilon_{\rho_2 v_{2,3}}}$ & EOC &
$\norm{\epsilon_{E_2}}$ & EOC &
$\norm{\epsilon_{\psi}}$ & EOC 
\\
\hline
$8$ & $1.15 \times 10^{-6}$ & $-$ & $3.64 \times 10^{-7}$ & $-$ & $6.03 \times 10^{-7}$ & $-$ & $6.03 \times 10^{-7}$ & $-$ & $1.98 \times 10^{-7}$ & $-$ & $2.08 \times 10^{-6}$ & $-$ & $9.89 \times 10^{-8}$ & $-$ \\
$16$ & $2.93 \times 10^{-8}$ & $5.30$ & $1.23 \times 10^{-8}$ & $4.88$ & $2.26 \times 10^{-8}$ & $4.74$ & $2.26 \times 10^{-8}$ & $4.74$ & $3.97 \times 10^{-9}$ & $5.64$ & $5.72 \times 10^{-8}$ & $5.18$ & $6.05 \times 10^{-9}$ & $4.03$ \\
$32$ & $1.56 \times 10^{-9}$ & $4.23$ & $3.84 \times 10^{-10}$ & $5.01$ & $7.07 \times 10^{-10}$ & $5.00$ & $7.07 \times 10^{-10}$ & $5.00$ & $4.89 \times 10^{-11}$ & $6.34$ & $2.17 \times 10^{-9}$ & $4.72$ & $1.20 \times 10^{-10}$ & $5.66$ \\
$64$ & $4.96 \times 10^{-11}$ & $4.97$ & $1.59 \times 10^{-11}$ & $4.59$ & $2.29 \times 10^{-11}$ & $4.95$ & $2.29 \times 10^{-11}$ & $4.95$ & $1.98 \times 10^{-12}$ & $4.63$ & $7.23 \times 10^{-11}$ & $4.91$ & $4.27 \times 10^{-12}$ & $4.81$ \\
mean & & $4.83$ & & $4.83$ & & $4.90$ & & $4.90$ & & $5.54$ & & $4.94$ & & $4.83$ \\
\hline
\end{tabular}
}
\end{table}

\begin{table}[]
    \centering
    \caption{$L_2$ errors and EOCs for the convergence test with the \textbf{EC} scheme (i) for $N=5$ and different numbers of elements per direction $N_e$.}
    \label{tab:eoc_ec_4}
    \resizebox{\columnwidth}{!}{
    \begin{tabular}{ccccccccccccccc}
\hline
$N_e$ &
$\norm{\epsilon_{B_1}}$ & EOC &
$\norm{\epsilon_{B_2}}$ & EOC &
$\norm{\epsilon_{B_3}}$ & EOC &
$\norm{\epsilon_{\rho_1}}$ & EOC &
$\norm{\epsilon_{\rho_1 v_{1,1}}}$ & EOC &
$\norm{\epsilon_{\rho_1 v_{1,2}}}$ & EOC &
$\norm{\epsilon_{\rho_1 v_{1,3}}}$ & EOC 
\\
\hline
$4$ & $1.47 \times 10^{-6}$ & $-$ & $1.47 \times 10^{-6}$ & $-$ & $2.28 \times 10^{-7}$ & $-$ & $1.26 \times 10^{-6}$ & $-$ & $2.32 \times 10^{-6}$ & $-$ & $2.29 \times 10^{-6}$ & $-$ & $7.40 \times 10^{-7}$ & $-$ \\
$8$ & $4.47 \times 10^{-8}$ & $5.04$ & $4.48 \times 10^{-8}$ & $5.03$ & $8.50 \times 10^{-9}$ & $4.75$ & $4.23 \times 10^{-8}$ & $4.90$ & $1.02 \times 10^{-7}$ & $4.51$ & $1.01 \times 10^{-7}$ & $4.51$ & $1.52 \times 10^{-8}$ & $5.61$ \\
$16$ & $1.41 \times 10^{-9}$ & $4.99$ & $1.41 \times 10^{-9}$ & $4.99$ & $3.69 \times 10^{-10}$ & $4.53$ & $1.56 \times 10^{-9}$ & $4.76$ & $3.23 \times 10^{-9}$ & $4.97$ & $3.21 \times 10^{-9}$ & $4.97$ & $4.54 \times 10^{-10}$ & $5.06$ \\
$32$ & $4.43 \times 10^{-11}$ & $4.99$ & $4.43 \times 10^{-11}$ & $4.99$ & $1.17 \times 10^{-11}$ & $4.98$ & $4.83 \times 10^{-11}$ & $5.01$ & $9.87 \times 10^{-11}$ & $5.03$ & $9.79 \times 10^{-11}$ & $5.03$ & $1.43 \times 10^{-11}$ & $4.99$ \\
mean & & $5.01$ & & $5.00$ & & $4.75$ & & $4.89$ & & $4.84$ & & $4.84$ & & $5.22$ \\
\hline
$N_e$ &
$\norm{\epsilon_{E_1}}$ & EOC &
$\norm{\epsilon_{\rho_2}}$ & EOC &
$\norm{\epsilon_{\rho_2 v_{2,1}}}$ & EOC &
$\norm{\epsilon_{\rho_2 v_{2,2}}}$ & EOC &
$\norm{\epsilon_{\rho_2 v_{2,3}}}$ & EOC &
$\norm{\epsilon_{E_2}}$ & EOC &
$\norm{\epsilon_{\psi}}$ & EOC 
\\
\hline
$4$ & $7.11 \times 10^{-6}$ & $-$ & $1.24 \times 10^{-6}$ & $-$ & $3.58 \times 10^{-6}$ & $-$ & $3.63 \times 10^{-6}$ & $-$ & $1.55 \times 10^{-6}$ & $-$ & $1.23 \times 10^{-5}$ & $-$ & $3.69 \times 10^{-7}$ & $-$ \\
$8$ & $2.80 \times 10^{-7}$ & $4.67$ & $3.87 \times 10^{-8}$ & $5.00$ & $1.52 \times 10^{-7}$ & $4.56$ & $1.53 \times 10^{-7}$ & $4.56$ & $4.10 \times 10^{-8}$ & $5.24$ & $3.23 \times 10^{-7}$ & $5.26$ & $5.94 \times 10^{-9}$ & $5.96$ \\
$16$ & $9.64 \times 10^{-9}$ & $4.86$ & $2.10 \times 10^{-9}$ & $4.20$ & $5.63 \times 10^{-9}$ & $4.76$ & $5.66 \times 10^{-9}$ & $4.76$ & $1.32 \times 10^{-9}$ & $4.96$ & $1.61 \times 10^{-8}$ & $4.33$ & $4.55 \times 10^{-10}$ & $3.71$ \\
$32$ & $2.91 \times 10^{-10}$ & $5.05$ & $6.91 \times 10^{-11}$ & $4.93$ & $1.75 \times 10^{-10}$ & $5.01$ & $1.76 \times 10^{-10}$ & $5.01$ & $4.12 \times 10^{-11}$ & $5.00$ & $5.28 \times 10^{-10}$ & $4.93$ & $1.37 \times 10^{-11}$ & $5.05$ \\
mean & & $4.86$ & & $4.71$ & & $4.78$ & & $4.78$ & & $5.07$ & & $4.84$ & & $4.91$ \\
\hline
\end{tabular}
}
\end{table}

\begin{table}[]
    \centering
    \caption{$L_2$ errors and EOCs for the convergence test with the \textbf{ES} (ii) and \textbf{EC+LLF} (iii) schemes for $N=2$ and different numbers of elements per direction $N_e$.}
    \label{tab:eoc_es_1}
    \resizebox{\columnwidth}{!}{
    \begin{tabular}{ccccccccccccccc}
\hline
$N_e$ &
$\norm{\epsilon_{B_1}}$ & EOC &
$\norm{\epsilon_{B_2}}$ & EOC &
$\norm{\epsilon_{B_3}}$ & EOC &
$\norm{\epsilon_{\rho_1}}$ & EOC &
$\norm{\epsilon_{\rho_1 v_{1,1}}}$ & EOC &
$\norm{\epsilon_{\rho_1 v_{1,2}}}$ & EOC &
$\norm{\epsilon_{\rho_1 v_{1,3}}}$ & EOC 
\\
\hline
$16$ & $2.03 \times 10^{-5}$ & $-$ & $2.04 \times 10^{-5}$ & $-$ & $1.11 \times 10^{-5}$ & $-$ & $3.14 \times 10^{-5}$ & $-$ & $9.23 \times 10^{-5}$ & $-$ & $9.20 \times 10^{-5}$ & $-$ & $5.04 \times 10^{-6}$ & $-$ \\
$32$ & $1.71 \times 10^{-6}$ & $3.57$ & $1.72 \times 10^{-6}$ & $3.57$ & $1.26 \times 10^{-6}$ & $3.13$ & $4.51 \times 10^{-6}$ & $2.80$ & $1.16 \times 10^{-5}$ & $2.99$ & $1.16 \times 10^{-5}$ & $2.98$ & $5.43 \times 10^{-7}$ & $3.21$ \\
$64$ & $2.26 \times 10^{-7}$ & $2.92$ & $2.25 \times 10^{-7}$ & $2.93$ & $1.25 \times 10^{-7}$ & $3.33$ & $6.02 \times 10^{-7}$ & $2.90$ & $1.47 \times 10^{-6}$ & $2.98$ & $1.47 \times 10^{-6}$ & $2.98$ & $6.13 \times 10^{-8}$ & $3.15$ \\
$128$ & $3.04 \times 10^{-8}$ & $2.90$ & $3.03 \times 10^{-8}$ & $2.90$ & $1.34 \times 10^{-8}$ & $3.23$ & $7.72 \times 10^{-8}$ & $2.96$ & $1.85 \times 10^{-7}$ & $2.99$ & $1.85 \times 10^{-7}$ & $2.99$ & $7.75 \times 10^{-9}$ & $2.98$ \\
mean & & $3.13$ & & $3.13$ & & $3.23$ & & $2.89$ & & $2.99$ & & $2.98$ & & $3.11$ \\
\hline
$N_e$ &
$\norm{\epsilon_{E_1}}$ & EOC &
$\norm{\epsilon_{\rho_2}}$ & EOC &
$\norm{\epsilon_{\rho_2 v_{2,1}}}$ & EOC &
$\norm{\epsilon_{\rho_2 v_{2,2}}}$ & EOC &
$\norm{\epsilon_{\rho_2 v_{2,3}}}$ & EOC &
$\norm{\epsilon_{E_2}}$ & EOC &
$\norm{\epsilon_{\psi}}$ & EOC 
\\
\hline
$16$ & $1.95 \times 10^{-4}$ & $-$ & $6.77 \times 10^{-5}$ & $-$ & $1.33 \times 10^{-4}$ & $-$ & $1.33 \times 10^{-4}$ & $-$ & $6.10 \times 10^{-6}$ & $-$ & $2.24 \times 10^{-4}$ & $-$ & $2.65 \times 10^{-5}$ & $-$ \\
$32$ & $1.92 \times 10^{-5}$ & $3.34$ & $1.04 \times 10^{-5}$ & $2.71$ & $1.79 \times 10^{-5}$ & $2.89$ & $1.79 \times 10^{-5}$ & $2.89$ & $9.29 \times 10^{-7}$ & $2.71$ & $2.69 \times 10^{-5}$ & $3.06$ & $3.71 \times 10^{-6}$ & $2.83$ \\
$64$ & $1.86 \times 10^{-6}$ & $3.37$ & $1.40 \times 10^{-6}$ & $2.89$ & $2.31 \times 10^{-6}$ & $2.96$ & $2.31 \times 10^{-6}$ & $2.96$ & $1.30 \times 10^{-7}$ & $2.83$ & $3.34 \times 10^{-6}$ & $3.01$ & $4.83 \times 10^{-7}$ & $2.94$ \\
$128$ & $2.02 \times 10^{-7}$ & $3.20$ & $1.79 \times 10^{-7}$ & $2.97$ & $2.91 \times 10^{-7}$ & $2.99$ & $2.91 \times 10^{-7}$ & $2.99$ & $1.72 \times 10^{-8}$ & $2.92$ & $4.17 \times 10^{-7}$ & $3.00$ & $6.10 \times 10^{-8}$ & $2.99$ \\
mean & & $3.30$ & & $2.86$ & & $2.95$ & & $2.95$ & & $2.82$ & & $3.02$ & & $2.92$ \\
\hline
\end{tabular}
}
\end{table}

\begin{table}[]
    \centering
    \caption{$L_2$ errors and EOCs for the convergence test with the \textbf{ES} (ii) and \textbf{EC+LLF} (iii) schemes for $N=3$ and different numbers of elements per direction $N_e$.}
    \resizebox{\columnwidth}{!}{
    \begin{tabular}{ccccccccccccccc}
\hline
$N_e$ &
$\norm{\epsilon_{B_1}}$ & EOC &
$\norm{\epsilon_{B_2}}$ & EOC &
$\norm{\epsilon_{B_3}}$ & EOC &
$\norm{\epsilon_{\rho_1}}$ & EOC &
$\norm{\epsilon_{\rho_1 v_{1,1}}}$ & EOC &
$\norm{\epsilon_{\rho_1 v_{1,2}}}$ & EOC &
$\norm{\epsilon_{\rho_1 v_{1,3}}}$ & EOC 
\\
\hline
$8$ & $7.33 \times 10^{-6}$ & $-$ & $7.34 \times 10^{-6}$ & $-$ & $2.03 \times 10^{-6}$ & $-$ & $7.07 \times 10^{-6}$ & $-$ & $1.26 \times 10^{-5}$ & $-$ & $1.26 \times 10^{-5}$ & $-$ & $2.24 \times 10^{-6}$ & $-$ \\
$16$ & $3.79 \times 10^{-7}$ & $4.27$ & $3.80 \times 10^{-7}$ & $4.27$ & $9.40 \times 10^{-8}$ & $4.43$ & $3.81 \times 10^{-7}$ & $4.21$ & $6.38 \times 10^{-7}$ & $4.30$ & $6.37 \times 10^{-7}$ & $4.31$ & $5.24 \times 10^{-8}$ & $5.42$ \\
$32$ & $1.55 \times 10^{-8}$ & $4.61$ & $1.55 \times 10^{-8}$ & $4.61$ & $6.34 \times 10^{-9}$ & $3.89$ & $2.56 \times 10^{-8}$ & $3.90$ & $3.93 \times 10^{-8}$ & $4.02$ & $3.92 \times 10^{-8}$ & $4.02$ & $2.58 \times 10^{-9}$ & $4.34$ \\
$64$ & $9.84 \times 10^{-10}$ & $3.98$ & $9.84 \times 10^{-10}$ & $3.98$ & $3.87 \times 10^{-10}$ & $4.03$ & $1.58 \times 10^{-9}$ & $4.02$ & $2.37 \times 10^{-9}$ & $4.05$ & $2.37 \times 10^{-9}$ & $4.05$ & $1.55 \times 10^{-10}$ & $4.06$ \\
mean & & $4.29$ & & $4.29$ & & $4.12$ & & $4.04$ & & $4.12$ & & $4.13$ & & $4.61$ \\
\hline
$N_e$ &
$\norm{\epsilon_{E_1}}$ & EOC &
$\norm{\epsilon_{\rho_2}}$ & EOC &
$\norm{\epsilon_{\rho_2 v_{2,1}}}$ & EOC &
$\norm{\epsilon_{\rho_2 v_{2,2}}}$ & EOC &
$\norm{\epsilon_{\rho_2 v_{2,3}}}$ & EOC &
$\norm{\epsilon_{E_2}}$ & EOC &
$\norm{\epsilon_{\psi}}$ & EOC 
\\
\hline
$8$ & $3.63 \times 10^{-5}$ & $-$ & $1.09 \times 10^{-5}$ & $-$ & $2.81 \times 10^{-5}$ & $-$ & $2.81 \times 10^{-5}$ & $-$ & $4.55 \times 10^{-6}$ & $-$ & $6.64 \times 10^{-5}$ & $-$ & $3.04 \times 10^{-6}$ & $-$ \\
$16$ & $2.15 \times 10^{-6}$ & $4.08$ & $5.54 \times 10^{-7}$ & $4.30$ & $1.36 \times 10^{-6}$ & $4.36$ & $1.36 \times 10^{-6}$ & $4.36$ & $7.30 \times 10^{-8}$ & $5.96$ & $3.67 \times 10^{-6}$ & $4.18$ & $1.68 \times 10^{-7}$ & $4.18$ \\
$32$ & $1.40 \times 10^{-7}$ & $3.94$ & $3.68 \times 10^{-8}$ & $3.91$ & $8.44 \times 10^{-8}$ & $4.01$ & $8.44 \times 10^{-8}$ & $4.01$ & $3.85 \times 10^{-9}$ & $4.25$ & $2.31 \times 10^{-7}$ & $3.99$ & $1.01 \times 10^{-8}$ & $4.05$ \\
$64$ & $8.61 \times 10^{-9}$ & $4.02$ & $2.33 \times 10^{-9}$ & $3.98$ & $5.22 \times 10^{-9}$ & $4.01$ & $5.22 \times 10^{-9}$ & $4.01$ & $2.41 \times 10^{-10}$ & $4.00$ & $1.44 \times 10^{-8}$ & $4.00$ & $6.28 \times 10^{-10}$ & $4.01$ \\
mean & & $4.01$ & & $4.06$ & & $4.13$ & & $4.13$ & & $4.74$ & & $4.06$ & & $4.08$ \\
\hline
\end{tabular}
}
\end{table}

\begin{table}[]
    \centering
    \caption{$L_2$ errors and EOCs for the convergence test with the \textbf{ES} (ii) and \textbf{EC+LLF} (iii) schemes for $N=4$ and different numbers of elements per direction $N_e$.}
    \resizebox{\columnwidth}{!}{
    \begin{tabular}{ccccccccccccccc}
\hline
$N_e$ &
$\norm{\epsilon_{B_1}}$ & EOC &
$\norm{\epsilon_{B_2}}$ & EOC &
$\norm{\epsilon_{B_3}}$ & EOC &
$\norm{\epsilon_{\rho_1}}$ & EOC &
$\norm{\epsilon_{\rho_1 v_{1,1}}}$ & EOC &
$\norm{\epsilon_{\rho_1 v_{1,2}}}$ & EOC &
$\norm{\epsilon_{\rho_1 v_{1,3}}}$ & EOC 
\\
\hline
$8$ & $2.30 \times 10^{-7}$ & $-$ & $2.30 \times 10^{-7}$ & $-$ & $1.10 \times 10^{-7}$ & $-$ & $4.18 \times 10^{-7}$ & $-$ & $9.62 \times 10^{-7}$ & $-$ & $9.58 \times 10^{-7}$ & $-$ & $6.85 \times 10^{-8}$ & $-$ \\
$16$ & $6.35 \times 10^{-9}$ & $5.18$ & $6.38 \times 10^{-9}$ & $5.17$ & $4.24 \times 10^{-9}$ & $4.70$ & $1.54 \times 10^{-8}$ & $4.76$ & $3.44 \times 10^{-8}$ & $4.80$ & $3.44 \times 10^{-8}$ & $4.80$ & $1.91 \times 10^{-9}$ & $5.17$ \\
$32$ & $2.04 \times 10^{-10}$ & $4.96$ & $2.04 \times 10^{-10}$ & $4.96$ & $1.27 \times 10^{-10}$ & $5.06$ & $5.54 \times 10^{-10}$ & $4.80$ & $1.21 \times 10^{-9}$ & $4.84$ & $1.20 \times 10^{-9}$ & $4.84$ & $6.57 \times 10^{-11}$ & $4.86$ \\
$64$ & $6.99 \times 10^{-12}$ & $4.87$ & $7.00 \times 10^{-12}$ & $4.87$ & $3.58 \times 10^{-12}$ & $5.15$ & $2.05 \times 10^{-11}$ & $4.76$ & $3.99 \times 10^{-11}$ & $4.92$ & $3.99 \times 10^{-11}$ & $4.92$ & $2.37 \times 10^{-12}$ & $4.79$ \\
mean & & $5.00$ & & $5.00$ & & $4.97$ & & $4.77$ & & $4.85$ & & $4.85$ & & $4.94$ \\
\hline
$N_e$ &
$\norm{\epsilon_{E_1}}$ & EOC &
$\norm{\epsilon_{\rho_2}}$ & EOC &
$\norm{\epsilon_{\rho_2 v_{2,1}}}$ & EOC &
$\norm{\epsilon_{\rho_2 v_{2,2}}}$ & EOC &
$\norm{\epsilon_{\rho_2 v_{2,3}}}$ & EOC &
$\norm{\epsilon_{E_2}}$ & EOC &
$\norm{\epsilon_{\psi}}$ & EOC 
\\
\hline
$8$ & $2.52 \times 10^{-6}$ & $-$ & $6.16 \times 10^{-7}$ & $-$ & $1.36 \times 10^{-6}$ & $-$ & $1.36 \times 10^{-6}$ & $-$ & $1.10 \times 10^{-7}$ & $-$ & $2.57 \times 10^{-6}$ & $-$ & $2.61 \times 10^{-7}$ & $-$ \\
$16$ & $7.49 \times 10^{-8}$ & $5.07$ & $2.74 \times 10^{-8}$ & $4.49$ & $4.98 \times 10^{-8}$ & $4.77$ & $4.98 \times 10^{-8}$ & $4.77$ & $3.10 \times 10^{-9}$ & $5.15$ & $8.03 \times 10^{-8}$ & $5.00$ & $1.07 \times 10^{-8}$ & $4.60$ \\
$32$ & $1.92 \times 10^{-9}$ & $5.28$ & $1.07 \times 10^{-9}$ & $4.69$ & $1.77 \times 10^{-9}$ & $4.81$ & $1.77 \times 10^{-9}$ & $4.81$ & $1.06 \times 10^{-10}$ & $4.88$ & $2.58 \times 10^{-9}$ & $4.96$ & $3.80 \times 10^{-10}$ & $4.82$ \\
$64$ & $5.11 \times 10^{-11}$ & $5.23$ & $3.64 \times 10^{-11}$ & $4.87$ & $5.88 \times 10^{-11}$ & $4.92$ & $5.87 \times 10^{-11}$ & $4.92$ & $3.64 \times 10^{-12}$ & $4.86$ & $8.24 \times 10^{-11}$ & $4.97$ & $1.24 \times 10^{-11}$ & $4.94$ \\
mean & & $5.19$ & & $4.68$ & & $4.83$ & & $4.83$ & & $4.96$ & & $4.98$ & & $4.79$ \\
\hline
\end{tabular}
}
\end{table}

\begin{table}[]
    \centering
    \caption{$L_2$ errors and EOCs for the convergence test with the \textbf{ES} (ii) and \textbf{EC+LLF} (iii) schemes for $N=5$ and different numbers of elements per direction $N_e$.}
    \label{tab:eoc_es_4}
    \resizebox{\columnwidth}{!}{
    \begin{tabular}{ccccccccccccccc}
\hline
$N_e$ &
$\norm{\epsilon_{B_1}}$ & EOC &
$\norm{\epsilon_{B_2}}$ & EOC &
$\norm{\epsilon_{B_3}}$ & EOC &
$\norm{\epsilon_{\rho_1}}$ & EOC &
$\norm{\epsilon_{\rho_1 v_{1,1}}}$ & EOC &
$\norm{\epsilon_{\rho_1 v_{1,2}}}$ & EOC &
$\norm{\epsilon_{\rho_1 v_{1,3}}}$ & EOC 
\\
\hline
$4$ & $9.34 \times 10^{-7}$ & $-$ & $9.36 \times 10^{-7}$ & $-$ & $1.91 \times 10^{-7}$ & $-$ & $9.89 \times 10^{-7}$ & $-$ & $1.84 \times 10^{-6}$ & $-$ & $1.83 \times 10^{-6}$ & $-$ & $2.41 \times 10^{-7}$ & $-$ \\
$8$ & $6.72 \times 10^{-9}$ & $7.12$ & $6.73 \times 10^{-9}$ & $7.12$ & $2.26 \times 10^{-9}$ & $6.40$ & $1.09 \times 10^{-8}$ & $6.51$ & $1.49 \times 10^{-8}$ & $6.95$ & $1.48 \times 10^{-8}$ & $6.95$ & $1.75 \times 10^{-9}$ & $7.11$ \\
$16$ & $9.07 \times 10^{-11}$ & $6.21$ & $9.08 \times 10^{-11}$ & $6.21$ & $3.34 \times 10^{-11}$ & $6.08$ & $1.51 \times 10^{-10}$ & $6.17$ & $2.02 \times 10^{-10}$ & $6.20$ & $2.02 \times 10^{-10}$ & $6.20$ & $1.76 \times 10^{-11}$ & $6.63$ \\
$32$ & $1.26 \times 10^{-12}$ & $6.17$ & $1.26 \times 10^{-12}$ & $6.17$ & $5.11 \times 10^{-13}$ & $6.03$ & $2.36 \times 10^{-12}$ & $6.00$ & $3.11 \times 10^{-12}$ & $6.02$ & $3.11 \times 10^{-12}$ & $6.02$ & $2.63 \times 10^{-13}$ & $6.06$ \\
mean & & $6.50$ & & $6.50$ & & $6.17$ & & $6.23$ & & $6.39$ & & $6.39$ & & $6.60$ \\
\hline
$N_e$ &
$\norm{\epsilon_{E_1}}$ & EOC &
$\norm{\epsilon_{\rho_2}}$ & EOC &
$\norm{\epsilon_{\rho_2 v_{2,1}}}$ & EOC &
$\norm{\epsilon_{\rho_2 v_{2,2}}}$ & EOC &
$\norm{\epsilon_{\rho_2 v_{2,3}}}$ & EOC &
$\norm{\epsilon_{E_2}}$ & EOC &
$\norm{\epsilon_{\psi}}$ & EOC 
\\
\hline
$4$ & $3.24 \times 10^{-6}$ & $-$ & $1.58 \times 10^{-6}$ & $-$ & $2.46 \times 10^{-6}$ & $-$ & $2.47 \times 10^{-6}$ & $-$ & $4.97 \times 10^{-7}$ & $-$ & $5.08 \times 10^{-6}$ & $-$ & $3.82 \times 10^{-7}$ & $-$ \\
$8$ & $5.01 \times 10^{-8}$ & $6.02$ & $1.36 \times 10^{-8}$ & $6.86$ & $2.81 \times 10^{-8}$ & $6.45$ & $2.81 \times 10^{-8}$ & $6.46$ & $2.52 \times 10^{-9}$ & $7.62$ & $7.53 \times 10^{-8}$ & $6.07$ & $3.63 \times 10^{-9}$ & $6.72$ \\
$16$ & $7.20 \times 10^{-10}$ & $6.12$ & $1.94 \times 10^{-10}$ & $6.13$ & $4.11 \times 10^{-10}$ & $6.09$ & $4.11 \times 10^{-10}$ & $6.10$ & $2.73 \times 10^{-11}$ & $6.53$ & $1.18 \times 10^{-9}$ & $6.00$ & $4.91 \times 10^{-11}$ & $6.21$ \\
$32$ & $1.10 \times 10^{-11}$ & $6.03$ & $2.95 \times 10^{-12}$ & $6.04$ & $6.26 \times 10^{-12}$ & $6.04$ & $6.26 \times 10^{-12}$ & $6.04$ & $3.41 \times 10^{-13}$ & $6.32$ & $1.84 \times 10^{-11}$ & $6.00$ & $7.58 \times 10^{-13}$ & $6.02$ \\
mean & & $6.06$ & & $6.34$ & & $6.19$ & & $6.20$ & & $6.82$ & & $6.02$ & & $6.32$ \\
\hline
\end{tabular}
}
\end{table}

\q{c29r2}{
\change{

To illustrate the effect of the grid size and polynomial degree on the divergence cleaning effectiveness of the GLM technique, we show the $L_2$ and $L_{\infty}$ norms of the magnetic field divergence error at $t=1$ for the different resolutions used in the convergence test and the ES solver in Tables \ref{tab:eoc_divb_n23} and \ref{tab:eoc_divb_n45}. The error decreases as the solution gets better resolved with an experimental order of convergence (EOC) of approximately $\mathcal{O}(N)$.

\begin{table}[h!]
    \caption{\change{Magnetic field divergence error at time $t=1$ for the different resolutions of the convergence test and polynomial degrees $N=2$ and $N=3$.}}
    \label{tab:eoc_divb_n23}
    \centering
    \begin{tabular}{ccccc|ccccc}
    \hline
    \multicolumn{5}{c|}{$N=2$} & \multicolumn{5}{c}{$N=3$} \\
    \hline
         $N_e$ & $\norm{{\Nabla \cdot \vec{B}}}$ & EOC & $\norm{{\Nabla \cdot \vec{B}}}_{\infty}$ & EOC &
         $N_e$ & $\norm{{\Nabla \cdot \vec{B}}}$ & EOC & $\norm{{\Nabla \cdot \vec{B}}}_{\infty}$ & EOC \\
    \hline
$16$ & $9.11E-04$ & $--$ & $2.29E-03$ & $--$ &       $8$ & $2.93E-04$ & $--$ & $1.48E-03$ & $--$ \\
$32$ & $2.35E-04$ & $1.96$ & $6.36E-04$ & $1.85$ & $16$ & $3.41E-05$ & $3.10$ & $1.67E-04$ & $3.14$ \\
$64$ & $6.07E-05$ & $1.95$ & $1.66E-04$ & $1.93$ & $32$ & $4.14E-06$ & $3.04$ & $2.01E-05$ & $3.06$ \\
$128$ & $1.53E-05$ & $1.99$ & $4.20E-05$ & $1.99$ & $64$ & $5.14E-07$ & $3.01$ & $2.49E-06$ & $3.01$ \\
    \hline
    \end{tabular}
    \caption{\change{Magnetic field divergence error at time $t=1$ for the different resolutions of the convergence test and polynomial degrees $N=4$ and $N=5$.}}
    \label{tab:eoc_divb_n45}
    \centering
    \begin{tabular}{ccccc|ccccc}
    \hline
    \multicolumn{5}{c|}{$N=4$} & \multicolumn{5}{c}{$N=5$} \\
    \hline
         $N_e$ & $\norm{{\Nabla \cdot \vec{B}}}$ & EOC & $\norm{{\Nabla \cdot \vec{B}}}_{\infty}$ & EOC &
         $N_e$ & $\norm{{\Nabla \cdot \vec{B}}}$ & EOC & $\norm{{\Nabla \cdot \vec{B}}}_{\infty}$ & EOC \\
    \hline
$8$ & $1.64E-05$ & $--$ & $9.33E-05$ & $--$ &     $4$ & $2.82E-05$ & $--$ & $2.25E-04$ & $--$ \\
$16$ & $1.01E-06$ & $4.02$ & $4.77E-06$ & $4.29$ & $8$ & $6.65E-07$ & $5.41$ & $5.46E-06$ & $5.36$ \\
$32$ & $6.54E-08$ & $3.95$ & $2.93E-07$ & $4.02$ & $16$ & $1.96E-08$ & $5.09$ & $1.52E-07$ & $5.17$ \\
$64$ & $4.17E-09$ & $3.97$ & $1.91E-08$ & $3.94$ & $32$ & $6.02E-10$ & $5.02$ & $4.62E-09$ & $5.04$ \\
    \hline
    \end{tabular}
\end{table}

}
}

\subsection{Entropy Conservation and Stability Tests} \label{sec:ec_es}

To test the entropy dissipation properties of the schemes of the paper, we simulate a weak two-species magneto-hydrodynamic blast wave.
The initial condition is a multi-species and magnetized version of the weak Euler blast wave of \citet{Hennemann2020}.
This modification of the weak blast wave initial condition was initially proposed by \citet{czernik2022entropy} in the conext of multi-component MHD simulations.
The extension to multi-ion GLM-MHD is given by
\begin{align}\label{eq:ec_es_condini}
    \rho_k (\vec{x}, t = 0) &= \frac{-2^{(k - 1)}} {1 - 2^{N_i}} \rho_0(\vec{x}), 
    &
    p_k (\vec{x}, t = 0) &= 
    \begin{cases}
        1 & \mathrm{if}~r>0.5, \\
        1.245 & \mathrm{otherwise},
    \end{cases}
    \nonumber\\
    v_{k,1}  (\vec{x}, t = 0) &= 
    \begin{cases}
        0 & \mathrm{if}~r>0.5, \\
        0.1882 \cos \phi & \mathrm{otherwise},
    \end{cases}
    &
    v_{k,2}  (\vec{x}, t = 0) &= 
    \begin{cases}
        0 & \mathrm{if}~r>0.5, \\
        0.1882 \sin \phi & \mathrm{otherwise},
    \end{cases} 
    \nonumber\\
    v_{k,3} (\vec{x}, t = 0) &= 0,&
    B_1 (\vec{x}, t = 0) &= 1,\nonumber\\
    B_2 (\vec{x}, t = 0) &= 1, &
    B_3 (\vec{x}, t = 0) &= 1,\nonumber\\
    \red{\psi(\vec{x}, t = 0)} &\red{=0,} &
\end{align}
with $k=1,2$, the polar angle $\phi \coloneqq \arctan(y/x)$, the radius $r \coloneqq \sqrt{x^2 + y^2}$, and the auxiliary variable
\begin{equation}
    \rho_0 (\vec{x}) =
    \begin{cases}
        1 & \mathrm{if}~r>0.5, \\
        1.1691 \sin \phi & \mathrm{otherwise}.
    \end{cases} 
\end{equation}

As in the last section, we use different heat capacity ratios for the two ion species, $\gamma_1 = 2$ and $\gamma_2 = 4$, and also different charge-to-mass rations, $r_1=2$ and $r_2=1$, to consider important multi-ion effects.

For this problem, we also use a non-trivial electron pressure of the form \eqref{eq:pe_alpha} with $\alpha = 0.2$ to test the entropy dissipation properties of the electron pressure averaging introduced in Section \ref{sec:fvec}.

\q{c31r3_1}{
\change{
We use a ``weak blast wave'' simulation, such that we can evaluate the entropy dissipation of our DG methods without employing a shock-capturing approach. 
As mentioned in Section \ref{sec:intro}, DG methods typically require a shock-capturing mechanism to introduce additional dissipation whean dealing with sharp discontinuities in the solution, such as shocks. 
Notably, the \textbf{EC} solver is particularly unstable in the presence of discontinuities, as it is designed to be nearly dissipation-free.
}

We solve the multi-ion GLM-MHD system with the initial condition given by \eqref{eq:ec_es_condini} in the domain $\Omega = [-2,2]^2$ with periodic boundary conditions, the final time $t_f = 0.4$, and the three different solvers that we studied in Section \ref{sec:convergence}.
We use a tessellation of $16 \times 16$ quadrilateral elements of degree $N=3$.
\change{Figure \ref{fig:ec_es_contours} shows the density contours of the two ion species at $t = t_f = 0.4$ for the \textbf{ES} solver and two different resolutions: (i) a coarse resolution of $26 \times 16$ elements, and (ii) a finer resolution with $128 \times 128$ elements. As expected, some oscillations are visible, since a shock-capturing method is not employed in this test in order to examine the entropy production properties of the DG scheme.
}

\begin{figure}
    \centering
    \begin{minipage}[t]{0.3\linewidth}
        \begin{subfigure}[t]{\linewidth}
            \includegraphics[trim=316 621 1450 37, clip,width=\textwidth]{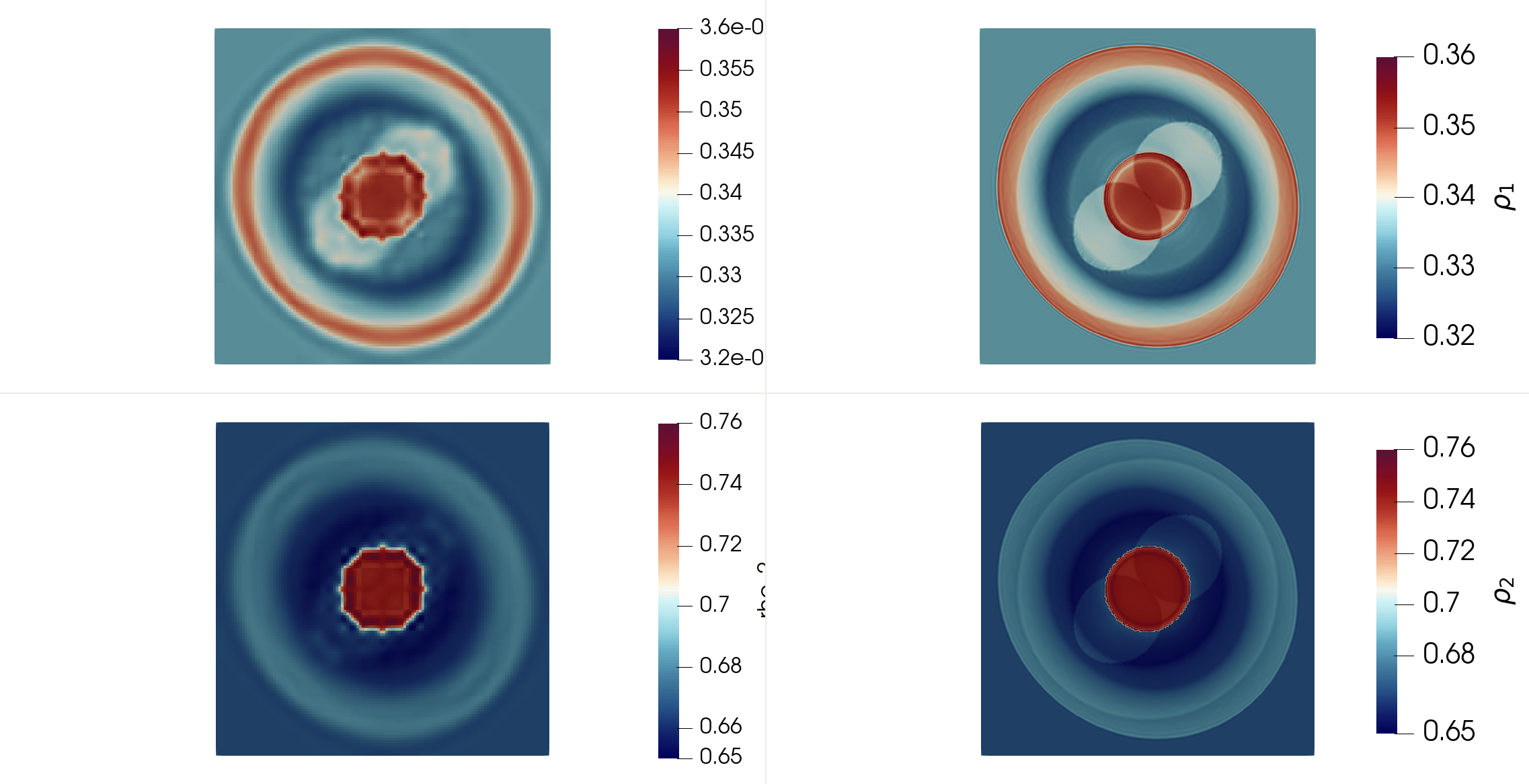}
    
            \includegraphics[trim=316 37 1450 621, clip,width=\textwidth]{figs/ec_es/ec_es.png}
            \caption{Mesh with $16 \times 16$ elements}
        \end{subfigure}
    \end{minipage}
    \begin{minipage}[t]{0.3\linewidth}
        \begin{subfigure}[t]{\linewidth}
            \includegraphics[trim=1450 621 316 37, clip,width=\textwidth]{figs/ec_es/ec_es.png}
    
            \includegraphics[trim=1450 37 316 621, clip,width=\textwidth]{figs/ec_es/ec_es.png}
            \caption{Mesh with $128 \times 128$ elements}
        \end{subfigure}
    \end{minipage}
    \begin{minipage}[t]{0.15\linewidth}
        \includegraphics[trim=2040 621 0 37, clip,height=2\textwidth]{figs/ec_es/ec_es.png}
        
        \includegraphics[trim=2040 37 0 621, clip,height=2\textwidth]{figs/ec_es/ec_es.png}
    \end{minipage}
    
    \caption{\change{Contours of $\rho_1$ (top) and $\rho_2$ (bottom) for the weak blast wave simulation at $t=0.4$ with the \textbf{ES} solver and polynomial degree $N=3$.
    We show results for the coarse resolution of the test, $16 \times 16$ elements (left), and a finer resolution, $128 \times 128$ elements (right).}}
    \label{fig:ec_es_contours}
\end{figure}
}

\change{To investigate the entropy dissipation properties of the various schemes}, we investigate the evolution of the total entropy of the domain,
\begin{equation}
    S_{\Omega}(t) = \int_{\Omega, N} S(\vec{x}, t) \d \vec{x},
\end{equation}
and the total entropy \change{production} rate,
\begin{equation}
    \dot{S}_{\Omega}(t) = 
    \int_{\Omega, N} \dot{S}(\vec{x}, t) \d \vec{x}
    =
    \int_{\Omega, N} \entVar^T \dot{\state{u}}(\vec{x}, t) \d \vec{x},
\end{equation}
where the sub-index $N$ on the integral denotes that we compute the integral using an LGL quadrature rule with $N+1$ points per direction in each element. 
I.e., the same quadrature rule that is used by the LGL-DGSEM method.

Figure \ref{fig:ECES_plots} shows the evolution of the total entropy change and the maximum entropy change rate as a function of the CFL number for the three different schemes studied in this section: (i) the entropy-conservative (\textbf{EC}) scheme, (ii) the provably entropy-stable scheme (\textbf{ES}), and (iii) the scheme that uses entropy-conservative fluxes in the volume integral together with the standard LLF scheme at the element interfaces (\textbf{EC+LLF}).

First, Figure \ref{fig:ec_es_total_entropy} illustrates the variation in total entropy, comparing the final and initial states across the domain for each simulation. 
The results indicate that all schemes lead to a net dissipation of entropy. 
As expected, the \textbf{EC+LLF} (ii) scheme and the \textbf{ES} scheme (iii) emerge as the most dissipative, succeeded by the \textbf{EC} scheme (i). 
Since the \textbf{EC} scheme is entropy conservative at the semi-discrete level, the observed net entropy dissipation in the \textbf{EC} scheme is attributed to the explicit time-integration method. 
Consequently, as the time-step size decreases, the net entropy dissipation of the scheme approaches zero with fourth-order accuracy, reflective of the RK4-5 method utilized.

\q{c31r3}{
\begin{figure}
\centering
\begin{subfigure}[b]{0.49\linewidth}
    \includegraphics[trim=0 0 0 0, clip,
    width=\linewidth]{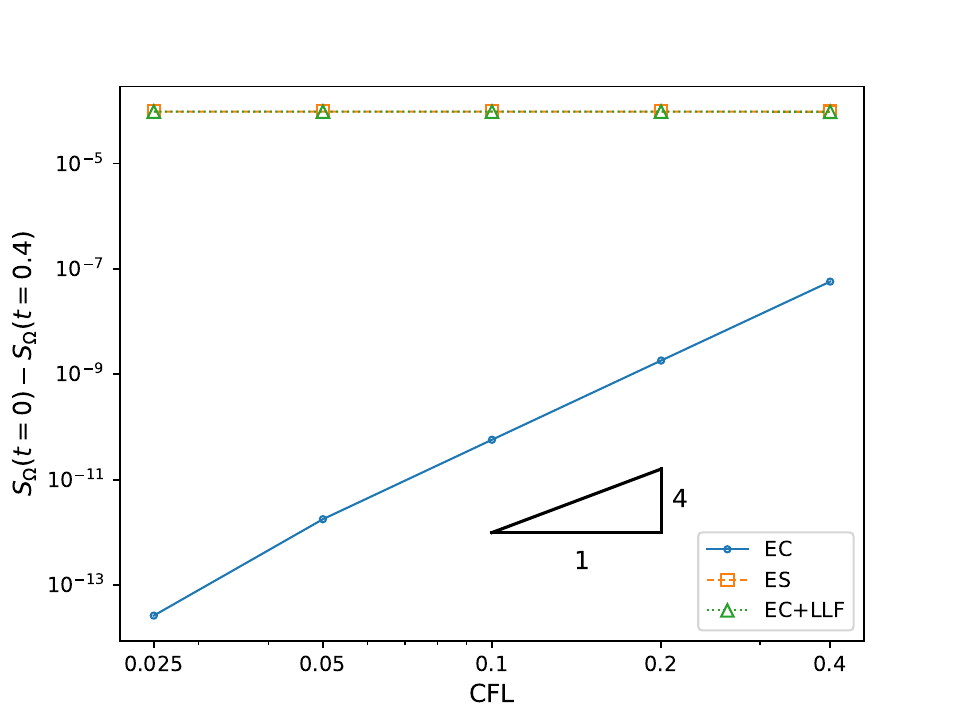}
    \caption{Total entropy change}
    \label{fig:ec_es_total_entropy}
\end{subfigure}
\begin{subfigure}[b]{0.49\linewidth}
    \includegraphics[trim=0 0 0 0, clip,
    width=\linewidth]{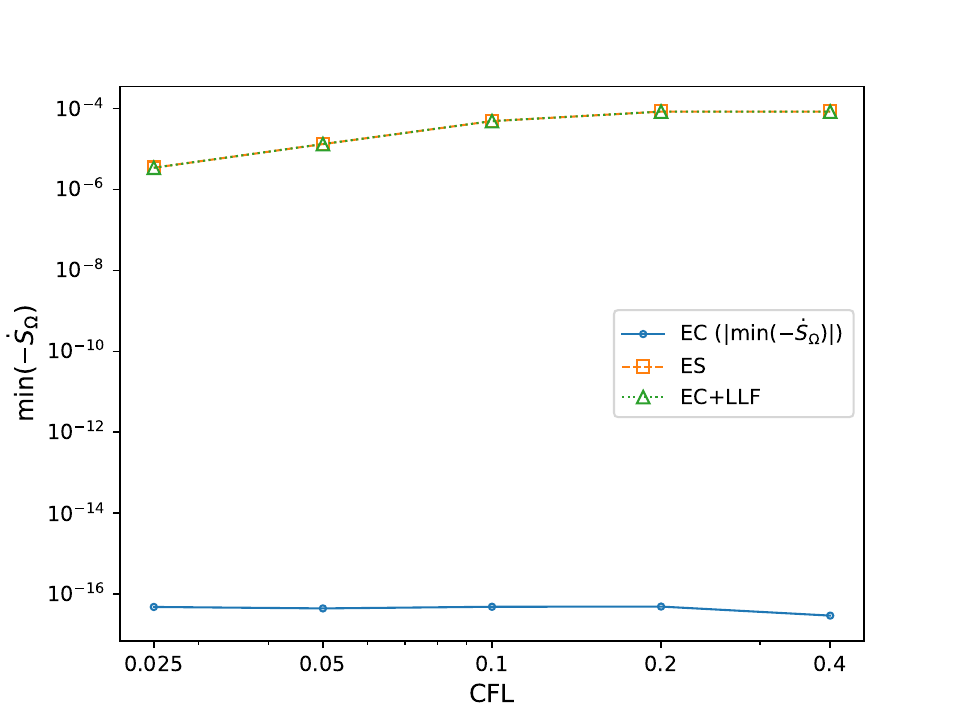}
    \caption{Maximum entropy change rate}
    \label{fig:ec_es_total_entropy_change}
\end{subfigure}
\caption{
Log-log plots of the total entropy change and the maximum entropy \change{production} rate \change{over all time steps of the simulation, $-\max_{t \in [t, t_f]} \dot{S}_{\Omega} (t) = \min_{t \in [t, t_f]} (-\dot{S}_{\Omega} (t))$,} as a function of the CFL number for three different schemes: ((i) the entropy-conservative (\textbf{EC}) scheme, (ii) the provably entropy-stable scheme (\textbf{ES}), and (iii) the scheme that uses entropy-conservative fluxes in the volume integral together with the standard LLF scheme at the element interfaces (LLF).
}
\label{fig:ECES_plots}
\end{figure}
}

Finally, Figure \ref{fig:ec_es_total_entropy_change} presents the maximum rate of total entropy \change{production} observed across all time steps for each simulation, offering a detailed view of the schemes' performance over time. This visualization confirms that both the \textbf{ES} scheme (ii) and the \textbf{EC+LLF} scheme (iii) consistently dissipate entropy throughout the simulation.
A noteworthy observation is that the maximum rate of semi-discrete entropy dissipation for these schemes decreases with the reduction of the time-step size. This trend highlights the interplay between spatial and temporal discretization in influencing the schemes' entropy characteristics. For clarity and to accommodate a logarithmic scale, we plot the absolute values of the maximum total entropy change rate for the entropy-conservative (\textbf{EC}) method (i), given that these values approach machine precision zero. 
This result shows that the \textbf{EC} scheme achieves ideal entropy conservation at the semi-discrete level, and it confirms that the net entropy dissipation illustrated in Figure \ref{fig:ec_es_total_entropy} is a consequence of the time-integration scheme.

The results of this section support the mathematical analysis presented in Section \ref{sec:discretization}.

\subsection{Magnetized Kelvin-Helmholtz Instability}\label{sec:khi}

We consider a multi-ion extension of the magnetized Kelvin-Helmholtz instability problem proposed by \citet{mignone2010high} to test the robustness of the scheme in turbulent under-resolved conditions.

We consider two ion species for this simulation: a hydrogen ion H$^+$ ($k=1$) and an ionized hydrogen molecule H$_2^+$ ($k=2)$.
The heat capacity ratios are $\gamma_1 = 5/3$ and $\gamma_2 = 1.4$ for the monoatomic and diatomic molecules, respectively.
Provided a suitable non-dimensionalization, the charge-to-mass ratios are $r_1 = 1$ and $r_2 = 1/2$.
We do not consider collisional source terms between the ion species and assume a trivial electron pressure gradient, $\Nabla p_e = 0$.

As in the single-species case, the initial condition is given by a shear layer with velocities in the $x$ direction, constant density and pressure, a small single-mode perturbation in the $y$-velocity, and a uniform magnetic field in the $xz$ plane.
This initial condition evolves into a Kelvin-Helmholtz instability and breaks down into MHD turbulence.
We set partial densities that sum up to unity, partial pressures that are computed with the individual heat capacity ratios, and the same velocities for the two ion species:
\begin{align}\label{eq:khi_condini}
    \rho_k(x,y,t=0) &= \frac{1}{2}, &p_k(x,y,t=0) &= \frac{1}{\gamma_k}, &\red{\psi(x,y,t=0)} &\red{= 0,} \nonumber\\
    v_{k,1} (x,y,t=0) &= \frac{1}{2} \tanh \left( \frac{y}{y_0} \right),
    &v_{k,2}(x,y,t=0) &= v_{2,0} \sin (2 \pi x) \exp \left(- \frac{y^2}{\sigma^2} \right), &
    v_{k,3}(x,y,t=0) &= 0, \nonumber \\
    B_1 (x,y,t=0) &= c_a \cos \theta,
    &B_2(x,y,t=0) &= 0, &B_3(x,y,t=0) &= c_a \sin \theta, 
\end{align}
with $k=1,2$, $y_0=1/20$ is the steepness of the shear, $c_a=0.1$ is the Alfvén speed, and $\theta = \pi/3$ is the angle of the initial magnetic field.
Moreover, the parameters of the perturbation are $v_{2,0}=0.01$ and $\sigma=0.1$.

We solve the multi-ion GLM-MHD system with the initial condition given by \eqref{eq:khi_condini} in the domain $\Omega = [-1,1]^2$ with final time $t_f = 20$, and the two dissipative solvers that we studied in Section \ref{sec:convergence}: (ii) the provably entropy-stable (\textbf{ES}) scheme, and (iii) the \textbf{EC+LLF} scheme.
The entropy-conserving (\textbf{EC}) scheme (i) does not provide enough dissipation to run this simulation stably.
Moreover, we also test a standard LGL-DGSEM discretization of the multi-ion GLM-MHD equations that uses the standard LLF solver at the element interfaces.
The standard LGL-DGSEM does not use a split formulation.
For all simulations, we use a tessellation of $128 \times 128$ quadrilateral elements of degree $N=3$, which corresponds to $512 \times 512$ degrees of freedom (DOFs).

We would like to remark that the computational domain of the original study by \citet{mignone2010high} spans $x \in [0,1]$, $y \in [-1,1]$, with periodic boundary conditions at the $x=0$ and $x=1$ boundaries, and perfectly conducting slip-wall boundary conditions at the $y=-1$ and $y=1$ boundaries.
However, we use a square-shaped domain to accommodate this case to the two-dimensional \texttt{TreeMesh} solver of \texttt{Trixi.jl}.
We use periodic boundary conditions at the $x=-1$ and $x=1$ boundaries, and perfectly conducting slip-wall boundary conditions at the $y=-1$ and $y=1$.

\newcommand{\supout}{\mathrm{out}}
\newcommand{\supin}{\mathrm{in}}
We impose the perfectly conducting slip-wall boundary condition weakly by evaluating the surface numerical flux and non-conservative terms with an outer state that mirrors the normal velocity and magnetic fields,
\begin{align}
    \state{u}^{\supout} (\state{u}^{\supin}, \vec{n})
    &=
    \left(
    \rho^{\supout}_k,
    \rho^{\supout}_k \vec{v}^{\supout}_k,
    E^{\supout}_k,
    \vec{B}^{\supout},
    \red{\psi^{\supout}}
    \right)^T
    \nonumber\\
    &=
    \left(
    \rho^{\supin}_k,
    \rho^{\supin}_k \left(\vec{v}^{\supin}_k - 2 \vec{v}^{\supin}_k \cdot \vec{n}\right),
    E^{\supin}_k,
    \vec{B}^{\supin} - 2 \vec{B}^{\supin} \cdot \vec{n},
    \red{\psi^{\supin}}
    \right)^T,
\end{align}
where $\vec{n}$ is the outward-pointing unit normal vector at the boundary of the domain.
This simple way of prescribing the slip-wall boundary performs well in this particular example, as the main flow features develop away from the wall.
In more challenging setups, it might be necessary to use a more robust slip-wall boundary condition.
For instance, entropy-stable wall boundary conditions have been developed in \cite{hindenlang2020stability,chan2022entropy1}.

Figure~\ref{fig:KHI_contours_es} shows the evolution of the magnetized Kelvin-Helmholtz instability problem for the fourth-order ($N=3$) LGL-DGSEM that uses the provably entropy-stable (\textbf{ES}) solver (ii).
We show the densities of species H$^+$ and H$_2^+$, and the ratio of the poloidal field, 
\begin{equation}
    B_p := \sqrt{B_1^2 + B_2^2},
\end{equation}
to the toroidal field, $B_t := B_3$.
Even though we run the simulations on a square domain, $\Omega = [-1,1]^2$, we plot the quantities only for $x \in [0,1]$ to match the visualization presented in previous studies \cite{mignone2010high,rueda2023entropy}, as the solution is periodic.
At early stages of the simulation ($t \le 5$), the perturbation follows a linear growth phase, in which the cat’s eye vortex structure forms, see e.g. the top row of Figure~\ref{fig:KHI_contours_es} and \cite{mignone2010high}.
As the simulation continues, magnetic field lines become distorted and energy is transferred to smaller scales in the onset of MHD turbulence.
Energy is then dissipated by the artificial viscosity and resistivity of the methods.

Figure~\ref{fig:KHI_contours_llf} shows the evolution of the magnetized Kelvin-Helmholtz instability problem for the fourth-order ($N=3$) \textbf{EC+LLF} solver (iii): the LGL-DGSEM solver that uses the entropy-conserving flux in the volume integral and standard LLF fluxes at the element interfaces.
Again, we show the densities of species H$^+$ and H$_2^+$, and the ratio of the poloidal to the toroidal magnetic field.
The evolution of the two species' densities and the magnetic field is almost identical as with the \textbf{ES} (ii) solver up to time $t=8$ (second panel).
However, as turbulence evolves for later times of the simulation, the contours of Figure~\ref{fig:KHI_contours_llf} differ from those in Figure~\ref{fig:KHI_contours_es}.
At those stages of the simulation, the evolution of the state quantities is determined by the slightly different (numerical) dissipation properties of the schemes (ii) and (iii).

\newcommand{\tablerow}[2]{
 \includegraphics[trim=557 152 1508 359,clip,height=0.28\linewidth]{figs/khi/#2_t_#1.png}
 &
 \includegraphics[trim=0 164 1715 164,clip,height=0.28\linewidth]{figs/khi/#2_t_#1.png}
 &
 \includegraphics[trim=760 164 960 164,clip,height=0.28\linewidth]{figs/khi/#2_t_#1.png}
 &
 \includegraphics[trim=1515 164 213 164,clip,height=0.28\linewidth]{figs/khi/#2_t_#1.png}
 &
 \includegraphics[trim=2065 152 0 359,clip,height=0.28\linewidth]{figs/khi/#2_t_#1.png}
 \\
}

\begin{figure}[h!]
\centering

\begin{tabular}{cccccc}
  & $\mathrm{H}^+$ & $\mathrm{H}^+_2$ & & \\
    \tablerow{05}{multiion_es_khi}
    \tablerow{08}{multiion_es_khi}
    \tablerow{12}{multiion_es_khi}
    \tablerow{20}{multiion_es_khi}
\end{tabular}
\caption{Evolution of the magnetized Kelvin-Helmholtz instability problem using a fourth-order \textbf{ES} discretization (ii) of the multi-ion MHD model at times $t=5$ (top panel), $t=8$ (second panel), $t=12$ (third panel) and $t=20$ (bottom panel) with the highest resolution ($256 \times 512$ DOFs).
We show the densities of the ion species, $H^+$ and $H^+_2$, and the ratio of the poloidal field, $B_p$, to the toroidal field, $B_t := B_3$. 
}
\label{fig:KHI_contours_es}
\end{figure}

\begin{figure}[h!]
\centering

\begin{tabular}{cccccc}
  & $\mathrm{H}^+$ & $\mathrm{H}^+_2$ & & \\
    \tablerow{05}{multiion_khi}
    \tablerow{08}{multiion_khi}
    \tablerow{12}{multiion_khi}
    \tablerow{20}{multiion_khi}
\end{tabular}
\caption{Evolution of the magnetized Kelvin-Helmholtz instability problem using a fourth-order \textbf{EC+LLF} (iii) discretization of the multi-ion MHD model at times $t=5$ (top panel), $t=8$ (second panel), $t=12$ (third panel) and $t=20$ (bottom panel) with the highest resolution ($256 \times 512$ DOFs).
We show the densities of the ion species, $H^+$ and $H^+_2$, and the ratio of the poloidal field, $B_p$, to the toroidal field, $B_t := B_3$. 
}
\label{fig:KHI_contours_llf}
\end{figure}

Figure~\ref{fig:KHI_plots} shows the evolution of the $L_2$ norm of the normalized poloidal magnetic energy, defined as
\begin{equation}
    \langle B_p^2 \rangle(t) := \frac{\int_{\Omega} B_p^2(t) \d \vec{x}}{\int_{\Omega} B_p^2(t=0) \d \vec{x}},
\end{equation} 
and the $L_2$ norm of the divergence error.
We show the solution obtained with five different methods/schemes (we use a numbering that is consistent with the numbering introduced in Section \ref{sec:convergence}):
\begin{itemize}
    \item[(ii)] \textbf{ES:} The provably entropy-stable split-form LGL-DGSEM discretization of the multi-ion GLM-MHD system (solid blue line).
    \item[(iii)] \textbf{EC+LLF:} The dissipative split-form LGL-DGSEM discretization of the multi-ion GLM-MHD system that uses the EC flux in the volume integral and the standard LLF flux at the element interfaces (dotted black line).
    \item[(iv)] \textbf{ES (no GLM):} The provably entropy-stable split-form LGL-DGSEM discretization (ii) without GLM divergence cleaning (dash-dot-dotted magenta line).
    \item[(v)] \textbf{EC+LLF (no GLM):} The dissipative split-form LGL-DGSEM discretization (iii) without GLM divergence cleaning  (dashed green line).
    \item[(vi)] \textbf{Std DG:} The standard LGL-DGSEM discretization of the multi-ion GLM-MHD system that does not use a split formulation and uses a standard LLF solver at the element interfaces (dash-dotted red line).
\end{itemize}

Only the solvers (ii) and (iii) are able to run the simulation until the end time, which highlights the importance of using the GLM divergence cleaning technique and an entropy-dissipative discretization.
In accordance with the observations made for Figures~\ref{fig:KHI_contours_es} and \ref{fig:KHI_contours_llf}, the evolution of the normalized poloidal magnetic energy and the divergence error is almost identical for the schemes (ii) and (iii) until time $t=9$.
After that, the different dissipation properties of schemes (ii) and (iii) lead to slightly different results in the onset of MHD turbulence.

Figure~\ref{fig:KHI_plots_divB} shows that schemes (iv) and (v) crash shortly after the formation of the cat's eye vortex structure, as the divergence error grows uncontrollably due to the absence of the GLM divergence cleaning technique.
The growth in the divergence error results in a non-physical evolution of the poloidal magnetic energy, as observed in Figure~\ref{fig:KHI_plots_Bp}, and a flow in which a negative density/pressure causes the scheme to crash. 
Remarkably, the \textbf{ES} scheme without GLM (iv) manages to run for a little longer than the \textbf{EC+LLF} scheme (v).

It is noteworthy that the standard LGL-DGSEM discretization of the GLM-MHD system (vi) crashes when turbulence starts to develop in the domain.
In other words, the GLM divergence cleaning technique is not enough to improve the robustness of the scheme.
It is the combination of divergence cleaning and entropy stability that equips the LGL-DGSEM with enough robustness to run this example until the end time.

\begin{figure}
    \centering
    \begin{subfigure}[b]{0.48\linewidth}
        \includegraphics[trim=0 0 0 0, clip,
        width=\linewidth]{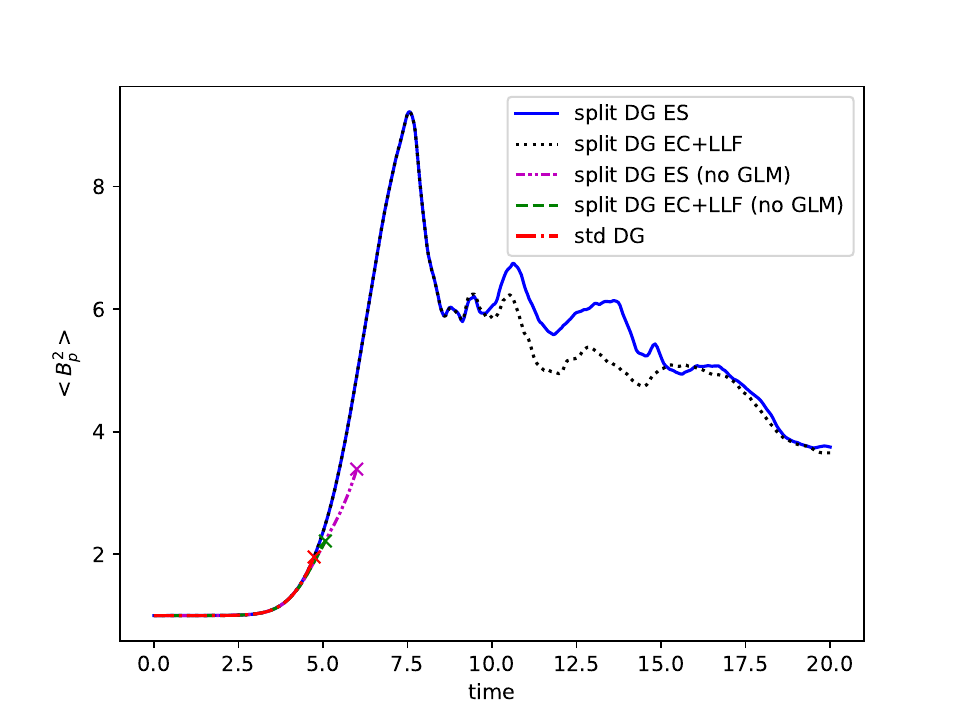}
        \caption{Normalized poloidal magnetic energy}
        \label{fig:KHI_plots_Bp}
    \end{subfigure}
    \begin{subfigure}[b]{0.48\linewidth}
        \includegraphics[trim=0 0 0 0, clip,
        width=\linewidth]{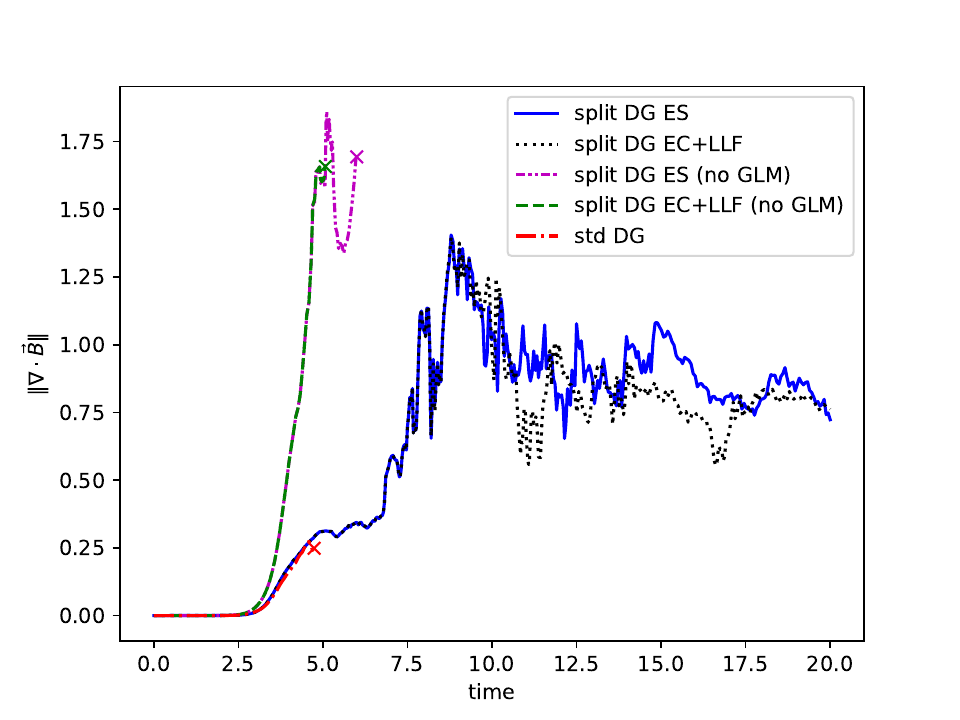}
        \caption{Divergence error}
        \label{fig:KHI_plots_divB}
    \end{subfigure}
    \caption{Evolution of the normalized volume-integrated poloidal magnetic energy $\langle B_p^2 \rangle$ and the $L_2$ norm of the divergence error $\norm{\Nabla \cdot \vec{B}}$ for the magnetized KHI problem with two species (H$^+$ and H$_2^+$) and five different numerical schemes.
    An ``x'' marks the crash of a simulation.}
    \label{fig:KHI_plots}
\end{figure}

\q{c33r3}{
\change{
To facilitate a better comparison between Figures \ref{fig:KHI_contours_es} and \ref{fig:KHI_contours_llf}, and to visualize the effects of entropy stability on the solution, Figure \ref{fig:khi_slices} shows the density of the first ion species, $\rho_1$, along a slice of the simulation domain at $y=0$ for the solvers with GLM divergence cleaning at various simulation times. 

We selected the same simulation times as shown in Figures \ref{fig:KHI_contours_es} and \ref{fig:KHI_contours_llf}. However, instead of showing results at $t=5$, we present results at $t=4.5$ in Figure \ref{fig:khi_slice_t_045}, just before the \textbf{Std DG} solver crashes. At $t=4.5$, the \textbf{ES} and \textbf{EC+LLF} methods produce similar results, but the \textbf{Std DG} solver exhibits strong density oscillations near the periodic boundaries, ultimately leading to its failure.

Figure \ref{fig:khi_slice_t_08} shows that both the \textbf{ES} and \textbf{EC+LLF} methods predict similar density values up to time $t=8$. However, as turbulence develops, the density predictions of the two methods begin to diverge significantly, as observed in Figures \ref{fig:khi_slice_t_12} and \ref{fig:khi_slice_t_20}. This behavior was also noted in Figures \ref{fig:KHI_contours_es} and \ref{fig:KHI_contours_llf}.

\begin{figure}
    \centering
    \begin{subfigure}[b]{0.4\linewidth}
        \includegraphics[width=\textwidth]{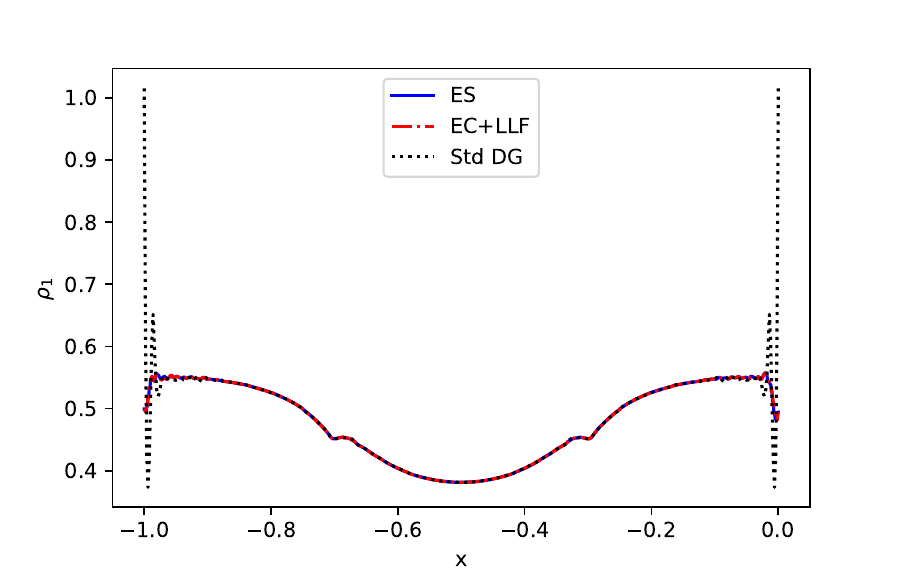}
        \caption{$t=4.5$}
        \label{fig:khi_slice_t_045}
    \end{subfigure}
    \begin{subfigure}[b]{0.4\linewidth}
        \includegraphics[width=\textwidth]{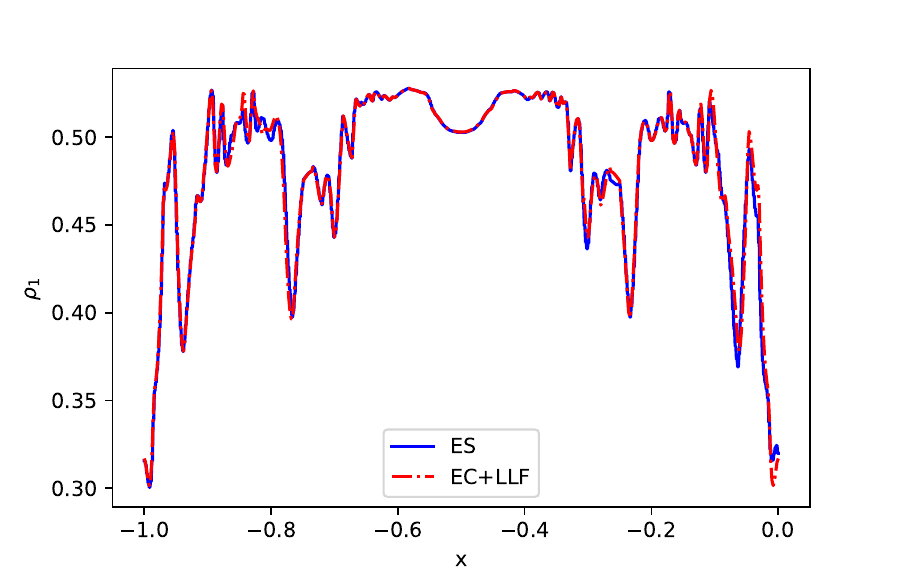}
        \caption{$t=8$}
        \label{fig:khi_slice_t_08}
    \end{subfigure}

    \begin{subfigure}[b]{0.4\linewidth}
        \includegraphics[width=\textwidth]{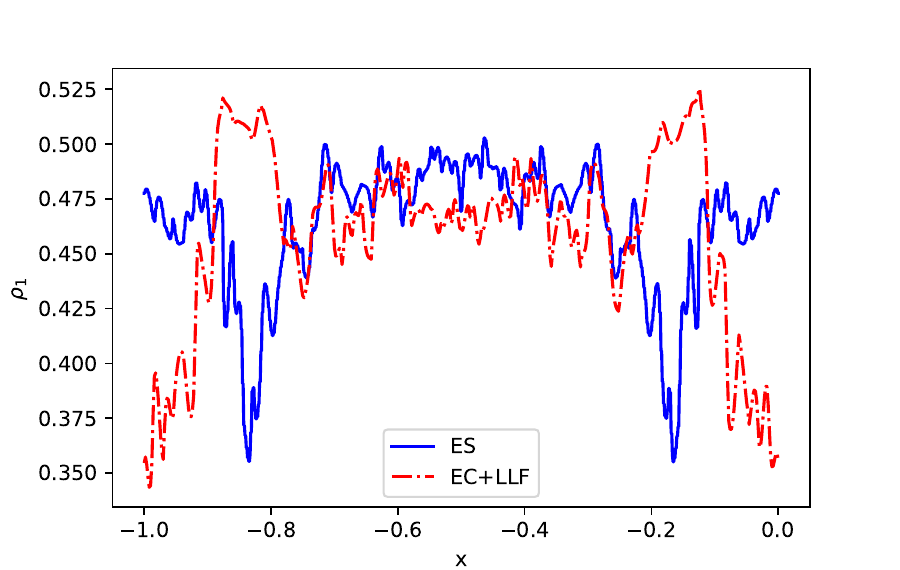}
        \caption{$t=12$}
        \label{fig:khi_slice_t_12}
    \end{subfigure}
    \begin{subfigure}[b]{0.4\linewidth}
        \includegraphics[width=\textwidth]{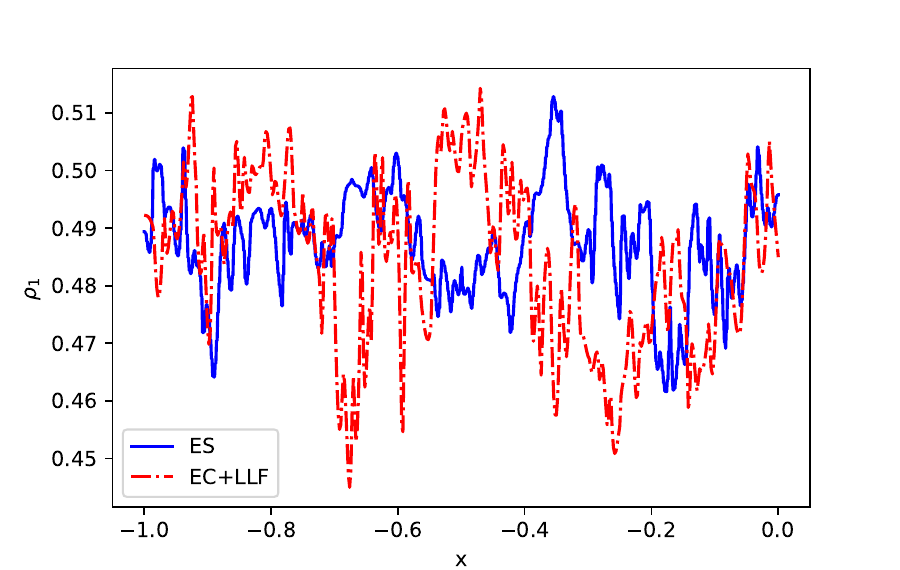}
        \caption{$t=20$}
        \label{fig:khi_slice_t_20}
    \end{subfigure}
    \caption{\change{Density of the first ion species, $\rho_1$, along a slice of the simulation domain at $y=0$ for the solvers with GLM divergence cleaning and different simulation times.}}
    \label{fig:khi_slices}
\end{figure}
}
}

\q{c9r2}{
\change{
\subsection{Performance Comparison}

To compare the computational performance of the various methods used in this paper, we revisit the weak blast simulation from Section \ref{sec:ec_es}. We run the simulation using the same end time, $t_f = 0.4$, and the same resolution, applying the four different discontinuous Galerkin discretization methods for the multi-ion GLM-MHD equations (the numbering of the methods remains consistent with the previous sections):
\begin{itemize}
    \item[(i)] \textbf{EC:} The provably entropy-conservative split-form LGL-DGSEM discretization of the multi-ion GLM-MHD system.
    \item[(ii)] \textbf{ES:} The provably entropy-stable split-form LGL-DGSEM discretization of the multi-ion GLM-MHD system.
    \item[(iii)] \textbf{EC+LLF:} The dissipative split-form LGL-DGSEM discretization of the multi-ion GLM-MHD system that uses the EC flux in the volume integral and the standard LLF flux at the element interfaces.
    \item[(vi)] \textbf{Std DG:} The standard LGL-DGSEM discretization of the multi-ion GLM-MHD system that uses the ``standard'' volume numerical flux and non-conservative terms \eqref{eq:std_vol_fluxes} and a standard LLF solver at the element interfaces.
\end{itemize}
Note that we have excluded the solver variants that do not employ the GLM divergence-cleaning technique, as our implementation always includes the additional variable \red{$\psi$} (in the previous section, the GLM technique was deactivated by setting \red{$c_h = 0$}). 
In the following, we present performance indices scaled with the number of variables in the system, allowing for an estimation of the computational cost associated with the GLM technique.

Table \ref{tab:pid} presents the performance indices and the number of time steps for the various methods used in this paper to solve the weak blast wave test with CFL$=0.4$. These measurements were taken on a single core of an AMD Ryzen Threadripper 3990X.
The performance index is calculated as the average computational time to perform a right-hand-side evaluation, divided by the number of spatial degrees of freedom in the mesh. The number of samples used to compute the performance index is the product of the number of time steps and the number of Runge-Kutta stages per time step (five in this case).
Additionally, we present the performance index scaled by the number of state variables for this two-species simulation, which is $n_v = 5 N_i + 4 = 14$. A detailed description of the performance index and the performance optimizations in the \texttt{Trixi.jl} package can be found in \cite{ranocha2023efficient}.

As expected, the computationally cheapest scheme is \textbf{Std DG}, as it only requires the computation of standard averages for the two-point fluxes and non-conservative terms. However, this variant is not robust enough to handle under-resolved turbulence, as demonstrated in Section \ref{sec:khi}.
Following in computational complexity is the \textbf{EC+LLF} scheme, which requires the computation of EC fluxes and non-conservative terms in the volume integral but only standard averages in the surface integral. The \textbf{EC} scheme is slightly more expensive, as it necessitates the computation of EC fluxes and non-conservative terms in both the volume and surface integrals.
The computationally most expensive scheme is the provably entropy-stable scheme (\textbf{ES}). This scheme requires the computation of the jump in entropy variables and their product with the discrete entropy Jacobian $\hat{\mathcal{H}}$ for the surface integral (see, e.g., \eqref{eq:esflux}), in addition to the EC fluxes and non-conservative terms in the volume and surface integrals. Since the discrete entropy Jacobian $\hat{\mathcal{H}}$ is a sparse matrix, in our implementation, we do not assemble the entire matrix but only compute the non-zero entries and perform the matrix-vector product with those non-zero entries.

\begin{table}[h]
    \caption{\change{Serial performance indices and number of time steps for the different methods of the paper to solve the weak blast wave test with CFL$=0.4$. The number of variables for this two-species simulation is $n_v = 5 N_i + 4 = 14$.}}
    \label{tab:pid}
    \centering
    \begin{tabular}{c|cccc}
        \hline
        Method & \textbf{EC} & \textbf{ES} & \textbf{EC+LLF} & \textbf{Std DG}   \\
        \hline
        time/DOF/rhs [s] & 
        $9.2104 \times 10^{-7}$ &
        $1.0517 \times 10^{-6}$ &
        $9.0783 \times 10^{-7}$ &
        $6.5737 \times 10^{-7}$ \\
        time/DOF/$n_v$/rhs [s] & 
        $6.5788 \times 10^{-8}$ &
        $7.5122 \times 10^{-8}$ & 
        $6.4845 \times 10^{-8}$ &
        $4.6955 \times 10^{-8}$ \\
        time steps & 
        $128$ &
        $126$ & 
        $126$ &
        $126$\\
        \hline
    \end{tabular}
    
\end{table}
}
}

\section{Conclusions \change{and Outlook}} \label{sec:Conclusions}

This paper makes three significant contributions to the numerical analysis of the ideal multi-ion magneto-hydrodynamics (MHD) equations. 
Our first contribution involves enhancing the equation system with a divergence cleaning mechanism, utilizing the Generalized Lagrange Multiplier (GLM) technique.
This addition ensures a more robust computational model by mitigating numerical divergence issues inherent in MHD simulations.

Our second contribution introduces an innovative algebraic manipulation of the system. 
This adjustment enables the identification of a strictly convex entropy function that aligns with the thermodynamic entropies of individual ion species.
Consequently, our model adheres to the second law of thermodynamics at a continuous level, ensuring physical accuracy in simulations.

The third and final contribution of our work lies in the development of both finite volume (FV) and high-order discontinuous Galerkin (DG) entropy-consistent discretizations for the GLM-MHD system. 
These discretizations are notable for their entropy conservation or stability, aligning with the second law of thermodynamics at the semi-discrete level across any number of ion species. Moreover, they are consistent with established discretization schemes for single-fluid MHD equations when reduced to a single ion species scenario. 

Through a series of numerical experiments, we validated the entropy-consistency, high-order accuracy, and improved robustness of our methods. These tests confirm the practical effectiveness and theoretical soundness of our contributions to the field of computational MHD.

\q{c4_r1}{
\change{
While we specifically address the one-dimensional case in the derivation of the methods in this paper, extending the methods to 3D, curvilinear and unstructured grids using tensor-product hexahedral elements  is a relatively straightforward process, as demonstrated in previous works by the authors, such as \cite{Bohm2018, Rueda-Ramirez2021, rueda2023entropy}.
Extensions of this class of entropy-stable discontinuous Galerkin methods to unstructured grids made up of triangles and hexahedra have been proposed in the literature.
For instance, \citet{chen2017entropy, crean2018entropy, chan2018discretely} presented high-order entropy-stable methods for triangles and tetrahedra. These methods lack a tensor-product structure, making them computationally more expensive than traditional DGSEM methods on hexahedra.
More recently,
efficient extensions of entropy-stable discontinuous Galerkin methods to unstructured grids using tensor-product summation-by-parts
operators on triangles and tetrahedra have been proposed by \citet{montoya2023efficient,montoya2024efficient}.
These discretizations for triangles and tetrahedra rely on two-point numerical fluxes, such as the ones developed in the present paper.
}
}

\q{c30_r3}{
\change{
In this paper, we have focused on the high-order convergence, entropy-dissipation properties and robustness for under-resolved turbulence simulations, and did not discuss shock-capturing strategies. 
As is typical for high-order DG schemes, the high-order DG schemes derived in this paper require a shock-capturing mechanism to introduce local additional dissipation when dealing with sharp discontinuities in the solution, such as shocks.
The structure of the DG schemes developed in this paper is suitable to use the entropy-consistent shock-capturing scheme for non-conservative systems developed by the authors in \cite{Rueda-Ramirez2020}, or the node-local subcell shock-capturing techniques developed by the authors in \cite{rueda2024flux}.
}
}

\q{c4_r1_2}{
\change{
Future work will involve developing numerical schemes that conserve total energy under less restrictive conditions, incorporating a self-consistent treatment of electron energy while maintaining entropy consistency, achieving symmetrization, and applying these schemes to challenging problems in space physics and astrophysics. These applications will include scenarios involving shocks, under-resolved MHD turbulence, 3D curvilinear and unstructured grids with mixed elements, and distributed parallel architectures.
}
}

\q{c3r1_2}{
The methods detailed in this paper are implemented for the two-dimensional Cartesian \texttt{TreeMesh} solver of the open-source framework \texttt{Trixi.jl} \cite{schlottkelakemper2020trixi, ranocha2022adaptive, schlottke2021purely}.
A reproducibility repository can be found on \url{https://github.com/amrueda/paper_2024_multi-ion_mhd}, \change{where we link the code with the 1D and 2D implementation of the methods}.
}

\section*{Acknowledgments}

The authors would like to thank Prof. Joachim Saur and Stephan Schlegel for the very insightful conversations about the multi-ion MHD system.
We would also like to thank Dr. Michael Schlottle-Lakemper for his helpful comments during the preparation of this paper.
\change{The authors would also like to thank the anonymous reviewers for the useful comments during the peer-review process.
}

Gregor Gassner and Andrés M. Rueda-Ramírez acknowledge funding through the Klaus-Tschira Stiftung via the project ``HiFiLab''.

Gregor Gassner and Aleksey Sikstel acknowledge funding from the German Science Foundation DFG through the research unit “SNuBIC” (DFG-FOR5409).

We furthermore thank the Regional Computing Center of the University of Cologne (RRZK) for providing computing time on the High Performance Computing (HPC) system ODIN as well as support.

\printcredits

\bibliographystyle{model1-num-names}

\bibliography{Biblio.bib}

\section*{Appendices}
\appendix

\section{Derivation of the Entropy-Conservative Fluxes and Non-Conservative Terms} \label{app:ec_derivation}

\q{34r4}{
As reviewed in Section \ref{sec:fvec}, numerical fluxes and non-conservative terms provide entropy conservation if they satisfy a generalization of Tadmor's shuffle condition \eqref{eq:entropy_prod_ec},
\begin{equation} \label{eq:entropy_prod_ec2}
    r_{(i,j)} =
    r (\state{u}_i,\state{u}_j) = 
    \jump{\entVar}_{(i,j)}^T 
    \numfluxb{f}_{(i,j)}^{\ec}
    + \entVar^T_{j} {\Jan}^{\diamond, \ec}_{(j,i)}
    - \entVar^T_{i} {\Jan}^{\diamond, \ec}_{(i,j)}
    - \jump{\Psi}_{(i,j)} = 0,
\end{equation}
where $r_{(i,j)}$ is the entropy production between degrees of freedom $i$ and $j$, and $\Psi$ is the so-called entropy (flux) potential, defined as \eqref{eq:entropy_potential}
\begin{equation}\label{eq:entropy_potential2}
\Psi = \entVar^T \left(\state{f} + \Jan \right) - {f}^S 
=
\underbrace{\entVar^T \state{f}^{\supEuler} - {f}^S}_{\coloneqq \Psi^{\supEuler}}
+
\underbrace{\entVar^T \left(\state{f}^{\supMHD} + \Jan^{\supMHD} \right)}_{\coloneqq \Psi^{\supMHD}}
+ {\color{red}\underbrace{\entVar^T\left(
\state{f}^{\supGLM} + \Jan^{\supGLM}
\right)}_{\coloneqq \Psi^{\supGLM}}}.
\end{equation}

Our approach to obtaining entropy-conservative two-point numerical fluxes and non-conservative terms involves breaking down the entropy-conserving flux into Euler, MHD and GLM components, 
\begin{equation}\label{eq:fluxes_app}
    \numfluxb{f}^{\ec}_{(i,j)} = \numfluxb{f}^{\ec,\supEuler}_{(i,j)}
    + \numfluxb{f}^{\ec, \supMHD}_{(i,j)}
    + {\color{red}  \numfluxb{f}^{\ec, \supGLM}_{(i,j)}},
\end{equation}
and the entropy-conserving two-point non-conservative terms into Godunov-Powell, Lorentz, multi-ion, and GLM components,
\begin{equation}\label{eq:noncons_app}
    {\Jan}^{\diamond, \ec}_{(i,j)} = 
    {\Jan}^{\Powell\diamond, \ec}_{(i,j)} +
    {\Jan}^{\rm{Lor}\diamond, \ec}_{(i,j)} +
    {\Jan}^{\rm{multi}\diamond, \ec}_{(i,j)} +
    {\color{red}{\Jan}^{\rm{GLM}\diamond, \ec}_{(i,j)}}.
\end{equation}

\change{
By substituting \eqref{eq:fluxes_app} and \eqref{eq:noncons_app} into \eqref{eq:entropy_prod_ec2}, we derive a scalar equation with several unknown terms. To simplify the computation of the various entropy-conserving fluxes and non-conservative terms, we analyze each set of terms individually.
Specifically those related to the Euler, mMHD, and GLM components.
We ensure that the entropy production from each of these components is zero. This approach allows us to develop a discretization that results in zero total entropy production.
}
}
As discussed above, there are several solutions to \eqref{eq:entropy_prod_ec2}, and we will focus on finding one consistent with the single-fluid GLM-MHD discretization of \citet{Derigs2018}.
The GLM terms are important to control the magnetic field divergence error and provide robustness to the schemes. 
However, all derivations of this paper are valid with or without GLM terms. 
Therefore, we mark all GLM terms in \red{red} to improve the readability.

We start by obtaining an entropy-conservative flux for the Euler terms.
Since the Euler terms of each ion species of the multi-ion GLM-MHD system are completely decoupled from each other and correspond to individual copies of the Euler equations, it is straightforward to extend any existing entropy-conservative flux to these terms, e.g. \cite{chandrashekar2013kinetic, Ismail2009, Ranocha2020}.
In order to keep consistency with the EC flux of \citet{Derigs2018} for single-fluid MHD, we choose the entropy-conserving and kinetic-energy preserving flux of \citet{chandrashekar2013kinetic}.
The multi-ion extension of this flux reads as
\begin{equation}\label{eq:ECFlux_Euler}
\state{f}^{\ec, \supEuler}(\state{u}_L,\state{u}_R) =
\begin{pmatrix} 
\rho^{\ln}_{k} \avg{v_{k,1}} \\
\rho^{\ln}_{k} \avg{v_{k,1}}^2 + \overline{p}_k \\ 
\rho^{\ln}_{k} \avg{v_{k,1}} \avg{v_{k,2}} \\
\rho^{\ln}_{k} \avg{v_{k,1}} \avg{v_{k,3}}  \\
f^{\ec, \supEuler}_{E_k} \\
\vec{0} \\
0
\end{pmatrix}
\end{equation}
with 
\begin{equation}\label{eq:ECFlux2}
\begin{split}
f^{\ec, \supEuler}_{E_k} = & f_{\rho_k}^{\ec,\supEuler}\bigg[\frac{1}{2 (\gamma_k-1) \beta^{\ln}_{k}} - \frac{1}{2} \left(\avg{v_{k,1}^2} + \avg{v_{k,2}^2} + \avg{v_{k,3}^2}\right) \bigg] \\
&+ f_{\rho_k v_{k,1}}^{\ec,\supEuler} \avg{v_{k,1}} 
+ f_{\rho_k v_{k,2}}^{\ec,\supEuler} \avg{v_{k,2}} 
+ f_{\rho_k v_{k,3}}^{\ec,\supEuler} \avg{v_{k,3}}
\end{split}
\end{equation}
and
\begin{equation*}
\overline{p}_k = \frac{\avg{\rho_k}}{2\avg{\beta_k}}.
\end{equation*}

Following the derivations of \citet{chandrashekar2013kinetic}, it is possible to show that the flux \eqref{eq:ECFlux_Euler} fulfills a conservative Tadmor shuffle condition in the multi-species sense, such that
\begin{equation}\label{eq:TadmorEuler}
    \jump{\entVar}_{(i,j)}^T 
\numfluxb{f}_{(i,j)}^{\ec, \supEuler}
=
\jump{\Psi^{\supEuler}}_{(i,j)}
\end{equation}
holds.
Hence, by substituting Equation~\eqref{eq:TadmorEuler} into Equation~\eqref{eq:entropy_prod_ec2}, it becomes evident that in order to achieve an entropy-conserving FV discretization of the multi-ion GLM-MHD system, it is sufficient to identify the MHD numerical fluxes and non-conservative terms that satisfy
\begin{equation} \label{eq:TadmorReduced}
\jump{\entVar}_{(i,j)}^T 
\numfluxb{f}_{(i,j)}^{\ec, \supMHD}
+ \entVar^T_{j} {\Jan}^{\diamond, \ec}_{(j,i)}
- \entVar^T_{i} {\Jan}^{\diamond, \ec}_{(i,j)}
- \jump{\Psi^{\supMHD}}_{(i,j)} = 0.
\end{equation}

Contracting the local (non-conservative) part of the Godunov-Powell term yields
\begin{align}\label{eq:PowellphiContraction}
    \entVar^T \stateG{\phi}^{\Powell} =&
    \left( \sum_k \frac{2 r_k \rho_k \beta_k}{n_e e} \vec{v}_k \right) \cdot \vec{B}
    - \left( \sum_k 2 \beta_k \right) \vec{v}^+ \cdot \vec{B}
    + 2 \beta_+ \vec{v}^+ \cdot \vec{B}
    \nonumber\\
    =& 2 \left( \sum_k \beta_k \vec{v}^+_k \right) \cdot \vec{B}.
\end{align}
Moreover, contracting the term $\Jan^{\mathrm{Lor}}$ yields
\begin{align}\label{eq:PhiIonContraction}
    \entVar^T \Jan^{\mathrm{Lor}} =& 
    \left[
    \sum_k 2 \beta_k \left( v^+_{k,1} \left( \frac{1}{2} \|\vec{B}\|^2 + p_e \right) 
    - \left( \vec{v}^+_k \cdot \vec{B} \right) B_1 \right)
    \right]
    -2 \left( \sum_k \beta_k v^+_{k,1} \right) p_e
    \nonumber \\
    =&
    \left( \sum_k \beta_k v^+_{k,1} \right) \| \vec{B} \|^2
    - 2 \left( \sum_k \beta_k \vec{v}^+_k \right) \cdot \vec{B} B_1.
\end{align}

Since the multi-ion Godunov-Powell non-conservative term $\noncon^{\Powell}$ \eqref{eq:noncon_Powell} is very similar to the single-species MHD Godunov-Powell term, we adopt the central discretization proposed by \citet{Winters2016}, and used in \cite{Chandrashekar2016, Derigs2018,Bohm2018,Rueda-Ramirez2020,rueda2023entropy,rueda2024flux}, and extend it to the multi-ion framework.
The Godunov-Powell numerical non-conservative term then reads as
\begin{equation}
    \Jan^{\Powell\diamond, \ec}_{(i,j)} =
    \stateG{\phi}^{\Powell}_i \avg{B_1}_{(i,j)},
\end{equation}
which is consistent with the single-fluid GLM-MHD term.

We compute the contribution to the entropy production of the Godunov-Powell non-conservative term as 
\begin{align} \label{eq:PowellPhiContraction}
    \entVar_j^T \Jan^{\Powell\diamond, \ec}_{(j,i)}
    - \entVar_i^T \Jan^{\Powell\diamond, \ec}_{(i,j)}
    =& 
    2 \jump{\left( \sum_k \beta_k \vec{v}^+_k \right) \cdot \vec{B}}_{(i,j)} \avg{B_1}_{(i,j)}
    \nonumber\\
    =& 
    {
    2 \left( \sum_k \jump{\beta_k}_{(i,j)} \avg{\vec{v}^+_k  \cdot \vec{B}}_{(i,j)}\right) \avg{B_1}_{(i,j)}
    }
    \nonumber\\
    &+ {
    2 \left( \sum_k \avg{\beta_k}_{(i,j)} \jump{\vec{v}^+_k  \cdot \vec{B}}_{(i,j)}\right) \avg{B_1}_{(i,j)}
    }
\end{align}
where we used the identity
\begin{equation}\label{eq:identityjump}
    \jump{a\,b} = \jump{a}\avg{b} + \avg{a}\jump{b}.
\end{equation}
Finally,  since the GLM terms contracted with the entropy variables read
\begin{equation}
    \entVar^T \left({\color{red}\state{f}^{\mathrm{GLM}} + \Jan^{\mathrm{GLM}}}\right) = 
    c_h\left(
    -2\red{\psi} B_1\sum_k \beta_k + 
    2\beta_+B_1\red{\psi}
    +2\red{\psi}B_1\beta_+
    \right)
    -2\sum_k\beta_k v^+_1 {\color{red}\psi^2} + 2\beta_+v^+_1{\color{red}\psi^2} = \red{2c_h\beta_+\psi B_1},
\end{equation}
their contribution to the entropy flux potential is
\begin{equation}\label{eq:jumpEntropyFluxGLM}
    \jump{\entVar^T \left({\color{red}\state{f}^{\mathrm{GLM}} + \Jan^{\mathrm{GLM}}}\right)}
    = \red{2c_h\jump{\beta_+ \red{\psi} B_1}}.
\end{equation}

We remark that there are different ways to apply the identity \eqref{eq:identityjump}, which lead to slightly different terms.
In this work, we follow the procedure detailed by \citet{Derigs2018} and apply \eqref{eq:identityjump} to obtain terms that depend on the jumps of $\beta_k$, such that we can obtain discretizations consistent with Derigs et al.'s.
By isolating the jump of other quantities, we would arrive at different (also valid) discretizations.
For instance, if we applied the identity \eqref{eq:identityjump} to obtain terms that depend on the jumps of $\vec{B}$, we would obtain discretizations consistent with the single-fluid MHD fluxes of \citet{Chandrashekar2016}.

Note that $\noncon^{\rm{Lor}}$ \eqref{eq:noncon_Lor} is very similar to the single-fluid MHD magnetic flux.
In fact, it contains the exact same momentum components of the single-fluid MHD magnetic flux, but it has a local non-conservative part and new terms containing the electron pressure.
In the single-fluid MHD limit, the local non-conservative terms of the momentum equation become unity and the electron pressure is commonly assumed as negligible.
To keep the consistency with the single-fluid GLM-MHD discretization of \citet{Derigs2018}, we adopt the momentum components of the single-species EC flux (used as well in \cite{Winters2016,Chandrashekar2016,Winters2017}), and propose a consistent non-conservative FV discretization of the form
\begin{equation}
    \Jan^{\rm{Lor}\diamond, \ec}_{(i,j)} =
    \begin{pmatrix} 
0 \\[0.1cm]
\frac{r_k \rho_k}{n_e e} \\[0.1cm]
\frac{r_k \rho_k}{n_e e} \\[0.1cm]
\frac{r_k \rho_k}{n_e e} \\[0.1cm]
v^+_{k,1} \\[0.1cm]
\vec{0} \\[0.1cm]
0
\end{pmatrix}_i
\circ
\begin{pmatrix} 
0 \\[0.1cm]
\frac{1}{2} 
\avg{\|\vec{B}\|^2}
- \avg{B_1}^2 + \bar{p}_e\\[0.1cm]
- \avg{B_1} \avg{B_2} \\[0.1cm]
- \avg{B_1} \avg{B_3} \\[0.1cm]
\bar{p}_e \\[0.1cm]
\vec{0} \\[0.1cm]
0
\end{pmatrix}_{(i,j)},
\end{equation}
with a yet to be determined symmetric ``average'' of the electron pressure $\bar{p}_e$.

Contracting $\Jan^{\rm{Lor}\diamond, \ec}_{(i,j)}$ with the entropy variables reads
\begin{align}\label{eq:Phi2Contraction}
    \entVar_i^T \Jan^{\rm{Lor}\diamond, \ec}_{(i,j)} =&
    \left(\sum_k \beta_k v^+_{k,1}\right)_i \left( \avg{\|\vec{B}\|^2} + 2\bar{p}_e \right) 
    -2 \left( \sum_k \beta_k \vec{v}^+_k \right)_i \cdot \avg{\vec{B}}  \avg{B_1} 
    -2 \left( \sum_k \beta_k v^+_{k,1} \right)_i \bar{p}_e
    \nonumber \\
    =&
    \left( \sum_k \beta_k v^+_{k,1} \right)_i \avg{\| \vec{B} \|^2}
    - 2 \left( \sum_k \beta_k \vec{v}^+_k \right)_i \cdot  \avg{\vec{B}} \avg{B_1},
\end{align}
which shows that any consistent and symmetric ``average'' for the electron pressure can be used.
For instance, we can use the simple average $\bar{p}_e = \avg{p_e}$.

With the goal of being fully consistent with the single-species EC flux of \citet{Winters2016} and \citet{Derigs2018}, but also with the multi-ion GLM-MHD source term, we propose a numerical MHD flux of the form
\begin{equation}\label{eq:ec_flux_mhd}
    \numfluxb{f}^{\ec,\supMHD}_{(i,j)}
    + {\color{red}  \numfluxb{f}^{\ec, \supGLM}_{(i,j)}}
    =
    \begin{pmatrix}
    0 \\
    \vec{0} \\
    \numflux{f}^{\ec,\supMHD}_{E_k,(i,j)} + \red{2c_h \avg{\psi} \avg{B_1} - c_h\avg{\psi B_1}}\\
    \red{c_h \avg{\psi}} \\
    \avg{v^+_1}\avg{B_2} - \avg{v^+_2}\avg{B_1}\\
    \avg{v^+_1}\avg{B_3} - \avg{v^+_3}\avg{B_1}\\
    \red{c_h\avg{B_1}}
    \end{pmatrix},
\end{equation}
and will solve for the unknown term $\numflux{f}^{\ec,\supMHD}_{E_k,(i,j)}$.

The contraction of the term \eqref{eq:ec_flux_mhd} with the jump of entropy variables reads
\begin{align}\label{eq:entvar_fMHD}
    \jump{\entVar^T}_{(i,j)} \left(
    \numfluxb{f}^{\ec,\supMHD}_{(i,j)}
    + {\color{red}  \numfluxb{f}^{\ec, \supGLM}_{(i,j)}}
    \right)
    =&
    -2 \sum_k \jump{\beta_k} \left(
    \numflux{f}^{\ec,\supMHD}_{E_k,(i,j)}
    + \red{2c_h\avg{\psi}\avg{B_1} - c_h\avg{\psi B_1}}
    \right)\nonumber\\
    &+ 2 \jump{\beta^+ B_2} \left( \avg{v^+_1}\avg{B_2} - \avg{v^+_2}\avg{B_1}\right)
    \nonumber\\
    &+ 2 \jump{\beta^+ B_3} \left(
    \avg{v^+_1}\avg{B_3} - \avg{v^+_3}\avg{B_1}
    \right)
    \nonumber\\
    & + \red{2c_h\jump{\beta_+B_1}\avg{\psi} + 2c_h\jump{\beta_+\psi}\avg{B_1}}\nonumber\\
    =&
    -2 \sum_k \jump{\beta_k} \numflux{f}^{\ec,\supMHD}_{E_k,(i,j)} 
    + 2 \jump{\beta^+ \vec{B}}
    \cdot 
    \left( \avg{\vec{B}}     \avg{v^+_1}
    - \avg{\vec{v}^+} \avg{B_1}
    \right)\nonumber\\
    &+ \red{2c_h\jump{\beta_+\psi B_1}},
\end{align}
where the last GLM term is obtained using the identity~\eqref{eq:identityjump} as follows:
\begin{align*}
    &\red{-2\sum_k\jump{\beta_k}\left(
    2c_h\avg{\psi}\avg{B_1} - c_h\avg{\psi B_1}
    \right) 
    +2c_h\jump{\beta_+B_1}\avg{\psi} + 2c_h\jump{\beta_+\psi}\avg{B_1}}\\
    & = \red{-2\sum_k\jump{\beta_k}\left(
2c_h\avg{\psi}\avg{B_1} - c_h\avg{\psi B_1}
    \right) +
    2c_h
    \left(
    \sum_k\jump{\beta_k B_1}\avg{\psi} + \jump{\beta_k \psi}\avg{B_1}
    \right)
    }\\
    & = \red{-2c_h\sum_k\jump{\beta_k}\left(
2\avg{\psi}\avg{B_1} - \avg{\psi B_1}
    \right)} \\
    & \quad \quad\red{+
    2c_h
    \left(
    \sum_k\left(\jump{\beta_k}\avg{B_1} + \avg{\beta_k}\jump{B_1}\right)\avg{\psi} + \left(\jump{\beta_k}\avg{ \psi} + \avg{\beta_k}\jump{ \psi}
    \right)\avg{B_1}
    \right)
    }\\
    & = \red{-2c_h\sum_k\jump{\beta_k}\left(
2\avg{\psi}\avg{B_1} - \avg{\psi B_1}
    \right)
    +
    2c_h
    \left(
    \sum_k2\jump{\beta_k}\avg{B_1}\avg{\psi} 
    +
    \avg{\beta_k}\jump{\psi B_1}
    \right)
    }\\
    & = \red{2c_h\sum_k \jump{\beta_k \psi B_1} = 2c_h \jump{\beta_+ \psi B_1}}.
\end{align*}

Finally, taking into account that the non-conservative term $\noncon^{\mathrm{multi}}$ \eqref{eq:noncon_multi} contains terms with the same structure as the induction equation terms of \eqref{eq:ec_flux_mhd}, we assume a numerical non-conservative term of the form
\begin{align}\label{eq:phi3_num}
    \Jan^{\rm{multi}\diamond, \ec}_{(i,j)}
    =&
    (B_2)_i
    \begin{pmatrix}
    0 \\ \vec{0} \\ \avg{v^-_{k,1}} \avg{B_2} - \avg{v^-_{k,2}} \avg{B_1}  \\ \vec{0}\\0
    \end{pmatrix}
    + (B_3)_i
    \begin{pmatrix}
    0 \\ \vec{0} \\ \avg{v^-_{k,1}} \avg{B_3} - \avg{v^-_{k,3}} \avg{B_1}  \\ \vec{0}\\0
    \end{pmatrix}
    \nonumber\\
    =&
    \begin{pmatrix}
    0 \\ \vec{0} \\ 
    \vec{B}_i \cdot \avg{\vec{B}}
    \avg{v^-_{k,1}} -
    \vec{B}_i \cdot
    \avg{\vec{v}^-_{k}} \avg{B_1}  \\ \vec{0}\\0
    \end{pmatrix}
    =
    \begin{pmatrix}
    0 \\ \vec{0} \\ 
    \vec{B}_i \cdot \vec{h}^{\rm{multi}*}_{E_k,(i,j)}  \\ \vec{0} \\0
    \end{pmatrix},
\end{align}
where we have introduced a new auxiliary quantity $\vec{h}^{\rm{multi}*}_{E_k,(i,j)}$.

Contracting $\Jan^{\rm{multi}\diamond, \ec}$ with the entropy variables reads%
\begin{align}
    \entVar_i^T \Jan^{\rm{multi}\diamond, \ec}_{(i,j)} 
    =& -2 
    \left(
    \sum_k 
    \left(\beta_k \vec{B} \right)_i \avg{v^-_{k,1}} \right) \cdot \avg{\vec{B}}
    + 2 
    \left(
    \sum_k 
    \left( \beta_k \vec{B} \right)_i \cdot \avg{\vec{v}^-_{k}} \right) \avg{B_1},
\end{align}
and therefore we have
\begin{align}\label{eq:ent_phimulti_d}
    \entVar_j^T \Jan^{\rm{multi}\diamond, \ec}_{(j,i)}  - \entVar_i^T \Jan^{\rm{multi}\diamond, \ec}_{(i,j)} 
    =& -2 
    \left(
    \sum_k 
    \jump{\beta_k \vec{B}} \avg{v^-_{k,1}} \right) \cdot \avg{\vec{B}}
    + 2 
    \left(
    \sum_k 
    \jump{\beta_k \vec{B}} \cdot \avg{\vec{v}^-_{k}} \right) \avg{B_1}.
\end{align}

The next term contributing to the jump of the entropy potential is $\Jan^{\rm{multi}}$.
Its contraction with the entropy variables reads as
\begin{align}\label{eq:ent_phimulti}
    \entVar^T \Jan^{\rm{multi}} =& -2 \sum_k \beta_k
    \left(
    v^-_{k,1} B_2^2
    - v^-_{k,2} B_1 B_2
    + v^-_{k,1} B_3^2
    - v^-_{k,3} B_1 B_3
    \right)
    \nonumber\\
    =& {
    -2 
    \left(
    \sum_k \beta_k     v^-_{k,1} \right) 
    \| \vec{B} \|^2
    + 2 
    \left(
    \sum_k \beta_k     \vec{v}^-_{k} \right) \cdot \vec{B} B_1
    }.
\end{align}

The contribution of the GLM terms to the EC-fluxes vanishes:
\begin{align}
     \entVar^T_i \red{\Jan^{\rm{GLM}\diamond, \ec}_{(i,j)}} = \red{
     \left(
     -\sum_k 2\beta_k v_1^+\red{\psi} + 2\beta_+\psi v_1^+
     \right)_i
     \red{\avg{\psi}}_{(i,j)}} = 0,
\end{align}
and combining the GLM-terms in equations~\eqref{eq:entvar_fMHD} and~\eqref{eq:jumpEntropyFluxGLM} yields a zero-sum of all GLM terms.

It remains only to compute the term $\numflux{f}^{\ec,\supMHD}_{E_k,(i,j)}$.
Following the procedure of \citet{Derigs2018}, we apply the identity \eqref{eq:identityjump} to all terms that contribute to \eqref{eq:TadmorReduced}, in order to obtain terms that depend on the jump of ${\beta_k}$.
Now, we write only the terms that multiply $\jump{\beta_k}$ in \eqref{eq:TadmorReduced}:
\begin{itemize}
    \item From $\jump{\entVar}^T \numfluxb{f}^{\ec,\supMHD}_{(i,j)}$ (see \eqref{eq:entvar_fMHD}): 
    $-2 \numflux{f}^{\ec,\supMHD}_{E_k,(i,j)} + 2 \avg{\vec{B}} \cdot \vec{\numflux{f}}^{\ec,\supMHD}_{\vec{B},(i,j)}$
    \item From $\jump{\entVar^T \left( \Jan^{\Powell} + \Jan^{\mathrm{Lor}}\right)}$ (see \eqref{eq:PowellphiContraction} and \eqref{eq:PhiIonContraction}):
    $\avg{v^+_{k,1}  \| \vec{B} \|^2}$
    \item From $\entVar_j^T \Jan^{\Powell\diamond, \ec}_{(j,i)}
    - \entVar_i^T \Jan^{\Powell\diamond, \ec}_{(i,j)}$ (see \eqref{eq:PowellPhiContraction}): 
    $2 \avg{\vec{v}^+_k  \cdot \vec{B}}_{(i,j)} \avg{B_1}_{(i,j)}$
    \item From $\entVar_j^T \Jan^{\rm{Lor}\diamond, \ec}_{(j,i)}
    - \entVar_i^T \Jan^{\rm{Lor}\diamond, \ec}_{(i,j)}$ (see \eqref{eq:Phi2Contraction}):
    $\avg{v^+_{k,1} } \avg{\| \vec{B} \|^2}
    - 2 \avg{\vec{v}^+_k } \cdot  \avg{\vec{B}} \avg{B_1}$
    \item From $\jump{\entVar^T  \Jan^{\mathrm{multi}} }$ (see \eqref{eq:ent_phimulti}):
    $-2 
    \avg{v^-_{k,1}  \| \vec{B} \|^2}
    + 2 
    \avg{ \vec{v}^-_{k} \cdot \vec{B} B_1}$
    \item From $\jump{\entVar^T  \state{f}^{\supMHD}}$ (see \eqref{eq:advective_fluxes_1D}):
    $-2 
    \avg{v^+_{k,1} \|\vec{B}\|^2} 
    +2 \avg{B_1 \left(\vec{v}^+_k\cdot\vec{B}\right)}
    + 2 \avg{v^+_1 \| \vec{B} \|^2}
    - 2 \avg{B_1 \vec{v}^+ \cdot \vec{B}}$
    \item From $\entVar_j^T \Jan^{\rm{multi}\diamond, \ec}_{(j,i)}
    - \entVar_i^T \Jan^{\rm{multi}\diamond, \ec}_{(i,j)}$ (see \eqref{eq:ent_phimulti_d}):
    $-2 \avg{\vec{B}} \cdot \vec{h}^{\rm{multi}*}_{E_k,(i,j)}$
\end{itemize}

We obtain the entropy-conserving flux for the quantity $E_k$ by combining all terms into \eqref{eq:TadmorReduced}:
\begin{align} \label{eq:ecflux_Ek}
    \numflux{f}^{\ec,\supMHD}_{E_k,(i,j)} =&
    \avg{\vec{B}} \cdot \vec{\numflux{f}}^{\ec,\supMHD}_{\vec{B},(i,j)}
    - \frac{1}{2} 
    \avg{v^+_{k,1}  \| \vec{B} \|^2}
    +
    \avg{\vec{v}^+_k  \cdot \vec{B}}_{(i,j)} \avg{B_1}_{(i,j)}
    + \frac{1}{2} 
    \avg{v^+_{k,1} } \avg{\| \vec{B} \|^2}
    \nonumber\\
    &- 
    \avg{\vec{v}^+_k } \cdot  \avg{\vec{B}} \avg{B_1}
    {\color{green}+ 
    \avg{v^-_{k,1}  \| \vec{B} \|^2}
    - 
    \avg{ \vec{v}^-_{k} \cdot \vec{B} B_1}}
    \nonumber\\
    &{\color{green}+
    \avg{v^+_{k,1} \|\vec{B}\|^2} 
    - \avg{B_1 \left(\vec{v}^+_k\cdot\vec{B}\right)}
    - 
    \avg{v^+_1 \| \vec{B} \|^2}
    +
    \avg{B_1 \vec{v}^+ \cdot \vec{B}}}
    -
    \avg{\vec{B}} \cdot \vec{h}^{\rm{multi}*}_{E_k,(i,j)}
    \nonumber\\
    =&
    \avg{\vec{B}} \cdot \vec{\numflux{f}}^{\ec,\supMHD}_{\vec{B},(i,j)}
    - \frac{1}{2} 
    \avg{v^+_{k,1}  \| \vec{B} \|^2}
    +
    \avg{\vec{v}^+_k  \cdot \vec{B}}_{(i,j)} \avg{B_1}_{(i,j)}
    + \frac{1}{2} 
    \avg{v^+_{k,1} } \avg{\| \vec{B} \|^2}
    \nonumber\\
    &- 
    \avg{\vec{v}^+_k } \cdot  \avg{\vec{B}} \avg{B_1}
    -
    \avg{\vec{B}} \cdot \vec{h}^{\rm{multi}*}_{E_k,(i,j)},
\end{align}
where the terms in green cancel each other out.

It is easy to see that, in the single-species case ($N_i=1$), the entropy-conserving flux \eqref{eq:ecflux_Ek} for the only ion species of the multi-ion GLM-MHD system is consistent with the EC flux for the energy term of \citet{Derigs2018},
\begin{align}
    \numflux{f}^{\ec,\mathrm{MHD}}_{E,(i,j)} + \red{\numflux{f}^{\ec,\mathrm{GLM}}_{E,(i,j)}}
    =&
    \avg{\vec{B}} \cdot \vec{\numflux{f}}^{\ec,\mathrm{MHD}}_{\vec{B},(i,j)}
    - \frac{1}{2} 
    \avg{v_{1}  \| \vec{B} \|^2}
    +
    \avg{\vec{v}  \cdot \vec{B}}_{(i,j)} \avg{B_1}_{(i,j)}
    + \frac{1}{2} 
    \avg{v_{1} } \avg{\| \vec{B} \|^2}
    \nonumber\\
    &- 
    \avg{\vec{v} } \cdot  \avg{\vec{B}} \avg{B_1}
    + \red{2 c_h \avg{\psi}\avg{B_1}
    - c_h \avg{\psi B_1}},
\end{align}
because in the single species case we have $\vec{v} = \vec{v}_k = \vec{v}^+_k$ and $\vec{v}^-_k = 0$, and hence 
$\vec{\numflux{f}}^{\ec,\supMHD}_{\vec{B},(i,j)} = \vec{\numflux{f}}^{\ec,\mathrm{MHD}}_{\vec{B},(i,j)}$
and
$\vec{h}^{\rm{multi}*}_{E_k,(i,j)} = 0$.

We can also check for consistency with the 
multi-ion GLM-MHD energy flux term \eqref{eq:advective_fluxes_1D}.
To do that, we evaluate the EC flux for the quantity $E_k$ \eqref{eq:ecflux_Ek} with the same left and right states:
\begin{align*}
    \numflux{f}^{\ec,\supMHD}_{E_k,(i,i)} +  \red{\numflux{f}^{\ec,\rm{GLM}}_{E_k,(i,i)}}
    =& \vec{B}_i \cdot \vec{f}^{\supMHD}_{\vec{B},i}
    - \vec{B}_i \cdot \vec{h}^{\rm{multi}*}_{E_k,(i,i)}
    + \red{c_h \left(\psi B_1\right)_i}
    \\
    =&
    \left[ v^+_{1} \|\vec{B}\|^2 - B_1 \left(\vec{v}^+\cdot\vec{B}\right)  \right]_i
    -
    \left[ v^-_{k,1} \|\vec{B}\|^2 - B_1 \left(\vec{v}^-_k\cdot\vec{B}\right)  \right]_i
    + \red{c_h \left(\psi B_1\right)_i}
    \\
    =& \left[ 
    v^+_{k,1} \|\vec{B}\|^2 - B_1 \left(\vec{v}^+_k\cdot\vec{B}\right)
    \right]_i
    + \red{c_h \left(\psi B_1\right)_i}
    \\
    =&f^{\supMHD}_{E_k} + \red{f^{\rm{GLM}}_{E_k}}.
\end{align*}

\change{

\section{Cartesian 2D Discretization} \label{app:2d}
}
\q{c3r1}{

\change{
We obtain the two-dimensional DGSEM discretization of the multi-ion GLM-MHD system on Cartesian meshes with uniform grid size, $\Delta x = \Delta y$, by extending the one-dimensional variant using tensor-product basis expansions.
For brevity, we use the same polynomial degree in all three spatial directions, although different polynomial degrees can be used for different directions, as in \cite{RuedaRamirez2019,RuedaRamirez2019a}. The two-dimensional version of \eqref{eq:DGSEM} reads
\begin{align} \label{eq:DGSEM_2D}
J \omega_i \omega_j \dot{\state{u}}_{ij} 
+ &
\omega_j \left( 
\sum_{k=0}^N S_{ik} \left( \state{f}^{*1}_{(i,k)j} + \Jan^{\star 1}_{(i,k)j} \right)
- 
\delta_{i0} \left( \numfluxb{f}^1_{(0,L)j} + \Jan^{\diamond 1}_{(0,L)j} \right)
+ \delta_{iN} \left( \numfluxb{f}^1_{(N,R)j} + \Jan^{\diamond 1}_{(N,R)j} \right)
\right)
\nonumber \\
+ &
\omega_i 
\underbrace{
\Bigg(
\sum_{k=0}^N S_{jk} \left( \state{f}^{*}_{i(j,k)} + \Jan^{\star 2}_{i(j,k)} \right)
}_{\mathrm{Volume \, term}}
- 
\underbrace{
\delta_{j0} \left( \numfluxb{f}^2_{i(0,L)} + \Jan^{\diamond 2}_{i(0,L)} \right)
+ \delta_{jN} \left( \numfluxb{f}^2_{i(N,R)} + \Jan^{\diamond 2}_{i(N,R)} \right)
\Bigg)
}_{\mathrm{Surface \,  term}}
= \state{0},
\end{align}
where the 1D mapping Jacobian determinant $J = \Delta x / 2$ is again constant inside each element, and the main difference with the 1D discretization is that we now have fluxes and non-conservative terms in the $x$ and $y$ directions, denoted with the super-script $1$ and $2$, respectively.

The numerical fluxes and non-conservative terms in the $x$ direction are exactly the same ones derived for the 1D discretization.
For instance, the entropy-conserving fluxes and non-conservative terms are presented in Section \ref{sec:fvec}.

Following the same procedure detailed in Appendix \ref{app:ec_derivation}, we can derive numerical entropy-conserving fluxes and non-conservative terms that are consistent with both the multi-ion GLM-MHD terms \eqref{eq:multi-ion_mod_div} and the single-fluid GLM-MHD discretization of \citet{Derigs2018}.
The EC flux in the $y$ direction reads
\begin{equation}\label{eq:ECFlux_y}
\state{f}^{*\ec, 2}(\state{u}_L,\state{u}_R) =
\begin{pmatrix} 
\rho^{\ln}_{k} \avg{v_{k,2}} \\
\rho^{\ln}_{k} \avg{v_{k,2}}\avg{v_{k,2}}\\ 
\rho^{\ln}_{k} \avg{v_{k,2}} \avg{v_{k,2}} + \overline{p}_k  \\
\rho^{\ln}_{k} \avg{v_{k,2}} \avg{v_{k,3}}  \\
f^{\ec, \supEuler}_{E_k} + \blue{\numflux{f}^{\ec,\supMHD, 2}_{E_k}} + \red{2 c_h \avg{\psi} \avg{B_2} - c_h \avg{\psi B_2}}\\
    \blue{\avg{v^+_2}\avg{B_1} - \avg{v^+_1}\avg{B_2}}\\
    \red{c_h \avg{\psi}} \\
    \blue{\avg{v^+_2}\avg{B_3} - \avg{v^+_3}\avg{B_2}}\\
    \red{c_h \avg{B_2} }
\end{pmatrix}
\end{equation}
where the mean pressure is defined in \eqref{eq:pbar}, the Euler energy term $f^{\ec, \supEuler}_{E_k}$ is the same as in the $x$ direction \eqref{eq:EulerEnergyEC}, and MHD energy terms reads as
\begin{align*}
\blue{\numflux{f}^{\ec,\supMHD, 2}_{E_k,i(j,k)} =}&
    \blue{\avg{\vec{B}} \cdot \vec{\numflux{f}}^{\ec,\supMHD, 2}_{\vec{B},i(j,k)}
    - \frac{1}{2} 
    \avg{v^+_{k,2}  \| \vec{B} \|^2}
    +
    \avg{\vec{v}^+_k  \cdot \vec{B}}_{i(j,k)} \avg{B_2}_{i(j,k)}
    + \frac{1}{2} 
    \avg{v^+_{k,2} } \avg{\| \vec{B} \|^2}}
    \\
    & 
    \blue{-\avg{\vec{v}^+_k } \cdot  \avg{\vec{B}} \avg{B_2}
    -
    \avg{\vec{B}} \cdot \underbrace{\left( \avg{\vec{B}}
    \avg{v^-_{k,2}} -
    \avg{\vec{v}^-_{k}} \avg{B_2} \right)}_{\vec{h}^{\rm{multi}*2}_{E_k,i(j,k)}}},
\end{align*}
where $\vec{\numflux{f}}^{\ec,\supMHD, 2}_{\vec{B},i(j,k)}$ corresponds to the magnetic field components of the flux given in \eqref{eq:ECFlux_y}, i.e.,
\begin{equation}
\vec{\numflux{f}}^{\ec,\supMHD, 2}_{\vec{B},i(j,k)}
= 
\begin{pmatrix}
    \blue{\avg{v^+_2}\avg{B_1} - \avg{v^+_1}\avg{B_2}}, &
    \red{c_h \avg{\psi}}, &
    \blue{\avg{v^+_2}\avg{B_3} - \avg{v^+_3}\avg{B_2}}
\end{pmatrix}^T.
\end{equation}

As before, the black terms come from the Euler discretization, the blue terms come from the MHD discretization, and the red terms come from the GLM discretization.
The four non-conservative two-point terms that provide entropy conservation are the Godunov-Powell term,
\begin{equation}
\label{eq:fv_PhiGP_y}
    {
    \Jan^{\Powell\diamond, \ec, 2}_{i(j,k)} =
    \stateG{\phi}^{\Powell}_{ij} \avg{B_2}_{i(j,k)}},
\end{equation}
the Lorentz non-conservative term
\begin{equation}
\label{eq:fv_PhiLor_y}
{
    \Jan^{\rm{Lor}\diamond, \ec, 2}_{i(j,k)} =
    \begin{pmatrix} 
0 \\[0.1cm]
\frac{r_k \rho_k}{n_e e} \\[0.1cm]
\frac{r_k \rho_k}{n_e e} \\[0.1cm]
\frac{r_k \rho_k}{n_e e} \\[0.1cm]
v^+_{k,2} \\[0.1cm]
\vec{0} \\[0.1cm]
0
\end{pmatrix}_{ij}
\circ
\underbrace{
\begin{pmatrix} 
0 \\[0.1cm]
- \avg{B_2} \avg{B_1} \\[0.1cm]
\frac{1}{2} 
\avg{\|\vec{B}\|^2}
- \avg{B_2}^2 + \bar{p}_e\\[0.1cm]
- \avg{B_2} \avg{B_3} \\[0.1cm]
\bar{p}_e \\[0.1cm]
\vec{0} \\[0.1cm]
0 
\end{pmatrix}_{i(j,k)}
}_{\state{h}^{\rm{Lor}\diamond, \ec}_{i(j,k)}}
},
\end{equation}
with the arbitrary symmetric ``average'' of the electron pressure $\bar{p}_e$,
the ``multi-ion'' term,
\begin{align}
    \Jan^{\rm{multi}\diamond, \ec, 2}_{i(j,k)}
    =&
    (B_1)_{ij}
    \begin{pmatrix}
    0 \\ \vec{0} \\ \avg{v^-_{k,2}} \avg{B_1} - \avg{v^-_{k,1}} \avg{B_2}  \\ \vec{0} \\ 0
    \end{pmatrix}
    + (B_3)_{ij}
    \begin{pmatrix}
    0 \\ \vec{0} \\ \avg{v^-_{k,2}} \avg{B_3} - \avg{v^-_{k,3}} \avg{B_2}  \\ \vec{0}\\ 0
    \end{pmatrix}
    \nonumber\\
    =&
    \begin{pmatrix}
    0 \\ \vec{0} \\ 
    \vec{B}_{ij} \cdot \avg{\vec{B}}
    \avg{v^-_{k,2}} -
    \vec{B}_{ij} \cdot
    \avg{\vec{v}^-_{k}} \avg{B_2}  \\ \vec{0}\\ 0
    \end{pmatrix}
    =
        \begin{pmatrix}
        0 \\ \vec{0} \\ 
        \vec{B}_{ij} \cdot \vec{h}^{\rm{multi}*2}_{E_k,i(j,k)}  \\ \vec{0}\\ 0
        \end{pmatrix},
\end{align}
and the GLM term,
\begin{align} \label{eq:ECGLM_y}
\red{
        \Jan^{\rm{GLM}\diamond, \ec, 2}_{i(j,k)} =
    \begin{pmatrix} 
0 \\
\vec{0} \\
v^+_{2} \psi \\
\vec{0} \\
v^+_{2}
\end{pmatrix}_{ij}
\avg{\psi}_{i(j,k)}.
    }
\end{align}

}
}


%
%

\end{document}